\documentclass[12pt]{amsart}
\usepackage[pdftex]{graphicx,color}
\usepackage{amssymb,latexsym,bm,color}
\usepackage[dvips,centering,includehead,width=15.1cm,height=22.7cm]{geometry}
\usepackage{amsmath, amsthm}
\usepackage{epsfig}
\usepackage{amssymb}
\input{xy} \xyoption{all}
\usepackage{color}
\usepackage{enumerate}

\def\lbr{[[}
\def\rbr{]]}

\definecolor{darkgreen}{rgb}{0.0, 0.7, 0.0}

\numberwithin{equation}{section}

\newtheorem{theorem}{Theorem}[section]

\newtheorem{lemma}[theorem]{Lemma}

\newtheorem{corollary}[theorem]{Corollary}
\newtheorem{proposition}[theorem]{Proposition}

\newtheorem{conjecture}[theorem]{Conjecture}

\theoremstyle{definition}
\newtheorem{definition}[theorem]{Definition}
\newtheorem{example}[theorem]{Example}
\newtheorem{remark}[theorem]{Remark}

\newtheorem{notation}[theorem]{Notation}

\newcommand{\thmref}[1]{Theorem~\ref{#1}}

\newcommand{\propref}[1]{Proposition~\ref{#1}}
\newcommand{\corref}[1]{Corollary~\ref{#1}}

\newcommand{\remref}[1]{Remark~\ref{#1}}
\newcommand{\conjref}[1]{Conjecture~\ref{#1}}

\newcommand{\res}{\mathop{\text{\rm res}}}
\newcommand{\eps}{{\varepsilon}}
\newcommand{\M}{{\mathcal M}}

\newcommand{\C}{{\mathcal C}}

\newcommand{\D}{{\mathcal D}}

\renewcommand{\O}{{\mathcal O}}
\newcommand{\tN}{{\widetilde{N}}}

\newcommand{\wt}{{\mathrm{wt}}}
\newcommand{\len}{{\mathrm{len}}}
\newcommand{\comb}{{\mathrm{comb}}}
\newcommand {\Hilb}{{\mathrm{Hilb}}}
\newcommand {\trop}{{\mathrm{trop}}}
\DeclareMathOperator {\FD}{{\bf FD}}
\DeclareMathOperator {\rk}{rk}
\DeclareMathOperator {\Area}{Area}
\DeclareMathOperator {\mult}{mult}
\DeclareMathOperator {\F}{\Sigma}
\DeclareMathOperator {\conv}{conv}
\DeclareMathOperator {\dive}{div}
\newcommand{\GGamma}{ {\bf \Gamma} }
\newcommand{\kk}{ {\bf k} }

\renewcommand{\aa}{ {\bf a} }

\newcommand{\maxdelta}{ 10 }

\newcommand{\RR}{{\mathbb R}}
\newcommand{\ZZ}{{\mathbb Z}}

\newcommand{\QQ}{{\mathbb Q}}
\newcommand{\PP}{{\mathbb P}}
\newcommand{\CC}{{\mathbb C}}

\newcommand{\oo}{\multiput(0,0)(10,0){2}{\circle{2}}}
\newcommand{\ooo}{\multiput(0,0)(10,0){3}{\circle{2}}}
\newcommand{\oooo}{\multiput(0,0)(10,0){4}{\circle{2}}}
\newcommand{\ooooo}{\multiput(0,0)(10,0){5}{\circle{2}}}
\newcommand{\oooooo}{\multiput(0,0)(10,0){6}{\circle{2}}}

\newcommand{\Eeee}{\put(1,0){\line(1,0){8}}}
\newcommand{\eEee}{\put(11,0){\line(1,0){8}}}
\newcommand{\eeEe}{\put(21,0){\line(1,0){8}}}
\newcommand{\eeeE}{\put(31,0){\line(1,0){8}}}
\newcommand{\eeeeE}{\put(41,0){\line(1,0){8}}}

\newcommand{\eeeBB}{\qbezier(30.2,1)(34,5)(40,5)\qbezier(40,5)(46,5)(49.8,1)}

\newcommand{\eeOe}{\qbezier(20.8,0.6)(25,4)(29.2,0.6)\qbezier(20.8,-0.6)(25,-4)(29.2,-0.6)}

\newcommand{\eeeeO}{\qbezier(40.8,0.6)(45,4)(49.2,0.6)\qbezier(40.8,-0.6)(45,-4)(49.2,-0.6)}

\title{Refined curve counting with tropical geometry}
\author{Florian Block \and Lothar G\"ottsche}
\address{Florian Block, Department of Mathematics, University of
  California, Berkeley, Berkeley, USA}
\email{block@math.berkeley.edu}
\address{Lothar G\"ottsche, International Centre for Theoretical Physics, Strada Costiera 11, 34151 Trieste, Italy}
\email{gottsche@ictp.it}

\thanks{The first author was supported by the EPSRC grant EP/I008071/1
and a Feodor Lynen-Fellowship of the Alexander von Humboldt-Foundation.}

\begin{document}

\begin{abstract}
The Severi degree is the degree of the Severi variety
parametrizing
plane curves of degree $d$ with $\delta$ nodes.
Recently, G\"ottsche and Shende gave two
refinements of Severi degrees, polynomials in a variable $y$, which
are conjecturally equal, for large $d$.
At $y = 1$,
one of the refinements, the
relative Severi degree,
specializes to the (non-relative) Severi degree. 

We give a tropical description of the refined
Severi degrees, in terms of a refined tropical curve count for all
toric surfaces. We also refine the equivalent count of floor diagrams
for Hirzebruch and rational ruled surfaces.
 Our description implies that,
for fixed $\delta$, the
refined Severi degrees are polynomials
in $d$ and $y$, for large $d$. As a consequence, we show that,
for  $\delta \le \maxdelta$ and all $d$, both refinements of
G\"ottsche and Shende agree and equal our refined counts of tropical
curves and floor diagrams.
\end{abstract}

\maketitle

\section{Introduction}
\label{sec:intro}

A \emph{$\delta$-nodal curve} is a reduced (not necessarily
irreducible) curve with $\delta$ simple nodes and no other singularities.
The \emph{Severi degree} $N^{d, \delta}$ is the degree of the
Severi variety parametrizing plane $\delta$-nodal curves of degree $d$.
Equivalently, $N^{d, \delta}$ is the number of $\delta$-nodal 
plane curves of degree $d$ through $\tfrac{(d+3)d}{2} - \delta$ generic points in
the complex projective plane $\PP^2$.

Severi degrees are generally difficult to compute. Their
study goes back to the midst of 19th century, when Steiner~\cite{St48},
in 1848,
showed that the degree $N^{d, 1}$ of the discriminant of
$\PP^2$ is $3(d-1)^2$. Only in 1998, Caporaso and
Harris~\cite{CH98} computed $N^{d, 
  \delta}$ for any $d$ and $\delta$, by their celebrated recursion
(involving \emph{relative Severi degrees} $N^{d, \delta}(\alpha,
\beta)$ counting curves satisfying tangency conditions to a fixed line).

Di Francesco and Itzykson~\cite{DI}, in 1994, conjectured the numbers
$N^{d, \delta}$ to be polynomial in $d$, for fixed $\delta$ and
$d$ large enough. In 2009, Fomin-Mikhalkin~\cite{FM} showed that,
for each $\delta \ge 1$, there is a polynomial $N_\delta(d)$ in $d$
with $N^{d, \delta} = N_\delta(d)$, provided that $d \ge 2
\delta$. The polynomials $N_{\delta}(d)$ are called \emph{node polynomials}.

More generally, for $S$ a projective algebraic surface, and $L$ a line bundle on $S$, the Severi degree
$N^{(S,L),\delta}$ is the number of $\delta$-nodal curves in the complete linear system $|L|$ through 
$\dim|L|-\delta$ general points of $S$ .
In  \cite{Go} it was conjectured that the  Severi degrees of arbitrary smooth
projective surfaces $S$ with a sufficiently ample line bundle
$L$ are given by universal polynomials. Specifically the conjecture predicts  for each
fixed $\delta$, the existence of a polynomial  $\widetilde N^{(S,L),\delta}$ in the intersection numbers
$L^2$, $LK_S$, $K_S^2$, $c_2(S)$ such that 
$N^{(S,L),\delta}=\widetilde N^{(S,L),\delta}$ for $L$ sufficiently ample. We call the $\widetilde N^{(S,L),\delta}$
the \emph{curve counting invariants}.
In addition the $\widetilde N^{(S,L),\delta}$ were conjectured to be given by a multiplicative generating function, i.e. 
there are universal power series $A_1,A_2,A_3,A_4\in \QQ[[q]]$, such that 
\begin{equation}
\label{multi}
\sum_{\delta\ge 0} \widetilde N^{(S,L),\delta}q^{\delta}=A_1^{L^2} A_1^{LK_S}A_3^{K_S^2} A_4^{c_2(S)}.
\end{equation}
Furthermore $A_1$ and $A_2$ are given explicitly in terms of modular forms. 
This  conjecture was proved by
Tzeng~\cite{Tz10} in 2010. A second proof was given shortly afterwards
by Kool, Shende, and Thomas~\cite{KST11}. In the latter proof, the
authors identified the numbers $\widetilde N^{(S,L),\delta}$  as coefficients of the
generating function of the topological Euler characteristics of relative Hilbert
schemes (see Section~\ref{sec:refined}). This is motivated by the proposed definition  the Gopakumar Vafa (BPS) 
invariants in terms of Pandharipande-Thomas invariants in \cite{PT10}. Thus the curve counting invariants can be viewed as
special cases of BPS invariants.
By definition for $S=\PP^2$ and $L=\O(d)$, the curve counting invariants coincide with the node polynomials: 
$\widetilde N^{(\PP^2,\O(d)),\delta}=N_\delta(d)$.

Inspired by this description, in \cite{GS12} \emph{refined invariants}   $\widetilde N^{(S,L),\delta}(y)$ are defined as coefficients of a very similar generating function, but with the topological Euler characteristic replaced by the normalized $\chi_{-y}$-genus, a
specialization of the Hodge polynomial.  They are Laurent polynomials in $y$, symmetric under $y\mapsto \frac{1}{y}$.
In \cite{GS12} a number of conjectures are made about the refined invariants $\widetilde N^{(S,L),\delta}(y)$. 
In particular they are conjectured to have a multiplicative generating function (as in \eqref{multi}), where now two of the universal power series are explicitly given in terms of Jacobi forms.
This fact was proven in the meantime in \cite{GS13} in case the canonical divisor $K_S$ is numerically trivial.

In this paper we will concentrate on the case that $S$ is a toric surface, and sometimes we restrict to the case that $S=\PP^2$, $L=\O(d)$, and denote $\widetilde N^{(\PP^2,\O(d)),\delta}(y)=\widetilde N^{d,\delta}(y)$.
In the case that $S$ is a toric surface and $L$ a toric line bundle, we will change slightly the definition of the Severi degrees. 
We denote  $N^{(S,L),\delta}$ the number of 
cogenus $\delta$ curves in $|L|$ passing though
$\dim|L|-\delta$ general points in $S$, which do not contain a 
 toric boundary divisor as a component. This is done because, as we will see  below, with this new definition (and not with the old one) the Severi degrees can be computed via tropical geometry and by a Caporaso-Harris type recursion formula. 
The Severi degrees as defined before we denote by 
 $N_*^{(S,L),\delta}$, but we will not consider them in the sequel.
 
  If $L$ is $\delta$-very ample (see below for the definition) it is
  easy to see (\remref{rem:delta*}) that  $N^{(S,L),\delta}=N^{(S,L),\delta}_*$.
 In case $S=\PP^2$ it is easy to see that $N^{d,\delta}=N^{d,\delta}_*$. 
 By definition the Caporaso-Harris type recursion of \cite{CH98}, \cite{Va00} always computes the invariants 
 $N^{(S,L),\delta}$ for $\PP^2$ and rational ruled surfaces.
 
 If $S$ is $\PP^2$ or a rational ruled surface, in \cite{GS12} \emph{refined Severi degrees} $N^{(S,L),\delta}(y)$ are defined by a modification of the Caporaso-Harris recursion.
 These are again Laurent polynomials in $y$, symmetric under $y\mapsto \frac{1}{y}$.
Again, in the case of $\PP^2$, we denote the refined Severi degrees by $N^{d,\delta}(y)$.
 The recursion specializes to that  \cite{CH98}, \cite{Va00} at $y=1$, so that $N^{(S,L),\delta}(1)=N^{(S,L),\delta}$.

 In this paper we will relate the refined Severi degrees 
 $N^{(S,L),\delta}(y)$ and $N^{d,\delta}(y)$ to tropical geometry.
 Mikhalkin \cite{Mi05} has shown that the Severi degrees of projective toric surfaces can be computed by toric geometry. 
Fix a lattice polygon $\Delta$ in $\RR^2$, i.e. $\Delta$ is the  convex hull of a finite subset of $\ZZ^2$.
Then $\Delta$ determines via its normal fan a projective toric surface $X(\Delta)$ and an ample line bundle $L=L( \Delta)$ on $X(\Delta)$ (and $H^0(X(\Delta),L(\Delta))$ can be identified with the vector space with basis $\Delta\cap \ZZ^2$). 
  Conversely a pair $(X,L)$ of a toric surface and a line bundle on $X$ determines  a lattice polygon. 
We denote by $N^{\Delta,\delta}$ the number of (possibly reducible) cogenus
$\delta$ curves 
of degree $\Delta$ in $(\CC^*)^2$ passing through $|\Delta\cap \ZZ^2|-1-\delta$ general points, as defined in \cite[Def.~5.1]{Mi05}. 
By definition $N^{(X(\Delta),L(\Delta)),\delta}=N^{\Delta,\delta}$. The invariants $N^{\Delta,\delta}$  can be computed in tropical geometry.

 If $X(\Delta)$ is $\PP^2$ or a rational ruled surface, we will in the
 future also write $N^{\Delta,\delta}(y):=N^{(X(\Delta),L(\Delta)),\delta}(y)$
 for the corresponding (refined) Severi degrees as defined in \cite{GS12}. By our definition we then have 
 $N^{\Delta,\delta}(1)=N^{\Delta,\delta}$.

In tropical geometry the Severi degrees $N^{\Delta,\delta}$ can be computed as the count of simple tropical curves  $C$ in $\RR^2$
 through $\dim|L(\Delta)|-\delta$ general points, counted with certain multiplicities $\mult_\CC(C)$.
 Roughly speaking, a simple tropical curve is a trivalent graph $C$ immersed in $\RR^2$, with some extra data.
 From this data, one assigns to each vertex $v$ of $C$ a  multiplicity $\mult_\CC(v)$, and defines the multiplicity $\mult_\CC(C)$ as the product
 $\prod_{v\ \text{vertex of }C} \mult_\CC(v)$.

 For any  integer $n$, and a variable $y$, we introduce the \emph{quantum number} $[n]_y$ by 
 \begin{equation}
 \label{quantum}
 [n]_y=\frac{y^{n/2}-y^{-n/2}}{y^{1/2}-y^{-1/2}}=y^{(n-1)/2}+\cdots +y^{-(n-1)/2}.
 \end{equation}
 By definition $[n]_1=n$.
 We introduce a new polynomial multiplicity $\mult(C; y)\in \ZZ_{\ge 0} [y^{1/2},y^{-1/2}]$ for tropical curves by
 $\mult(C; y)=\prod_{v\ \text{vertex of }C} [\mult_\CC(v)]_y$, and define the \emph{tropical refined Severi degrees} 
 $N^{\Delta,\delta}_{\text{trop}}(y)$ as the count of simple tropical curves  $C$ in $\RR^2$
 through $\dim|L(\Delta)|-\delta$ general points with multiplicity
 $\mult(C; y)$.
 By definition $N^{\Delta,\delta}_{\text{trop}}(y)\in \ZZ_{\ge 0} [y^{1/2},y^{-1/2}]$.
 By definition $[\mult_\CC(v)]_1=\mult_\CC(v)$,
 thus we see that 
 $N^{\Delta,\delta}_{\text{trop}}(1)=N^{\Delta,\delta}$.
 
 A priori, $N^{\Delta,\delta}_{\text{trop}}(y)$  should  depend on a configuration $\Pi$ of $\dim|L(\Delta)|-\delta$ general points in $\RR^2$ but
Itenberg and Mikhalkin show in \cite{IM12} that  $N^{\Delta,\delta}_{\text{trop}}(y)$ is a \emph{tropical invariant}, i.e. independent of 
$\Pi$. 
 
 We will prove that in the case of the plane and rational ruled surfaces, when the refined Severi degrees have been defined in \cite{GS12}, they
 equal the tropical refined Severi degrees.
 
 \begin{theorem}\label{thm:satisfyrecursion}
 Let $X(\Delta)$ be $\PP^2$ or a  rational ruled surface or $\PP1,1,m)$. 
 Then the tropical refined Severi degrees satisfy the recursion \eqref{refrec} for the refined Severi degrees.
 
Thus $N^{\Delta,\delta}_{\trop}(y)=N^{(X(\Delta),L(\Delta),\delta}(y).$
\end{theorem}

We also determine a Caporaso-Harris type recursion formula for $X(\Delta)$ the weighted projective space $\PP(1,1,m)$
(cf. \thmref{thm:refined=tropical}).

The computation of the Severi degrees via tropical geometry and the proof of the existence of node polynomials 
$N_\delta(d)$
 uses a class of decorated graphs called floor diagrams.
The new refined multiplicity $\mult(C; y)$ on tropical curves gives rise to a $y$-statistics  on floor diagrams, which allows to adapt the arguments to the refined tropical Severi degrees. This  statistic is a $q$-analog
of the one of Brugall\'e and Mikhalkin~\cite{BM2} who gave a
combinatorial formula for the Severi degrees $N^{d,
  \delta}$. Theorem~\ref{thm:satisfyrecursion} is a $q$-analog of their
\cite[Theorem~3.6]{BM2} for the refined Severi degrees $N^{d,
  \delta}(y)$.

Using our
combinatorial description, we show that the refined Severi degrees
become polynomials for sufficiently large degree.
\begin{theorem}
\label{thm:refinednodepoly}
For fixed $\delta \ge 1$, there is a
polynomial $N_\delta(d; y) \in \QQ[y, y^{-1}, d]$ of degree
$2\delta$ in $d$ and $\delta$ in $y$ and $y^{-1}$,
such that
\[N_\delta(d; y) = N^{d, \delta}(y),
\]
provided that $d \ge \delta$.
\end{theorem}
We call the $N_\delta(d; y)$ \emph{refined node polynomials}.

The refined invariants $\widetilde N^{(S,L),\delta}(y)$ were computed in \cite{GS12} for $\delta\le 10$, and there it was
conjectured that the refined Severi degrees $N^{d,\delta}(y)$ agree with the refined invariants $\widetilde N^{d,\delta}(y)$
for $d\ge \frac{\delta}{2}+1$. If we assume this conjecture, it would follow from \thmref{thm:refinednodepoly} that $\widetilde N^{d,\delta}(y)=N_\delta(d;y)$, in particular conjecturally the bound on $d$ can be considerably improved.
We  use the refined Caporaso-Harris recursion formula to compute $N^{d, \delta}(y)$ for $\delta\le 10$ and $d\le 30$. 
Together with \thmref{thm:refinednodepoly} this gives the following.
\begin{corollary} 
\label{cor:wheretheyagree}
For $\delta \le 10$ and any $d \ge \tfrac{\delta}{2} + 1$, we have $\widetilde N^{d,
  \delta}(y) = N^{d, \delta}(y)=N_\delta(d,y)$.
\end{corollary}

\begin{corollary} 
For $\delta \le 10$ and any $d \ge \tfrac{\delta}{2} + 1$, $\widetilde N^{d,
  \delta}(y)$, as a Laurent polynomial in $y$, has non-negative integral coefficients.
\end{corollary}

Our combinatorial description of the Laurent polynomials $N^{d,
  \delta}(y)$ allows for effective computation of the refined node
polynomials; for details see Remark~\ref{rem:smallrefinedpolys}.
For $\delta \le 3$, the polynomials $N_\delta(d; y)$ are explicitly given by
Remark~\ref{rem:smallrefinedpolys}.
For $\delta\le 10$ they are given by \thmref{thm:refinedInvEqualsRefinedSeveri}
(proving the formula of \conjref{Gconj} for $\delta\le 10$).

G\"ottsche and Shende also observed a connection between refined
invariants and \emph{real} algebraic geometry. Specifically, they
conjectured that $\widetilde N^{d, \delta}(-1)$ equals  the
\emph{tropical Welschinger invariant} $W_\trop^{d, \delta}$ (for the
definition and details see~\cite{IKS09}), for $d \ge \tfrac{\delta}{3}
+ 1$. Furthermore, by definition $N^{d,
  \delta}(-1)=  W_\trop^{d, \delta}$, i.e.  the refined Severi degree specializes, at $y = -1$ and for all $d$, to the tropical Welschinger invariant.
The numbers $W_\trop^{d, \delta}$, in
turn, equal counts of real plane curves (i.e., complex plane curves
invariant under complex conjugation), counted
with a sign, through particular configurations of real
points~\cite[Proposition~6.1]{Sh05}. Indeed, at $y = -1$, the new
$y$-statistic on floor diagrams specializes to the ``real
multiplicity'' of Brugall\'e and Mikhalkin~\cite{BM2}, and
Theorem~\ref{thm:satisfyrecursion}
becomes \cite[Theorem~3.9]{BM2} for the numbers $N^{d,
  \delta}(-1) = W_\trop^{d, \delta}$.

The recursion formula \ref{refrec} simplifies considerably if we specialize $y=-1$. Therefore we have been able to  use the recursion to compute $N^{d,\delta}(-1)$ for $\delta\le 15$ and $d\le 45$. As by \thmref{thm:refinednodepoly} $N_\delta(d,-1)$ is a polynomial in $d$ of degree at most $2\delta$, this determines $N_\delta(d,-1)$ for $d\le 15$.  On the other hand in \cite{GS12} the $\widetilde N^{(S,L),\delta}(-1)$ are computed  for all $S$, $L$ and 
$\delta\le 14$. 

\begin{corollary}
$\widetilde N^{d,\delta}(-1)=N_\delta(d,-1)= W_\trop^{d, \delta}$  for  $\delta\le 14$ and all $d \ge \tfrac{\delta}{3}+1$.
\end{corollary}

We expect our methods to compute refined Severi degrees also for other
toric surfaces. Specifically, we expect the argument to generalize to
toric surfaces of ``$h$-transverse'' polygons, along the lines
of~\cite{AB10} (see Remark~\ref{rmk:hTransverse}). Notice that such
surfaces are in general not smooth and are thus outside the realm of
the (non-refined) G\"ottsche conjecture~\cite{Go}.

One may speculate about the meaning of refined Severi degrees at other
roots of unity. At $y = -1$, we obtain a (signed) count of complex curves
invariant under the involution of complex conjugation, at least in
genus $0$. This shows the occurrence of a \emph{cyclic sieving
  phenomenon} \cite{Sa11} of order $2$. At least for $y = i$, the
imaginary unit, the refined Severi degree again specializes to an integer
$N^{\Delta, \delta}(i) \in \ZZ$. It would be interesting to find a non-tropical
enumerative interpretation for these numbers. 

This paper is organized as follows.
In Section~\ref{sec:refined}, we review, following G\"ottsche and
Shende, the refined invariants and refined Severi degrees, the latter
for the surfaces $\PP^2$, $\Sigma_m$, and $\PP(1,1,m)$. In
Section~\ref{sec:tropicalRefinedCounting}, we introduce a refinement
of tropical curve enumeration for toric surfaces and extend the notion
of refined Severi degrees to this class. In
Section~\ref{sec:properties} we discuss various polynomiality and other
properties of the refined Severi degrees. In Section~\ref{sec:FD}, we
refine the floor diagram technique of Brugall\'e and Mikhalkin and
template decomposition of Fomin and Mikhalkin, and use it in
Section~\ref{sec:refinednodepolys} to prove the results stated in
Section~\ref{sec:properties}. Finally, in
Section~\ref{sec:RefRelSevDegs}, we introduce tropical refined relative Severi
degrees and show that they agree with the refined Severi degrees of
the G\"ottsche and Shende.

\smallskip

{\bf Acknowledgements.}
The first author thanks Ilia Itenberg, Martin
Kool, and Damiano Testa for helpful discussions, and Diane Maclagan
for telling him about this problem. The second author thanks Sam Payne and Vivek Shende for very useful discussions.

\section{Refined invariants and refined Severi Degrees}
\label{sec:refined}

In this section  we review the
 definition of the closely related notions of the refined
invariants and the refined Severi
degrees from~\cite{GS12}. 
In Section~\ref{sec:tropicalRefinedCounting} we will show that the refined Severi degree also has a simple combinatorial 
interpretation in terms of tropical geometry.

Recall that the Severi degree $N^{d, \delta}$ is the
degree of the Severi variety parametrizing $\delta$-nodal 
plane curves of degree $d$ in $\PP^2$. Equivalently, $N^{d, \delta}$ is the number of such
curves through $\tfrac{(d+3)d}{2} - \delta$ generic points in $\PP^2$.
More generally given a line bundle $L$ on a surface $S$, one can define the Severi degree
$N^{(S,L),\delta}$ as the number of $\delta$-nodal reduced curves in the complete linear system $|L|=\PP(H^0(S,L))$
passing through $\dim|L|-\delta$ general points. 

\subsection{Refined invariants}

For a line bundle $L$ on $S$ we denote by $g(L):=\frac{L(L+K_S)}{2}+1$ the arithmetic genus of a curve in $|L|$.
 For $\delta \ge 0$, let $\PP^\delta$
be a general $\delta$-dimensional subspace of $|L|$. Let
$\C \to \PP^\delta$ be the \emph{universal curve}, i.e., $\C$ is the subscheme
\[
\C = \{ (p, [C]) \, : \, p \in C \} \subset S \times \PP^\delta
\]
with a natural map to $\PP^\delta$. Here, $[C]$ denotes the curve $C$ viewed as a point of $\PP^\delta$. Thus the fiber of $\C \to \PP^\delta$ over $[C] \in\PP^\delta$ is the
curve $C$. Let $S^{[n]}=\Hilb^n(S)$ be the Hilbert scheme of $n$ points in
$S$. Finally,
let $\Hilb^n(\C / \PP^\delta)$ be the relative Hilbert
scheme
\[
\Hilb^n(\C / \PP^\delta) = \{ ([Z], [C]) \, : \, Z \subset C \}
\subset S^{[n]} \times \PP^\delta.
\]
 Here, $[Z]$ is the the subscheme $Z$ viewed as a point of $S^{[n]}$ and $Z\subset C$ means that $Z$ is a subscheme of $C$.

Recall that a line bundle $L$ on $S$ is  called $\delta$-very ample, if the restriction map
$H^0(S,L)\to H^0(L|_Z)$ is surjective for all zero dimensional subschemes $Z\in S^{[\delta+1]}$.
In the introduction we had changed the definition of the Severi degrees for toric surfaces, defining $N^{(S,L),\delta}$ to be the count of $\delta$-nodal curves in $|L|$ through generic points, which do not contain a toric boundary divisor. The count of curves without this condition we denoted $N^{(S,L),\delta}_*$.
\begin{remark}\label{rem:delta*}
Let $L$ be $\delta$-very ample on a  surface $S$, then the curves in $|L|$ containing a given curve as a component occur in codimension at least $\delta+1$.
In particular if $L$ is a $\delta$-very ample toric line bundle on the
toric surface $S$, then $N^{(S,L),\delta}=N^{(S,L),\delta}_*$.
\end{remark}
\begin{proof}
  Let $C$ be be a curve on $S$. Let $Z$ be any $0$-dimensional subscheme of $C$ of length $\delta+1$. 
 Then by $\delta$-very ampleness the canonical restriction map $\rho:H^0(S,L)\to H^0(L|_Z)$ is surjective. The sections  $s$ of $L$ such that $Z(s)$ contains $C$ as a component
 lie in the kernel of $\rho$, thus curves having $C$ as a component occur in codimension at least $\delta+1$ in $|L|$.
  \end{proof}

We review the definition of the refined invariants $\widetilde N^{(S,L),\delta}(y)$ in case $\Hilb^n(\C / \PP^\delta)$ is nonsingular of dimension $n+\delta$ for all $n$. A sufficient condition for this is that
$L$ is $\delta$-very ample,
see \cite[Thm.~41]{GS12}.

In their proof~\cite{KST11} of the G\"ottsche
conjecture~\cite[Conjecture~2.1]{Go}, Kool, Shende, and Thomas showed,
partially based on~\cite{PT10},that,  if $L$ is $\delta$-very ample, 
 the Severi degrees $N^{(S,L), \delta}$ can be computed from the generating function of their Euler characteristics.
Specifically, they show~\cite[Theorem~3.4]{KST11} that, under this assumption, there exist integers $n_r$, for $r = 0,\ldots, \delta$, such that
\begin{equation}
\label{eqn:EulerFormula}
\sum_{i = 0}^\infty e(\Hilb^i(\C / \PP^\delta)) t^i = \sum_{r = 0}^
  \delta n_r \,  t^{r} \, (1-t)^{2g(L)-2-2r}.
\end{equation}
Here,  $e(-) = \sum_{i \ge 0} (-1)^i \rk H^i(-,\ZZ)$ denotes the topological
Euler characteristic.
Furthermore, they showed that the Severi degree $N^{(S,L),
  \delta}$ equals the coefficient $n_{\delta}$ in
(\ref{eqn:EulerFormula}).

Inspired by this description,  G\"ottsche and
Shende~\cite{GS12} suggest to replace in (\ref{eqn:EulerFormula}) the
Euler characteristic $e(-)$ by the  $\chi_{-y}$-genus
\begin{equation}
\label{eqn:Hodgepoly}
\chi_{-y}(-) = \sum_{p, q \ge 0} (-1)^{p+q} \, y^q \, h^{p,q}(-),
\end{equation}
where $h^{p, q}(-)$ are the Hodge numbers.
The polynomial $\chi_{-y}$ is the \emph{Hodge polynomial} $H(\tilde{x},\tilde{y})(-) =
\sum_{p,q \ge 0} \tilde{x}^p \, \tilde{y}^q \, h^{p,q}(-)$,
at $\tilde{x} = -y$ and $\tilde{y} = -1$. 
They prove the following:
\begin{proposition} \label{propchivan}

 Assume $\Hilb^n(\C / \PP^\delta)$ is nonsingular for all $n$. 
Then there exist polynomials $n_0(y), \ldots, n_{g(L)}(y)$ such that 
\begin{equation}
\label{eqn:HodgeFormula}
\sum_{i = 0}^\infty \chi_{-y}(\Hilb^n(\C / \PP^\delta)) t^n = \sum_{r = 0}^
  {g(L)} n_r(y) \,  t^{r} \, \big((1-t)(1-ty)\big)^{g(L)-r-1}.
\end{equation}
\end{proposition}

This is a weak form of an analogue of \eqref{eqn:EulerFormula}. They conjecture that a precise analogue holds.

\begin{conjecture}\label{conjchivan} Under the assumptions of \propref{propchivan}, we have that
$n_r(y)=0$ for $r>\delta$
\end{conjecture}

\begin{definition}
\label{def:refinedInv}
Under the assumptions of Proposition \ref{propchivan}
we put $\widetilde N^{(S,L),\delta}(y):=n_{\delta}(y)/y^{\delta}$, where $n_{\delta}(y)$ is the polynomial in   \eqref{eqn:HodgeFormula}. Following \cite{GS12}, we call the polynomials $\widetilde N^{(S,L), \delta}(y)$ the
\emph{refined invariants} of $S,L$ (there they are called normalized refined invariants). 
It is easy to see from the definition that $\widetilde N^{(S,L),\delta}(y)$ is a Laurent polynomial in $y$, symmetric under $y\mapsto \frac{1}{y}$.
\end{definition}

Finally we extend the definition of the refined invariants $\widetilde N^{(S,L),\delta}(y)$ to arbitrary $L$ and $\delta$, when the $\Hilb^n(\C / \PP^\delta)$ might be singular, or they might not even exist (e.g. if $\delta>\dim |L|$).

Let $f(z):=\frac{z(1-ye^{-z(1-y)})}{1-e^{-z(1-y)}}\in 1+z\QQ[y][[z]]$.
Now let $S$ be smooth projective surface, $L$ a line bundle on $S$. 
Let $Z_n(S)\subset S\times S^{[n]}$ be the universal family with projections
$p:Z_n(S)\to S^{[n]}$, $q:Z_n(S)\to S$.
Let $L^{[n]}:=p_*q^*L$, a vector bundle of rank $n$ on $S$, denote $l_1,\ldots,l_n$ its Chern roots, and denote $t_1,\ldots, t_{2n}$ the Chern roots of the tangent bundle $T_{S^{[n]}}$. The following is proven in \cite[Prop.~47]{GS12}.

\begin{proposition}
Assume $\Hilb^n(\C / \PP^\delta)$ is nonsingular for all $n$. Then 
\begin{equation}\label{eq:resform}
 \chi_{-y}(\Hilb^n(\C / \PP^\delta))=\res_{x=0} \left[
\left(\frac{f(x)}{x}\right)^{\delta+1}\int_{S^{[n]}}\prod_{i=1}^{2n} f(t_i) \prod_{j=1}^n \frac{l_j}{f(l_j+x)}\right].
\end{equation}
(By definition $\prod_{i=1}^{2n} f(t_i) \prod_{j=1}^n \frac{l_j}{f(l_j+x)}\in H^*(S^{[n]},\QQ)[y][[x]]$, thus the term in square brackets on the left hand side of \eqref{eq:resform} is a Laurent series in $x$ with coefficients in $\QQ[y]$.)
\end{proposition}

\begin{definition}Let $L$ be a line bundle on a projective surface $S$, let $\delta\in \ZZ_{\ge 0}$. 
The \emph{refined invariants} $\widetilde N^{(S,L),\delta}$ are defined by replacing $\chi_{-y}(\Hilb^n(\C / \PP^\delta))$ by the right hand side of 
\eqref{eq:resform}  in Definition \ref{def:refinedInv}
and \eqref{eqn:HodgeFormula}.

We write $\widetilde N^{d,\delta}(y)=\widetilde N^{(\PP^2,\O(d)),\delta}(y)$ for the refined invariants of 
$\PP^2$.
\end{definition}

At $y = 1$ we have  $\chi_{-1}(-) = e(-)$ and thus recover the
Severi degree as the special case $\widetilde N^{(S,L),  \delta}(1) =
N^{(S,L), \delta}$, for $\Hilb^n(\C/\PP^\delta)$ nonsingular, from \cite[Theorem~3.4]{KST11}. 
The  $\widetilde N^{(S,L), \delta}(y)$ satisfy  universal polynomiality~\cite{GS12}: for each
$\delta$, there is a polynomial $\widetilde
N_\delta(x_1,x_2,x_3,x_4;y)$, such that $\widetilde N^{(S,L), \delta}(y)=\widetilde N_\delta(L^2,LK_S,K_S^2,e(S);y)$. In particular there exist polynomials 
$\widetilde N_\delta(d; y)$ in $d$ and $y$ such
that $\widetilde N_\delta(d; y) = \widetilde N^{d, \delta}(y)$ for all $d,\delta$.
Assuming \conjref{conjchivan}, these polynomials have a multiplicative generating function:
there exist universal power series 
$A_1,A_2,A_3,A_4\in \QQ[y^{\pm 1}][[q]]$, such that
$$\sum_{\delta\ge 0} \widetilde N^{(S,L), \delta}(y)q^{\delta}=
A_1^{L^2}A_2^{LK_S}A_3^{K_S^2}A_4^{e(S)}.$$

More precisely in 
\cite[Conjecture~67]{GS12} a conjectural generating function for the refined invariants $\widetilde N^{(S,L),\delta}(y)$ is given:  Let 
 \begin{align*}
\widetilde\Delta(y,q)&:= q \prod_{n=1}^{\infty}(1-q^n)^{20}(1-yq^n)^{2}
(1-y^{-1}q^n)^2=q - (2y + 2 + 2y^{-1})q  + \ldots,\\
\widetilde{DG}_2(y,q)&:=\sum_{m= 1}^\infty  q^{m} \sum_{d | m}[d]_y^2\frac{m}{d}=q + (y + 4+ y^{-1})q^2  + \ldots.
\end{align*}
Denote $D:=q\frac{\partial}{\partial q}$.
 \begin{conjecture}
\label{Gconj}
There exist universal power series 
$B_1(y,q)$, $B_2(y,q)$ in $\QQ[y,y^{-1}]\lbr q\rbr $, such that 
\begin{equation}
\label{Gform}
\sum_{\delta\ge 0}
\tN^{(S,L),\delta}(y) (\widetilde{DG}_2)^\delta=\frac{(\widetilde{DG}_2/q)^{\chi(L)}B_1(y,q)^{K_S^2}B_2(y,q)^{LK_S}}
{\big(\widetilde\Delta(y,q)\, D \widetilde{DG}_2/q^2)^{\chi(\O_S)/2}}
\end{equation}
Here, to make the change of variables, all functions are viewed as elements of 
$\QQ[y,y^{-1}]\lbr q\rbr $.
\end{conjecture}
Equivalently, letting 
$$g(y,t)= t - (y +4+y^{-1})t^2 + (y^2 + 14y + 30 + 14y^{-1} + y^{-2})t^3 + \ldots
$$
 be the compositional inverse of $\widetilde{DG}_2$, \eqref{Gform} says 
\begin{equation}
\label{Gform1}
\sum_{\delta\ge 0}
\tN^{(S,L),\delta}(y) t^\delta=(t/g(y,t))^{\chi(L)}\frac{B_1(y,q)^{K_S^2}B_2(y,q)^{LK_S}}
{\big(\widetilde\Delta(y,q)\, D \widetilde{DG}_2/q^2)^{\chi(\O_S)/2}}\Big|_{q=g(y,t)}.
\end{equation}

In \cite{GS12} this conjecture is proven modulo $q^{11}$ and the power series $B_1(y,q)$, $B_2(y,q)$ are determined modulo $q^{11}$
(the result can be found directly after \cite[Conj.~67]{GS12}). 
Here we list $B_1(y,q)$, $B_2(y,q)$ for completeness modulo $q^6$. 
{\small\begin{align*}
&B_1(y,q)=1 - q - ((y^2 +3y + 1)/y)q^2 
+ ((y^4 + 10y^3 + 17y^2 + 
10y+ 1)/y^2)q^3  
 \\ &- ((18y^4 + 87y^3 + 135y^2 + 87y + 18)/y^2)q^4\\ &
+ ((12y^6 + 210y^5 + 728y^4 + 1061y^3 + 728y^2 + 210y + 12)/y^3)q^5 +O(q^6),\\ 
 &B_2(y,q)=\frac{1}{(1-yq)(1-q/y)}\Big(1 + 3q - ((3y^2 + y + 3)/y)q^2 \\ &+ ((y^4 + 8y^3 + 18y^2 + 8y + 1)/y^2)q^3- ((13y^4 + 53y^3 + 76y^2 + 53y + 13)/y^2)q^4 \\& + ((7y^6 + 100y^5 + 316y^4 + 455y^3 + 316y^2+ 100y  + 7)/y^3)q^5+O(q^6)\Big).\\
\end{align*}}
This gives a formula for the $\widetilde N^{S,L),\delta}(y)$ as explicit polynomials of degree at most $\delta$ in $L^2$, $LK_S$, $K_S^2$, $\chi(\O_S)$
proven for $\delta\le 10$.
The $\widetilde N^{d,\delta}(y)$ are obtained from this by specifying $\chi(L)=\binom{d+2}{2}$, $LK_S=-3$, $K_S^2=9$, $\chi(\O_S)=1$, giving them as polynomials of degree at most  $2\delta$ in $d$.

\subsection{Refined Severi degrees}
Throughout this section we take $S$ to be $\PP^2$,  a rational ruled surface, or a weighted projective space $\PP(1,1,m)$.
In case $S=\PP^2$, let $H$ be a line in $\PP^2$; in case $S$ is a rational ruled surface $\Sigma_m=\PP(\O_{\PP^1}\oplus\O_{\PP^1}(-m))$, let $H$ be the class of a section with $H^2=m$, let $E$ be the class of the section with $E^2=-m$ and $F$ the class of a fibre on $\Sigma_m$.
We denote $H$ the class of a line in $\PP(1,1,m)$ with $H^2=m$. 
For a rational ruled surface $\Sigma_m$ we can also allow $m$ to be negative. In this case $\Sigma_m=\Sigma_{-m}$, but the role of $H$ and $E$ is exchanged.
Therefore below in the case of $\Sigma_m$ we actually represent two different recursion formulas.

Caporaso and Harris showed that the  Severi degrees $N^{d,\delta}$ satisfy a recursion formula \cite{CH98}.
A similar recursion formula computes the Severi degrees $N^{(S,L),\delta}$ on rational ruled surfaces \cite{Va00}.
In \cite{GS12}  a refined Caporaso-Harris type recursion formula is used to define   Laurent polynomials 
$N^{(S,L),\delta}(y)$, which the authors call \emph{refined Severi degrees}. By definition for $y=1$ these polynomials specialize to the 
Severi degrees: $N^{(S,L),\delta}(1)=N^{(S,L),\delta}$.
We now briefly review this recursion and also extend it to $\PP(1,1,m)$.

By a {\it sequence} we mean a collection $\alpha=(\alpha_1,\alpha_2,\ldots)$ of nonnegative integers, almost all of which are zero. For two sequences
$\alpha$, $\beta$ we define 
$|\alpha|=\sum_i\alpha_i$, $I\alpha=\sum_i i\alpha_i$, 
$\alpha+\beta=(\alpha_1+\beta_1,\alpha_2+\beta_2,\ldots)$, and $\binom{\alpha}{\beta}=\prod_i \binom{\alpha_i}{\beta_i}$. We write $\alpha\le\beta$ to mean $\alpha_i\le \beta_i$ for all $i$.
We write $e_k$ for the sequence whose $k$-th element is $1$ and all other ones  $0$.
We usually omit writing down trailing zeros. 

For sequences $\alpha$, $\beta$, and $\delta\ge 0$, let $\gamma(L, \beta,\delta)=\dim|L|-HL+|\beta|-\delta$. 
The relative Severi degree $N^{(S,L),\delta}(\alpha,\beta)$ is the number of $\delta$-nodal curves in $|L|$ not containing $H$, through $\gamma(L, \beta,\delta)$ general points,
and with $\alpha_k$ given points of contact of order $k$ with $H$, and $\beta_k$ arbitrary points of contact of order $k$ with $H$.

\begin{definition}[{\cite[Recur.~76, Prop.~78]{GS12}}] \label{refCHrec}
Recall the definition of the quantum numbers $[n]_y=\frac{y^{n/2}-y^{-n/2}}{y^{1/2}-y^{-1/2}}$.
Let  $L$ be a line bundle on $S$ and let $\alpha$, $\beta$ be sequences with $I\alpha+I\beta=HL$, and 
let $\delta\ge 0$ be an integer. We define the \emph{refined relative Severi degrees} $N^{(S,L),\delta}(\alpha,\beta)(y)$ 
recursively as follows: if $\gamma(L, \beta,\delta)>0$, then 
\begin{equation}\label{refrec}
\begin{split}
N^{(S,L),\delta}(\alpha,\beta)(y)&=\sum_{k:\beta_k>0} [k]_y \cdot  N^{(S,L),\delta}(\alpha+e_k,\beta-e_k)(y)\\
&+\sum_{\alpha',\beta',\delta'}
\left(\prod_i [i]_y^{\beta_i'-\beta_i}\right)
\binom{\alpha}{\alpha'}\binom{\beta'}{\beta}  N^{(S,L-H),\delta'}(\alpha',\beta')(y).
\end{split}
\end{equation}
Here the second sum runs through all  $\alpha',\beta',\delta'$ satisfying the condition
\begin{equation}\label{relcong}
\begin{split}
\alpha'&\le \alpha, \ \beta'\ge \beta,\ I\alpha'+I\beta'=H(L-H),\\ \delta'&=\delta+g(L-H)-g(L)+|\beta'-\beta|+1=\delta-H(L-H)+|\beta'-\beta|.
\end{split}
\end{equation}
{\bf Initial conditions:} if $\gamma(L,\beta,\delta)=0$ we have  $N^{(S,L),\delta}(\alpha,\beta)(y)=0$ unless we are in one of the following cases
\begin{enumerate}
\item 
In case $S=\PP^2$ we put $N^{H,0}((1),(0))(y)=1$,
\item
In case $S=\Sigma_m$, let $F$ be the class of a fibre of the ruling;
we put $N^{kF,0}((k),(0))(y)=1$.
\item 
In case $S=\PP(1,1,m)$, $L=dH$, we put 
and $N^{H,0}((1),(0))(y)=1$.
\end{enumerate}
We abbreviate $N^{(S,L),\delta}(y):=N^{(S,L),\delta}((0),(LH))(y)$,  and,  
in case $S=\PP^2$, $N^{d,\delta}(\alpha, \beta)(y):=N^{(\PP^2,\O(d)),\delta}(\alpha, \beta)(y)$, 
$N^{d,\delta}(y):=N^{d,\delta}((0),(d))(y)$.
The refined relative Severi degrees  are 
Laurent polynomials in $y^{1/2}$, symmetric under $y\mapsto 1/y$.
\end{definition}

\begin{remark}
As mentioned in the beginning of this section, for $S$ a Hirzebruch
surface this recursion is defined for $m\in \ZZ$; in this case $\Sigma_{-m}=\Sigma_m$ but 
the class $H$ on $\Sigma_{-m}$ is the class $E$ on $\Sigma_m$.  For $m\in \ZZ$, we will write $N^{(\Sigma_m,L),\delta}(\alpha,\beta)(y)$ for the invariants obtained by this recursion.  Below in \thmref{thm:refined=tropical} we will see that $N^{(\Sigma_m,L),\delta}(y)=N^{(\Sigma_{-m},L),\delta}(y)$. In general we do not have $N^{(\Sigma_m,L),\delta}(\alpha,\beta)(y)=N^{(\Sigma_{-m},L),\delta}(\alpha,\beta)(y)$, because (expressed on $\Sigma_m$) the first counts curves with contact conditions along $H$ and the second with contact conditions along $E$.
\end{remark}

\begin{remark} The recursions for the refined Severi degrees are chosen so that they specialize at $y=1$ to the recursion for the usual Severi degrees. 
Furthermore the recursions for the tropical Welschinger numbers $W^{(S,L),\delta}_{\text trop}(\alpha,\beta)$ are obtained by specializing instead to $y=-1$. 
Thus we we get:
\begin{equation}\label{eq:pmone}
\begin{split}
N^{(S,L),\delta}(\alpha,\beta)(1)=N^{(S,L),\delta}(\alpha,\beta),\quad N^{(S,L),\delta}(1)=N^{(S,L),\delta}, \\
N^{(S,L),\delta}(\alpha,\beta)(-1)=W^{(S,L),\delta}_{\text trop}(\alpha,\beta),\quad N^{(S,L),\delta}(-1)=W^{(S,L),\delta}_{\text trop}.
\end{split}
\end{equation}
\end{remark}

According to \cite{KS12}, if the general $\PP^\delta \subset |L|$ contains no non-reduced curves
and no curves containing components with negative self intersection, the Severi degrees are computed
by the universal formulas.
  We expect the same for refined Severi degrees. 

\begin{conjecture}[\cite{GS12}]\label{ref-sev}
Let $S$ be $\PP^2$ or a rational ruled surface, let $L$ be a line bundle, and assume 
$\PP^\delta \subset |L|$ contains no non-reduced curves
and no curves containing components with negative self intersection.  
Then the refined Severi degrees are computed
by the universal formulas: $N^{(S,L),\delta}(y) = \widetilde N^{(S,L),\delta}(y)$.  Explicitly,
\begin{enumerate}
\item On $\PP^2$ we have $N^{d,\delta}(y)=\widetilde N^{d,\delta}(y)$, for $d\ge \frac{\delta}{2}+1$.
\item Assume $c+d>0$. We have $N^{(\PP^1\times\PP^1,cF+dH),\delta}(y)=\widetilde N^{(\PP^1\times\PP^1,cF+dH),\delta}(y)$,
for $c,d\ge \frac{\delta}{2}$.
\item On $S=\Sigma_m$ with $m>0$, assume $d+c>0$. Then $N^{(S,cF+dH),\delta}=\widetilde N^{(S,cF+dH),\delta}(y$ 
for $\delta\le \min(2d,c)$.
\end{enumerate}
\end{conjecture}

Below in Section~\ref{sec:tropicalRefinedCounting} we introduce the (tropical) refined Severi degrees $N^{\Delta,\delta}(y)$ of toric surfaces  $X(\Delta)$ with line bundles $L(\Delta)$ given by  convex lattice polygons $\Delta$, and we show in 
\thmref{thm:refined=tropical} that these coincide with the refined Severi degrees defined above in the case of $\PP^2$, $\Sigma_m$ and $\PP(1,1,m)$. 

We conjecture more generally:
\begin{conjecture}
\label{conj:smoothToric}
Let $\Delta$ be a convex lattice polygon, such that
 $S=X(\Delta)$ is a smooth surface and $L=L(\Delta)$ a $\delta$-very
ample line bundle. Then the (tropical) refined Severi degrees are computed
by the universal formulas: 
\[
N^{\Delta,\delta} (y)= \widetilde N^{(S,L),\delta}(y).
\]
\end{conjecture}

In \cite[Cor.~6]{KS12} the following is proven (without the restriction on toric surfaces) for the non-refined invariants, we expect the same is true also in the refined case.
\begin{conjecture}
Let $S$ be a classical toric del Pezzo surface. Assume the following loci have codimension more than $\delta$ in $|L|$:
\begin{enumerate}
\item the nonreduced curves,
\item the curves with a  $(-1)$ curve as a component. 
\end{enumerate}
Then \[
N^{(S,L),\delta}(y) = \widetilde N^{(S,L),\delta}(y).
\]
\end{conjecture}

\begin{remark}
For $m\ge 2$ the weighted projective space $\PP(1,1,m)$ is singular, so  \conjref{ref-sev} of \cite{GS12} does not apply.
In fact the refined invariants $\widetilde N^{(S,L),\delta}(y)$ have not even been defined in this case.

We instead compare the refined Severi degrees $N^{(\PP(1,1,m),dH),\delta}(y)$ to the corresponding refined invariants $\tN^{(\Sigma_m,dH),\delta}(y)$ on the minimal resolution 
$\Sigma_m$ of $\PP(1,1,m)$. 
\end{remark}

We obtain the following conjectures.

\begin{conjecture} \label{P11mconj}
There is a polynomial $N_{\delta}(d,m;y)$ of degree $2\delta$ in $d$ and $\delta$ in $m$, such that
$N^{(\PP(1,1,m),dH),\delta}=N_{\delta}(d,m;y)$ for $\delta\le min(2d-2,2m-1)$.
\end{conjecture}

\begin{conjecture}\label{P11mconjser}
There exist power series $C_1,C_2,C_3\in \QQ[y^{\pm1}][[q]]$, such that
$$\sum_{\delta\ge 0} N_{\delta}(d,m;y)(\widetilde{DG_2})^{\delta}=\Bigg(\sum_{\delta\ge 0}\tN^{(\Sigma_m,dH),\delta}(\widetilde{DG_2})^\delta\Bigg)C_1^{(m+2)d}C_2^{m+2} C_3
.$$
\end{conjecture}
\begin{remark}
\label{corseries}
We have used the Caporaso-Harris recursion to compute $N^{(\PP(1,1,m),dH),\delta}$ for $\delta\le 6$, $d\le 5$ and $m\le 5$.
The results confirm \conjref{P11mconj}, \conjref{P11mconjser}. Furthermore assuming these conjectures they determine
$C_1,C_2,C_3$ modulo $q^{7}$. We list them modulo $q^6$. Conjecturally this gives in particular  $N^{(\PP(1,1,m),dH),\delta}$ for 
$\delta\le 5$, $d\ge 4$, $m\ge 3$.

{\small\begin{align*}
C_1&=1 - ((y^2 + 3y + 1)/y)q + ((6y^2 + 11y + 6)/y)q^2 -((4y^4 + 36y^3 + 60y^2 + 36y+ 4)/y^2)q^3 \\
&+ ((y^6 + 54y^5 + 243y^4 + 373y^3 + 243y^2 + 54y + 1)/y^3)q^4 \\& - ((41y^6 + 525y^5 + 1723y^4 + 2478y^3 + 1723y^2 + 525y + 41)/y^3)q^5 +O(q^6),
%+ ((15y^8 + 665y^7 + 4644y^6 + 12585y^5 + 17217y^4 + 12585y^3 + 4644y^2 + 665y + 15)/y^4)q^6 + O(q^7)
\\
C_2&=\frac{1}{(1-qy)(1-q/y)}\Big( 1 + 2q- ((2y^2 + 2y + 2)/y)q^2 \\
&+ ((y^4 + 6y^3 + 11y^2 + 6y + 1)/y^2)q^3 - ((10y^4 +38y^3 + 56y^2 + 38y + 10)/y^2)q^4 \\
&+ ((7y^6 + 79y^5 + 241y^4 + 339y^3 + 241y^2 + 79y + 7)/y^3)q^5+O(q^6) \Big),\\
%&- ((y^8 + 99y^7 + 639y^6 + 1632y^5 + 2231y^4 + 1632y^3 + 639y^2 + 99*y + 1)/y^4)q^6 + O(q^7)\Big)
C_3&=1 + 2q - ((4y^2 + 6y + 4)/y)q^2 + ((20y^2 + 32y + 20)/y)q^3 - ((19y^4 + 100y^3 + 170y^2 \\&+ 100y + 19)/y^2)q^4 + ((4y^6 + 154y^5 + 564y^4 + 824y^3 + 564y^2 + 154y + 4)/y^3)q^5+O(q^6).
% - ((128y^6 + 1069y^5 + 3248y^4 +4332*y^3 + 3248*y^2 + 1069y + 128)/y^3)q^6 + O(q^7)
\end{align*}}
\end{remark}

Denote by $N_0^{(S,L),\delta}$ the irreducible Severi degrees, i.e.  the number of irreducible $\delta$-nodal curves in $|L|\ne |E|$ passing though $\dim|L|-\delta$ general points. In particular it is clear that $N_0^{(S,L),\delta}\ge 0$ and  $N_0^{(S,L),\delta}= 0$ if $\delta>g(L)$.
In \cite{Ge97} it is noted in case  $S=\PP^2$, and   in \cite{Va00} for rational ruled surfaces, that the $N_0^{(S,L),\delta}$ can be expressed by a formula in terms of the 
Severi degrees $N^{(S,L),\delta}$. In \cite{GS12}  \emph{irreducible refined Severi degrees} $N_0^{(S,L),\delta}(y)$ are defined by the same formula 
\begin{equation}
\label{eqn:irred_refined_Severi_definition}
\sum_{L,\delta} \frac{z^{\dim|L|-\delta}}{(\dim|L|-\delta)!} v^L  N_0^{(S,L),\delta}(y)=
\log\left(1+\sum_{L,\delta}\frac{z^{\dim|L|-\delta}}{(\dim|L|-\delta)!} v^L   N^{(S,L),\delta}(y)\right).
\end{equation}
Here $\big\{v^L\big\}_{L \text{ effective}, L\ne E}$ are elements of the Novikov ring, i.e. $v^{L_1}v^{L_2}=v^{L_1+L_2}$.
Evidently $ N_0^{(S,L),\delta}(y)$ is a Laurent polynomial in $y$ invariant under $y\mapsto 1/y$, and
$N_0^{(S,L),\delta}(1)=N_0^{(S,L),\delta}$.

We will show below that $N_0^{(S,L),\delta}(y)$ is a count of
irreducible tropical curves 
with Laurent polynomials in $y$ with nonnegative integer coefficients
as multiplicities, see \thmref{thm:irreducible_refined_Severi_degree}.
In particular, $N_0^{(S,L),\delta}(y) \in \ZZ_{\ge 0}[y^{\pm
  1}]$. Furthermore, $N_0^{(S,L),\delta}(y)=0$, if $\delta>g(L)$.

\section{Refined Tropical Curve Counting}
\label{sec:tropicalRefinedCounting}

We now define a refinement of Severi degrees for any toric surface, by
introducing a ``$y$-weight'' into Mikhalkin's tropical curve
enumeration. For the surfaces $S = \Sigma_m$ and $S = \PP(1,1,m)$,
the new invariants agree with the refined Severi degrees defined via the
recursion in Definition~\ref{refCHrec}. We extend our definition to
the case of tangency conditions in Section~\ref{sec:RefRelSevDegs}. We denote
tropical curves and classical curves with the same notation $C$, as it
usually will be clear which curves we are talking about.

\begin{definition}
A \emph{metric graph} is a non-empty graph whose edges $e$ have a \emph{length}
$l(e) \in
\RR_{>0} \cup \{ \infty \}$.

An \emph{abstract tropical curve} $C$ is a metric graph with all
vertices of valence $1$ or at least $3$ such
that, for an edge $e$ of $C$, we have length $l(e) = \infty$
precisely when $e$ is
adjacent to a leaf (i.e., a $1$-valent vertex) of $C$. We conventionally remove
the (infinitely far away) leaf vertices from $C$. 
\end{definition}

Note that we do not require the underlying graph of a metric graph to
be connected. Connectedness will correspond to the irreducibility of
algebraic curves. Let $\Delta$ be a lattice polygon in $\RR^2$. A
non-zero vector $u \in \ZZ^2$ is \emph{primitive} if its entries are coprime.

\begin{definition}
\label{def:parametrizedTropicalCurve}
A \emph{(parametrized) tropical curve of degree $\Delta$} is an
abstract tropical curve $C$, together with a continuous map $h: C \to
\RR^2$ satisfying:
\begin{enumerate}
\item (Rational slope) The map $h$ is affine linear on each edge $e$ of $C$, i.e.,
  $h|_e(t) = t\cdot v + a$ for some non-zero $v \in \ZZ^2$ and $a \in
  \RR^2$. If $V$ is a vertex of the edge $e$ and we parametrize $e$
  starting at $V$, then we call $v$ above the \emph{direction vector}
  of $e$ starting at $V$, and we write $v = v(V, e) \in \ZZ^2$. The lattice
  length of $v(V,e)$ (i.e, the greatest integral common divisor of
  the entries of $v(V, e)$) is the \emph{weight} $\omega(e)$ of
  $e$. We call the integral vector $u(V,e) = \tfrac{1}{\omega(e)} v(V,e)$ the
  \emph{primitive direction vector} of $e$.
\item (Balancing) Each vertex $V$ of $C$ is \emph{balanced}, i.e.,
\[
\sum_{e: \, V \in \partial e} v(V, e) = 0.
\]
\item (Degree) For each primitive vector $u \in \ZZ^2$, the total weight
  of the unbounded edges with primitive direction vector $u$ equals the lattice length of an edge of
  $\partial \Delta$ with outer normal vector $u$ (if there is no such
  edge, we require the total weight to be zero).
\end{enumerate}
\end{definition}

\begin{example}
\label{ex:tropicalCurve}
Below, in Figure~\ref{fig:HirzebruchCurve} (left), is an example of a
(parametrized) tropical curve of degree $\Delta$, pictured to its
right. One edge is of weight $2$, all others have weight $1$ (omitted
in the drawing).  All
vertices of $C$ are balanced, for vertex $v$ this means that
$2\binom{0}{1} + \binom{-1}{0} + \binom{1}{-2} = 0$.

\begin{figure}[htbp]
 \begin{tabular}{ccc}
\includegraphics[scale = 0.25]{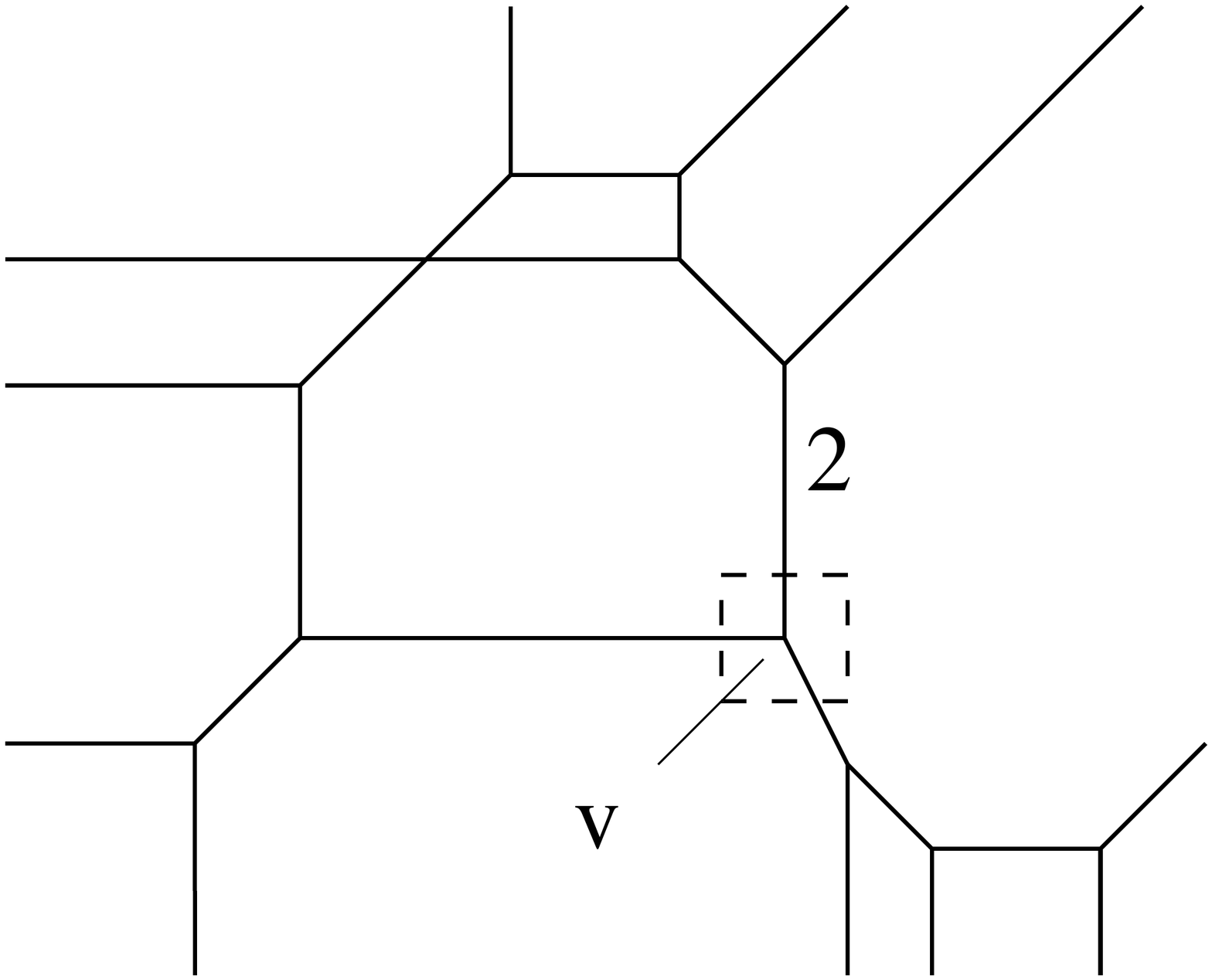}  \quad \quad
&
\quad 
\raisebox{55pt}{
\includegraphics[scale = 1.2]{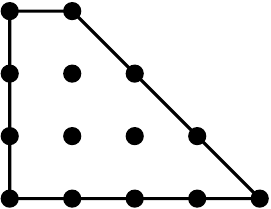}} \quad \quad
&
\quad 
\raisebox{25pt}{
 \includegraphics[scale = 0.15]{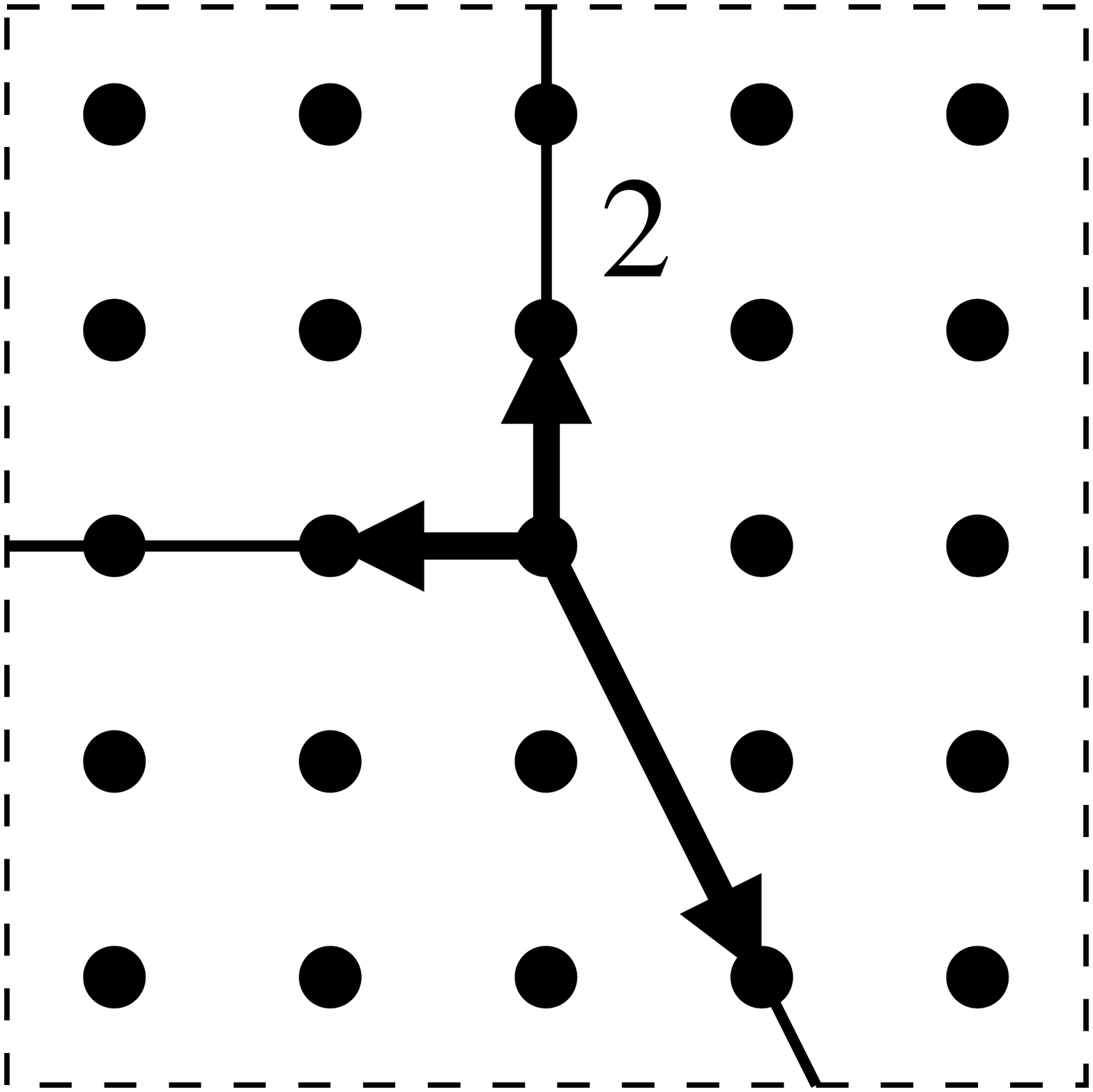}} \\
\vspace{-4em}
 \end{tabular}
\caption{A tropical curve (left) of degree $\Delta$ (middle) and a
  balanced vertex (right).}
\label{fig:HirzebruchCurve}
\end{figure}
\end{example}

In order to define the tropical analogs of the Severi degree and its
refinement, we recall the following tropical notions
(cf. \cite[Section~2]{Mi05}). We sometimes abuse notation and
simply write $C$ for the parametrized tropical curve $(C,
h)$ if no confusion can occur. 

\begin{definition}
\label{def:tropicalCurveProperties}
\begin{enumerate}
\item We say that a tropical curve $(C, h)$ is \emph{irreducible} if
 the underlying topological space of $C$ has exactly $1$
 component. The \emph{genus} $g(C,h)$ of an irreducible tropical curve $(C, h)$ is the
genus (i.e., the first Betti number) of the underlying topological
space of $C$.
\item
\label{itm:dualSubdivision}
 The \emph{dual subdivision} $\Delta_C$ of the parametrized
  tropical plane curve $(C, h)$ is the unique subdivision of $\Delta$
  whose $2$-faces $\Delta_v$ correspond to the vertices $v$ of $h(C)$ such that
 the (images of) edges $e$ of $C$ are orthogonal to the
 edges $e^\perp \in \RR^2$ of $\Delta_C$ and, further, that the
 lattice length of $e^{\perp}$ equals $\omega(e)$, see Figure~\ref{fig:dualSubdivision}.
\item The tropical curve $(C, h)$ is \emph{nodal} if its dual
  subdivision $\Delta_C$ consists only of triangles and
  parallelograms.
\item We say that $(C, h)$ is \emph{simple} if all vertices of
  $C$ are $3$-valent, the self-intersections of $h$ are disjoint
  from vertices, and the inverse image under $h$ of self-intersection
  points consists of exactly two points of $C$.
\item
\label{itm:numberOfNodesIrred} 
The \emph{number of nodes} $\delta(C, h)$ of a nodal irreducible
  tropical curve of degree $\Delta$ is $\delta(C, h) = |\Delta^0
  \cap \ZZ^2|  - g(C, h)$, where $ |\Delta^0
  \cap \ZZ^2|$ is the number of interior lattice points of
  $\Delta$. Equivalently, $\delta(C, h)$ is the number of
  parallelograms of the dual subdivision $\Delta_C$ if $(C, h)$ is simple.
\item
\label{itm:numberOfNodesRed}
Let $(C, h)$ be a nodal tropical curve with irreducible
  components $(C_1, h_1), \dots, (C_t, h_t)$ (i.e.,
  $C_i$ are the components of $C$ and $h_i$ are the
  restrictions of $h$ to $C_i$), of degrees $\Delta_1, \dots,
  \Delta_t$ and number of nodes $\delta_1, \dots, \delta_t$, respectively. (Note that the Minkowski sum $\Delta_1 +
  \cdots + \Delta_t$ equals $\Delta$.) The \emph{number of nodes} of
  $(C, h)$ is
\[
\delta(C, h) = \sum_{i=1}^t \delta_i + \sum_{i < j} \M(\Delta_i, \Delta_j),
\]
where $\M(\Delta_i, \Delta_j) := \tfrac{1}{2}(\Area(\Delta_i+\Delta_j)
- \Area(\Delta_i) - \Area(\Delta_j))$ is the \emph{mixed area} of $\Delta_i$
and $\Delta_j$. Here, $\Area(-)$ is the \emph{normalized} area, given by twice the
Euclidian area in $\RR^2$.

 Equivalently, $\delta(C, h)$ is the number of
  parallelograms of the dual subdivision $\Delta_C$ if $(C, h)$ is simple.
\end{enumerate}
\end{definition}

\vspace{0.6em}
\noindent {\bf Example~\ref{ex:tropicalCurve} (cont'd).}
The tropical curve of Example~\ref{ex:tropicalCurve} has genus $1$ as it is the image
of a trivalent genus $1$ graph. It is not the union of two tropical
curves and thus irreducible. Its number of nodes is, thus, equal to $|\Delta^0
\cap \ZZ^2| - g = 3 - 1 = 2$. The two tropical nodes are ``visible''
as the pair of edges crossing transversely as well as the edge of
weight $2$. (In general, a transverse intersection of two edges $e$
and $e'$ contributes $|u(V,e) \wedge u(V',e')|$ to $\delta(C)$, for any
adjacent vertices $V$ and $V'$, while an edge of multiplicity $m$ contributes $m-1$
to $\delta(C)$.)
\vspace{0.6em}

Definition~\ref{def:tropicalCurveProperties} (\ref{itm:numberOfNodesIrred}
is motivated by the classical degree-genus formula. 
In Definition~\ref{def:tropicalCurveProperties} (\ref{itm:numberOfNodesRed}),
the formula for $\delta(C, h)$ is chosen according to
Bernstein's theorem~\cite{Be75}, so that
Theorem~\ref{thm:MikhalkinCorrespondenceTheorem} holds.

\begin{figure}[htbp]
\vspace{-2em}
 \begin{tabular}{cc}
\includegraphics[scale = 0.25]{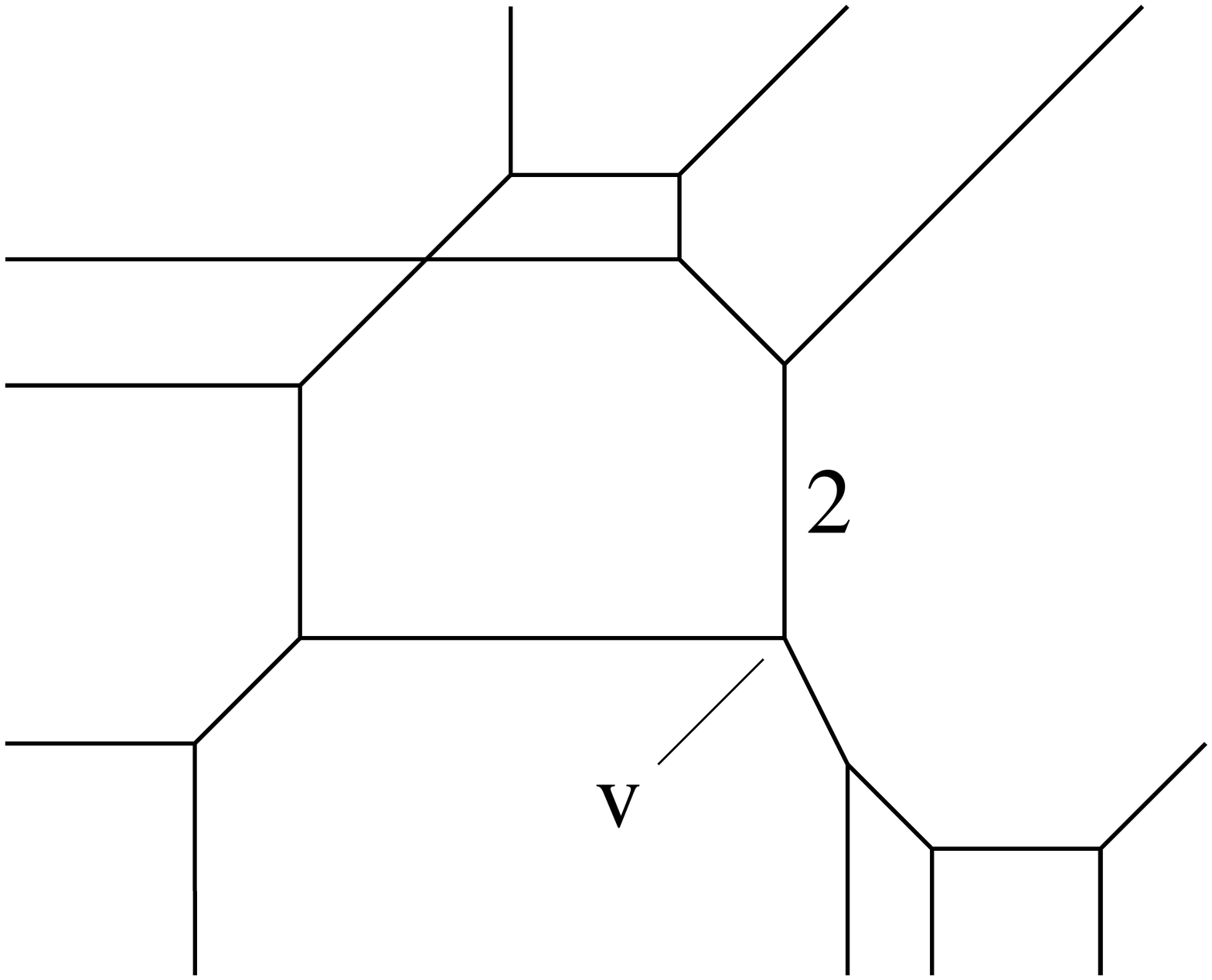}  \quad \quad \quad
&
\quad \quad \quad
\raisebox{25pt}{
\includegraphics[scale = 0.20]{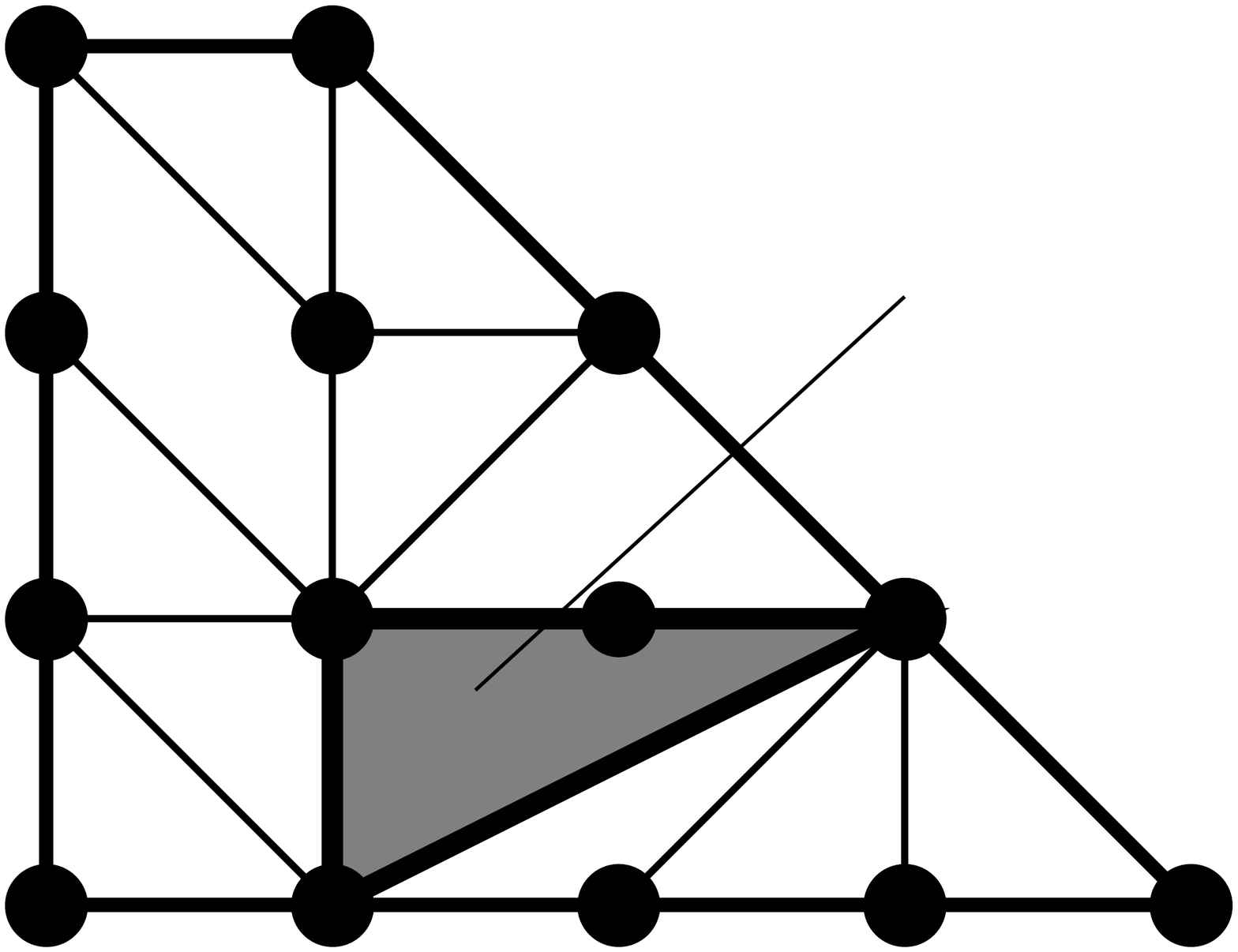}}
\begin{picture}(0,0)
\put(-35,120){$\Delta_v$}
\end{picture} \\
 \end{tabular}
\vspace{-2.5em}
\caption{The dual subdivision of the curve of
  Figure~\ref{ex:tropicalCurve}. The triangle $\Delta_v$ is dual to
  the vertex $v$.}
\label{fig:dualSubdivision}
\end{figure}

In~\cite{Mi05}, Mikhalkin assigns to a $3$-valent vertex $v$ of
a simple tropical curve
$(C, h)$ the \emph{(Mikhalkin) vertex multiplicity}
\begin{equation}
\label{eqn:MikhalkinvertexMultiplicity}
\mult_\CC(v) = \Area(\Delta_v).
\end{equation}
To the tropical curve $(C, h)$, he assigns the \emph{(Mikhalkin)
  multiplicity} 
\begin{equation}
\label{eqn:MikhalkinMultiplicity}
\mult_\CC(C, h) = \prod_v \mult_\CC(v) = \prod_v  \Area(\Delta_v),
\end{equation}
the product running over the $3$-valent vertices $v$ of $(C, h)$ and
$\Delta_v$ is the triangle in the subdivision $\Delta_C$ dual to
$v$ (cf.,
Definition~\ref{def:tropicalCurveProperties} and Figure~\ref{fig:dualSubdivision}).
If $v$ has adjacent edges $e_1$,$e_2$, and $e_3$, then the vertex
multiplicity
$\mult_\CC(v)$ equals the Euclidian area of the parallelogram
spanned by any two of the direction vectors starting at $v$.

\vspace{0.6em}
\noindent {\bf Example~\ref{ex:tropicalCurve} (cont'd).}
The dual subdivision of the tropical curve of
Example~\ref{ex:tropicalCurve} consists of $2$ triangles of
(normalized) area $2$ and $9$ triangles of area $1$. The Mikhalkin
multiplicity is thus $\mult_\CC(C) = 2^2 \cdot 1^9 = 4$.
(The quadrangle does not contribute to $\mult_\CC(C)$.)
\vspace{0.6em}

We associate to a tropical curve $(C, h)$ a refined
weight. Recall that, for an integer $n$, we denote by
\[
 [n]_y=\frac{y^{n/2}-y^{-n/2}}{y^{1/2}-y^{-1/2}}=y^{(n-1)/2}+\cdots +y^{-(n-1)/2}
\]
the \emph{quantum number} of $n$. In
particular, $[n]_1 = n$. We can think about $[n]_y$ as a (shifted)
$q$-analog of $n$.

\begin{definition}
\label{def:refinedMultiplicityTropicalCurve}
The \emph{refined vertex multiplicity} of a $3$-valent vertex $v$ of a simple
tropical curve $(C, h)$ is
\begin{equation}
\label{eqn:refinedMultiplicityVertex}
\mult(v;y) = [\Area(\Delta_v)]_y.
\end{equation}
The \emph{refined multiplicity} of a simple
tropical curve $(C, h)$ is
\begin{equation}
\label{eqn:refinedMultiplicity}
\mult(C, h;y) = \prod_v  [\Area(\Delta_v)]_y,
\end{equation}
the product running over the $3$-valent vertices of $(C, h)$.
\end{definition}

\vspace{0.6em}
\noindent {\bf Example~\ref{ex:tropicalCurve} (cont'd).}
The refined multiplicity of vertex $v$ of the tropical curve of
Example~\ref{ex:tropicalCurve} is $[\Area(\Delta_v)]_y = [2]_y =
y^{1/2} + y^{-1/2}$. As the dual subdivision consists of $2$ triangles
of area $2$ and $9$ triangles of area $1$, the refined multiplicity of
$(C, h)$ is
\[
\mult(C, h;y) = (y^{1/2} + y^{-1/2})^2 \cdot 1^9 = y+2 + y^{-1}.
\]
(Again, the quadrangle does not contribute.)
\vspace{0.6em}

We now define the tropical refinement of Severi degrees. For smooth
toric surfaces, these
invariants conjecturally agree with the refined invariants $\widetilde
N^{X(\Delta),L(\Delta),\delta}(y)$, provided $L(\Delta)$ is sufficiently ample,
see Conjecture~\ref{conj:smoothToric}.

As with classical curve counting, we require the configuration of
tropical points to be in \emph{tropically generic position}; the
precise definition is given in \cite[Definition~4.7]{Mi05}. Roughly,
tropically generic means there are no tropical curves of unexpectedly small degree passing through
the points.  By \cite[Proposition~4.11]{Mi05}, the
set of such points configurations is open and dense in the space of point
configurations in $\RR^2$. An important example of a tropically generic point configuration is
the following. The combinatorics of tropical curves passing through
such configurations is essentially given by the floor diagrams of Section~\ref{sec:FD}.

\begin{definition}[\cite{Brugalle_personal}]
\label{def:verticallyStretched}
Let $\Delta$ be a lattice polygon. A point configuration $\Pi=\{(x_1,
y_1), \dots, (x_N, y_N)\}$ in $\RR^2$ is called \emph{vertically
  stretched with respect to $\Delta$} if, for every tropical curve
$C$ of degree $\Delta$, we have
\begin{equation}
\label{eqn:vertically_stretched}
\begin{split}
\min_{i \neq j} |y_i-y_j| >& \max_{i \neq j} |x_i-x_j| \cdot |\text{maximal slope of
an edge of }C | \\
& \hspace{9em}\cdot \text{(number of edges of $C$)}.
\end{split}
\end{equation}
\end{definition}

The notion of a vertically stretched point configuration for a fixed
polygon $\Delta$ is well-defined, as (\ref{eqn:vertically_stretched})
depends only on $\Pi$ and the finitely many combinatorial types of
tropical curves of degree $\Delta$. Our definition of a vertically
stretched point configuration is slightly more restricted than in
\cite[Section~5]{BM2} but has the advantage of being explicit. It is
sufficient for the floor decomposition techniques of tropical
curves~\cite{Brugalle_personal}.

\begin{definition}
Fix a lattice polygon $\Delta$ and $\delta \ge 0$.
\begin{enumerate}
\item The \emph{(tropical) refined
  Severi degree} $N^{\Delta, \delta}(y)$ of the pair $(X(\Delta),
L(\Delta))$ is
\begin{equation}
\label{eqn:RefinedSeveriSum}
N^{\Delta, \delta}(y) := \sum_{(C, h)} \mult(C, h; y),
\end{equation}
where the sum is over all $\delta$-nodal tropical curves $(C, h)$
of degree $\Delta$
passing through $|\Delta \cap \ZZ^2| - 1 - \delta$ tropically generic points.
\item The \emph{(tropical)
irreducible refined Severi degree} of $(X(\Delta),
  L(\Delta))$ is
\begin{equation}
\label{eqn:IrreducibleRefinedSeveriSum}
N_{0}^{\Delta,\delta}(y) := \sum_{(C, h)} \mult(C, h; y),
\end{equation}
the sum ranging over all irreducible tropical curves of
degree $\Delta$ with $\delta$ nodes
passing through
$|\Delta \cap \ZZ^2| - 1 - \delta$ tropically generic points.
\end{enumerate}
\end{definition}

By Theorem~\ref{thm:irreducible_refined_Severi_degree}, the tropical
irreducible refined Severi degree agrees with its non-tropical version
defined in (\ref{eqn:irred_refined_Severi_definition}) for $\PP^2$,
Hirzebruch surfaces and rational ruled surfaces.
Note that a tropical curve through generic
points is, by definition, necessarily simple. Itenberg and Mikhalkin
showed that both refined tropical enumerations give indeed invariants.

\begin{theorem}[{\cite[Theorem~1]{IM12}}]
\label{thm:refinedSeveriIndep}
The sum
(\ref{eqn:IrreducibleRefinedSeveriSum}), and thus
$N_0^{\Delta,\delta}(y)$,  are independent of the tropical point
configuration, as long as the configuration is generic.
\end{theorem}

\begin{corollary}
The sum in (\ref{eqn:RefinedSeveriSum}), and thus $N^{\Delta,
  \delta}(y)$, are independent of the tropical point configuration, as
long as the configuration is generic.
\end{corollary}

\begin{proof}
The refined Severi degree can be expressed in terms of the irreducible refined
Severi degrees, which are, by
Theorem~\ref{thm:refinedSeveriIndep}, independent of the specific
location of the points.

Specifically, let $\Pi \subset \RR^2$ be a
tropically generic set of $|\Delta \cap \ZZ^2| - 1 - \delta$
points. Then (see also \cite[Section~2.3]{AB10})
\begin{equation}
\label{eqn:refined_in_terms_of_irreducible_refined}
N^{\Delta, \delta}(y) = \sum_{\Pi = \cup \Pi_i} \sum_{(\Delta_i,
  \delta_i)} \prod_i N^{\Delta_i, \delta_i}_0(y),
\end{equation}
where the first sum is over all partitions of $\Pi$, and the second
sum is over all pairs $(\Delta_i, \delta_i)$ which satisfy
\begin{equation}
\label{eqn:irreducible_refined_Severi_degree_conditions}
\begin{split}
|\Pi_i| =& |\Delta_i \cap \ZZ^2| - 1 - \delta_i, \quad \text{for all }
1 \le i \le t, \\
\Delta = &\Delta_1 + \cdots + \Delta_t  \quad \text{(Minkowski
  sum)},\\
\delta = & \sum_{i=1}^t \delta_i + \sum_{1 \le i < j \le t}
\M(\Delta_i, \Delta_j).
\end{split}
\end{equation}
Here, again $\M(\Delta_i, \Delta_j) = \tfrac{1}{2}(\Area(\Delta_i +
\Delta_j)-\Area(\Delta_i)-\Area(\Delta_j))$ is the mixed area of
the polygons $\Delta_i$ and $\Delta_j$.
\end{proof}

At $y = 1$, we recover Mikhalkin's (Complex) Correspondence Theorem.

\begin{theorem}[{Mikhalkin's (Complex) Correspondence Theorem
    \cite[Theorem~1]{Mi05}}]
\label{thm:MikhalkinCorrespondenceTheorem}
For any lattice polygon $\Delta$:
\begin{enumerate}
\item the (tropical) Severi degree $N^{\Delta, \delta}(1)$ equals the
  (classical) Severi degree $N^{\Delta, \delta}$, and
\item the (tropical) irreducible Severi degree $N_0^{\Delta,
    \delta}(1)$ equals the irreducible
  (classical) Severi degree $N_0^{\Delta, \delta}$.
\end{enumerate}
\end{theorem}

At $y = -1$, we recover Mikhalkin's Real Correspondence Theorem. The
\emph{classical Welschinger invariant} $W^{\Delta,
    \delta}(\Pi)$ and the \emph{irreducible classical Welschinger invariant} $W_0^{\Delta,
    \delta}(\Pi)$ count real curves resp.\ irreducible real curves of
  degree $\Delta$ with $\delta$ nodes
  through the real point configuration $\Pi$, counted with Welschinger
  sign. In positive genus, unlike for Severi degrees, both invariants depend on the point
  configuration $\Pi$, even for generic $\Pi$. For details see
  \cite[Section 7.3]{Mi05}.

\begin{theorem}[{Mikhalkin's Real Correspondence Theorem
    \cite[Theorem~6]{Mi05}}]
\label{thm:MikhalkinRealCorrespondenceTheorem}
For any lattice polygon $\Delta$:
\begin{enumerate}
\item the (tropical) Welschinger invariant $W_{\trop}^{\Delta,
    \delta}$ equals the (classical) Welschinger invariant $W^{\Delta,
    \delta}(\Pi)$ for some real point configuration $\Pi$, and
\item the irreducible (tropical) Welschinger invariant $W_{0, \trop}^{\Delta,
    \delta}$ equals the irreducible (classical) Welschinger invariant $W_0^{\Delta,
    \delta}(\Pi)$ for some real point configuration $\Pi$.
\end{enumerate}
\end{theorem}

\begin{remark}
The refined Severi degrees $N^{\Delta, \delta}(y)$ thus interpolate
between Severi degrees and Welschinger invariants. Similarly, the
refined irreducible Severi degrees $N_0^{\Delta, \delta}(y)$
interpolate between irreducible (classical) Severi degrees and
irreducible (classical) Welschinger invariants. 
\end{remark}

\section{Properties of refined Severi degrees}
\label{sec:properties}

In this section, we show a few properties of refined Severi
degrees. Specifically, we discuss the polynomiality of refined Severi
degrees in the parameters of $\Delta$ in
Section~\ref{sec:Refined_node_polynomials}, conjecture the polynomiality of their
coefficients (as Laurent polynomials in $y$) in
Section~\ref{sec:coefficient_polynomiality}, discuss
implications for the conjectures of G\"ottsche and Shende in
Section~\ref{sec:evidence_for_GS_conjectures}, and irreducible refined
Severi degrees in Section~\ref{sec:relation_irreducible_reducible}.

\subsection{Refined node polynomials}
\label{sec:Refined_node_polynomials}

We will now prove Conjecture~\ref{conj:smoothToric} for the projective
plane $\PP^2$ and $\delta
\le 10$, for $\PP^1\times \PP^1$ for $\delta\le 6$  and for all Hirzebruch surfaces $\F_m$ for $\delta \le 2$ 
and $\PP(1,1,m)$ for $\delta \le 2$.

First we state the existence of refined node polynomials
$N_\delta(d;y)$, $N_\delta(c,d,m;y)$, $N_\delta(d,m;y)$, refining some
results of \cite{FM} and \cite{AB10}. The proof of the following
theorem is in Section~\ref{sec:refinednodepolys}.

\begin{theorem}
\label{thm:refinedSeveridegreepoly}
For fixed $\delta \ge 1$:
\begin{enumerate}
\item ($\PP^2$) There is a polynomial $N_\delta(d; y) \in \QQ[y^{\pm
    1}][d]$  of degree $2\delta$ in $d$ such that, for $d \ge \delta$,
\[
N_\delta(d;y) = N^{d, \delta}(y).
\]

\item (Hirzebruch surface) There is a polynomial $N_\delta(c, d,m;y)
 \in \QQ[y^{\pm 1}][c,d,m]$ of degree $\delta$ in $c,m$ and $2\delta$
 in $d$ such that, for 
$c + m \ge 2 \delta$ and $d \ge \delta$
\[
N_\delta(c, d,m;y) = N^{(\F_m, cF+dH), \delta}(y).
\]

\item ($\PP(1, 1, m)$) There is a polynomial $N_\delta(d,m;y) \in
  \QQ[y^{\pm 1}][d,m]$ of degree $2\delta$ in $d$ and $\delta$ in $m$
  such that, for  $d \ge \delta$   and $m \ge 2 \delta$,
\[
N_\delta(d,m;y) = N^{\PP(1,1,m),dH), \delta}(y).
\]
\end{enumerate}
\end{theorem}

We call the polynomials $N_\delta(d; y)$,
$N_\delta(c, d,m;y)$, and
$N_\delta(d,m;y)$ \emph{refined node polynomials}. 

\begin{remark}
Theorem~\ref{thm:refinedSeveridegreepoly} generalizes to
toric surfaces from ``$h$-transverse'' polygons with bounds exactly as
in Theorems~1.2 and~1.3 in~\cite{AB10}. The argument of \cite{AB10}
generalizes to the refined setting by replacing all (Mikhalkin)
weights by refined weights. As the argument is long and technical, we
do not reproduce it here and restrain ourselves to more manageable
cases.
\end{remark}

\begin{theorem}
\label{thm:refinedInvEqualsRefinedSeveri}
\begin{enumerate}
\item ($\PP^2$) For $\delta \le 10$ and $d\ge \delta/2+1$ we have
\[
\tN^{d, \delta}(y) = N_\delta(d) = N^{d, \delta}(y).
\]

\item For $\delta\le 6$ and $c,d\ge \delta/2$,
we have 
$$\tN^{(\PP^1\times \PP^1,cF+dH), \delta}(y) =N_\delta(c,d,0;y)=N^{(\PP_1\times\PP_1,cF+dH), \delta}(y).$$
\item (Hirzebruch surfaces) 
For $\delta \le 2$ and $d\ge 1$, $c\ge \delta$ we have
\[
\tN^{(\F_m,cF+dH), \delta}(y) =N_\delta(c,d,m;y)=N^{(\F_m,cF+dH), \delta}(y).
\]
\item ($\PP(1,1,m)$) For $\delta \le 2$ and $d \ge 2$ and $m\ge 1$ we have
\[
N_\delta(d,m;y)=N^{(\PP(1,1,m),d), \delta}(y).
\]
and $N_\delta(d,m;y)$ is given by \conjref{P11mconjser} and \remref{corseries}.
\end{enumerate}
\end{theorem}

\begin{proof}
In \cite{GS12} we have computed $\widetilde N^{(S,L),\delta}(y)$ for all $(S,L)$ and all $\delta\le 10$. It is a polynomial of degree $\delta$ in the intersection numbers
$L^2$, $LK_S$, $K_S^2$ and $\chi(\O_S)$. 

(1) In the case $(S,L)=(\PP^2,\O(d))$ this gives that $\widetilde
N^{d,\delta}(y)$ as a polynomial of degree $2\delta$ in $d$.
Using the recursion \ref{refrec}
we compute $N^{d, \delta}(y)$ for all $\delta\le 10$ and all $d\le 30$. We find that $N^{d,\delta}(y)=\widetilde N^{d,\delta}(y)$ for $\delta\le 10$, and 
$\frac{\delta}{2}+1\le d\le 30$. 
We also know by \thmref{thm:refinedSeveridegreepoly} that $N_\delta(d)$ is a polynomial of degree $2\delta$ in $d$, and that $N_\delta(d)=N^{d,\delta}(y)$ for $d\ge \delta$.
Thus for $0\le \delta\le 10$ the two polynomials $N_\delta(d)$ and $\widetilde N^{d,\delta}(y)$ of degree $2\delta$ in $d$ have the same value for $\delta\le d\le 30$. 
Thus they are equal.

(2) Is very similar to (1). We compute $N^{(\PP_1\times\PP_1,cF+dH), \delta}(y)$ for $c\le 18$ and $d\le 12$ and $\delta\le 6$. 
We find that in this realm $N^{(\PP_1\times\PP_1,cF+dH), \delta}(y)=\tN^{(\PP^1\times \PP^1,cF+dH), \delta}(y)$ for $c,d\ge \delta/2$.
We know by \thmref{thm:refinedSeveridegreepoly}  and symmetry, that $N_\delta(c,d,0;y)$ is a polynomial of bidegree $(\delta,\delta)$ in 
$c,d$. Thus for $0\le \delta\le 6$, the two polynomials $N_\delta(c,d,0;y)$ and $\tN^{(\PP^1\times \PP^1,cF+dH), \delta}(y)$
have the same value, whenever $18\ge c\ge 2\delta$, $12\ge d\ge \delta$. Thus they are equal. 

(3) This case is again similar. We compute  $N^{(\F_m,cF+dH), \delta}(y)$ for $c\le 6$ and $d\le 6$, $m\le 4$ and $\delta\le 2$.
The claim follows in the same way as before.

(4) We compute $N^{(\PP(1,1,m)dH), \delta}(y)$ for  $d\le 6$, $m\le 6$ and $\delta\le 2$.
The claim follows in the same way as before.
\end{proof}

\begin{corollary}
The coefficients of the refined invariants $\tN^{(S, L), \delta}(y)$
are non-negative, i.e.,
\[
\tN^{(S, L), \delta}(y) \in \ZZ_{\ge 0}[y^{\pm 1}]
\]
provided either
\begin{itemize}
\item $S = \PP^2$, $L = dH$, $\delta \le 10$, and $d \ge
  \tfrac{d}{2} + 1$, or
\item $S=\PP^1\times\PP^1$, $L=cF+dH$, $\delta\le 6$, and $c,d\ge \delta/2$.
\item $S = \F_m$, $L = cF+dH$, $\delta \le 2$,  and $d\ge 1$, $c\ge \delta$.
\item $S = \PP(1,1,m)$, $L =dH$, $\delta \le 2$, and $d\ge 2$, $m\ge 1$.
\end{itemize}
\end{corollary}

\begin{proof}
For any lattice polygon, the refined Severi degree $N^{\Delta,\delta}(y)$
is a Laurent polynomial in $y$ with non-negative coefficients. The
corollary follows from Theorem~\ref{thm:refinedInvEqualsRefinedSeveri}. 
\end{proof}

\begin{conjecture}\label{posconj}
For any smooth projective surface $S$ and $\delta$-very ample line
bundle $L$ on $S$,
the refined invariants $\tN^{(S,L), \delta}(y)$ have non-negative
coefficients.
\end{conjecture}

We have the following evidence for this conjecture:
 In \cite{GS13} \conjref{Gconj} is proven for $S$ an abelian or K3 surface, and the positivity of $\tN^{(S,L), \delta}(y)$ follows for all 
 line bundles $L$ on $S$. 
 If $S$ is a toric surface and $L$ is $\delta$-very ample on $S$, then \conjref{posconj} is implied by  \conjref{conj:smoothToric}.
 Numerical computations give in all examples considered that \conjref{posconj} is true. 
 Comparing with \eqref{Gform1} numerical checks confirm that, in the realm checked, for $l>\delta$ all the coefficients of 
 $(\frac{t}{g(y,t)})^l$ of degree at most $\delta$ in $t$ are positive. If $L$ is $\delta$-very ample we expect 
 $\chi(L)>\delta$ and also $\chi(L)$ that is  large with respect to $K_S^2$ and $LK_S$. 
 Therefore we would expect 
that all coefficients of the left hand side of    \eqref{Gform1} of degree at most $\delta$ in $t$ are nonnegative.

\subsection{Coefficient polynomiality of refined Severi degrees}
\label{sec:coefficient_polynomiality}

The refined Severi degrees $N^{d, \delta}(y)$ of $\PP^2$,
as Laurent polynomial in $y$, have non-negative integral
coefficients. Furthermore, for fixed $\delta$, these coefficients behave
polynomially in $d$, for sufficiently large $d$, by
Theorem~\ref{thm:refinedSeveridegreepoly}. In this section, we conjecture
that particular coefficients of the refined Severi degree are
polynomial for $d$ \emph{independent of } $\delta$
(Conjecture~\ref{thm:coeff_poly_threshold}). We also give enumerative meaning to the
first leading coefficient
(Proposition~\ref{prop:leadingCoefficients}). For simplicity, we consider
only $\PP^2$ in this section.
Throughout this section, we fix the number of nodes $\delta \ge 1$.

\begin{notation}

We denote the coefficients of the refined Severi degree by
\[
N^{d, \delta}(y) = p^\delta_{d,0} \cdot y^\delta + p^\delta_{d,1} \cdot y^{\delta-1}
+p^\delta_{d,2} \cdot
y^{\delta-2} + \cdots + p^\delta_{d,\delta} \cdot y^{0} + \cdots + p^\delta_{d,0}
\cdot  y^{-\delta}
\]
for $p^\delta_{d,0}, p^\delta_{d,1}, \dots, p^\delta_{d,\delta} \in
\ZZ_{\ge  0}$.

Similarly, we write the coefficients of the refined node
  polynomial as
\[
N_\delta(d; y) = p^\delta_{0}(d) \cdot y^\delta + p^\delta_{1}(d) \cdot y^{\delta-1} + p^\delta_{2}(d) \cdot
y^{\delta-2} + \cdots + p^\delta_{\delta}(d) \cdot y^{0} + \cdots + p^\delta_{0}(d) \cdot y^{-\delta}
\]
for polynomials $p^\delta_{0}(d), p^\delta_{1}(d), \dots, p^\delta_{\delta}(d) \in
\ZZ[d]$.
\end{notation}

From Theorem~\ref{thm:refinedSeveridegreepoly}, the following is
immediate.

\begin{corollary}
For $0 \le i \le \delta$, we have $p^\delta_i(d) = p^\delta_{d,i}$, whenever $d \ge \delta$.
\end{corollary}

Conjecturally, we have the lower bound $d \ge \tfrac{\delta}{2} + 1$
(cf., Conjecture~\ref{ref-sev}), which still depends on $\delta$. We
conjecture that for the leading coefficients of the refined Severi
degree, this dependence disappears.

\begin{conjecture}
\label{thm:coeff_poly_threshold}
For $0 \le i \le \delta$, we have $p^\delta_i(d) = p^\delta_{d,i}$, whenever $d \ge
i + 2$.
\end{conjecture}

In other words, the larger the order of the coefficients of the
refined Severi degree, the sooner the polynomiality kicks in.
This conjecture was predicted as part of \cite[Conj.~89]{GS12}, where in addition a formula for the coefficients  $p^\delta_i(d)$
was conjectured. \propref{prop:leadingCoefficients} below gives a new proof for $i=0$.

\begin{remark}\label{ppolconj}
\begin{enumerate}
\item \corref{thm:coeff_poly_threshold} is part of  \cite[Conj.~89(1)]{GS12}.
\item More precisely this conjecture says that $p_{i}^\delta(d)$ is a polynomial of degree $2\delta$ in $d$, which is divisible by 
$\binom{\binom{d-1}{2}-3i}{\delta-i}$. Moreover \cite[Conj.~86,~Conj.~87]{GS12}  give a conjectural formula for the quotient
$p_{i}^\delta(d)/\binom{\binom{d-1}{2}-3i}{\delta-i}$ in terms of the $\widetilde N^{d,\delta}(y)$ with $\delta\le 3i$. 
Thus,  assuming these conjectures, \thmref{thm:refinedInvEqualsRefinedSeveri}
gives a formula for $p^\delta_{i}(d)$ for $i\le 3$. 

\item Computational evidence suggests that for $d\ge 2$ the bound in
\corref{thm:coeff_poly_threshold} is optimal: $p_i^\delta(d)=p_{d,i}^\delta$, if and only if $d\ge i+2$. We checked this
for $d\le 14$, $\delta\le 11$.  
\end{enumerate}

\end{remark}

We give a
formula for leading coefficient of the refined Severi degree.
This result was also obtained
in \cite[Proposition~83]{GS12} and \cite[Proposition 2.11]{IM12}. 

\begin{proposition}
\label{prop:leadingCoefficients}
The leading coefficients of $N^{d, \delta}(y)$ is given by
\[
p^\delta_{d,0} = \binom{\binom{d-1}{2}}{\delta} \quad \text{ for }d \ge 1.
\]
\end{proposition}
The formula could be interpreted as the number of ways to choose $\delta$ of the 
$\binom{d-1}{2}$ nodes of a genus $0$ nodal curve $C$ of degree $d$, i.e. as the number
of $\delta$-nodal curves obtained as partial resolutions of $C$.

We prove this proposition in Section~\ref{sec:refinednodepolys}.

The same formulas hold for the coefficients of the irreducible refined Severi degrees $N^{d,\delta}_0(y)$.
Again we can write $N^{d,\delta}_0(y)=p_{d,0}^{\delta,0} y^{\delta}+p_{d,1}^{\delta,0} y^{\delta-1}+\ldots+p_{d,1}^{\delta,0} y^{-\delta+1}+p_{d,0}^{\delta,0} y^{-\delta}$.
Assuming Conjecture~\ref{thm:coeff_poly_threshold}, a similar result also holds for the $p_{d,i}^{\delta,0}$, because of the following lemma.

\begin{lemma} Assuming Conjecture~\ref{thm:coeff_poly_threshold}, we
  have $p^{\delta,0}_{d,i}=p^\delta_{d,i}$ if $d\ge i+2$.
\end{lemma}
\begin{proof} 
If we specialize  the formula
\eqref{eqn:refined_in_terms_of_irreducible_refined}
to $N^{d,\delta}(y)$, we express 
$N^{d, \delta}(y)-N^{d,\delta}_0(y)$ as a sum of products 
$ \prod_{i=1}^t N^{d_i, \delta_i}_0(y)$, with $t\ge 2$, 
$d=d_1+\ldots + d_t$ 
and 
$$\delta=\sum_{i=1}^t \delta_i+\frac{1}{2}\sum_{1\le i< j\le t} \big((d_i+d_j)^2-d_i^2-d_j^2\big).$$
It is an easy exercise to see that for given $d$ the rightmost sum is minimal if $t=2$ and $\{d_1,d_2\}=\{1,d-1\}$, and the corresponding sum  is $d-1$. Thus in all summands for $N^{d, \delta}(y)-N^{d, \delta}_0(y)$ we have 
$\delta-\sum_{i}\delta_i\ge d-1$. As the $N^{d_i,\delta_i}_0(y)$ have degree at most $\delta_i$ in $y,y^{-1}$, we see that
$p^\delta_{d,i}=p^{\delta,0}_{d,i}$ for $i<d-1$.
\end{proof}

The argument also shows that $N^{d,\delta}(y)=N^{d,\delta}_0(y)$ if $\delta\le d-2$.  Thus we obtain the following corollary

\begin{corollary}
$N^{d,\delta}_0(y)=N_\delta(d;y)$ for $\delta\le d-2$.
\end{corollary}

\subsection{Numerical evidence for G\"ottsche and Shende's conjectures}
\label{sec:evidence_for_GS_conjectures}

\thmref{thm:refinedSeveridegreepoly} and \thmref{thm:refinedInvEqualsRefinedSeveri} provide strong evidence for 
\conjref{Gconj}, \conjref{ref-sev}: On $\PP^2$ and rational ruled surfaces, for $L$ sufficiently ample with respect to $\delta$,
$N^{(S,L),\delta}(y)$ is indeed given by a node polynomial in $L^2$, $LK_S$, $K_S^2$ and $\chi(\O_S)$.  Furthermore, if $\delta$ is not too large, we show that
this polynomial coincides with $\tN^{(S,L),\delta}(y)$. Unfortunately in the case of rational ruled surfaces we only prove this for $\delta\le 2$.
There is however  more and  stronger numerical evidence, even if it does not lead to a proof of formulas for higher $\delta$.
Below we list briefly some of this evidence.

\begin{enumerate}
\item In \cite{GS12} the $N^{d,\delta}(y)$ have been computed for $d\le 17$ and $\delta\le 32$.  Assuming \conjref{Gconj},\conjref{ref-sev} this determines the 
power series $B_1(y,q)$ and $B_2(y,q)$ modulo $q^{29}$, and thus all the refined invariants $\tN^{(S,L),\delta}(y)$ as polynomials in $L^2$, $LK_S$, $K_S^2$, $\chi(\O_S)$  for all $S$, $L$ and all $\delta\le 28$.
Denote for the moment $\widehat N^{(S,L),\delta}(y)$ the refined invariants obtained this way (and $\widehat N^{d,\delta}(y) $ the corresponding invariants of $\PP^2$.
For $\delta\le 10$ (where the $\tN^{(S,L),\delta}(y)$ have been  computed in \cite{GS12}) $\widehat N^{(S,L),\delta}(y)=\tN^{(S,L),\delta}(y)$.

The computation mentioned above gives $N^{d,\delta}(y)=\widehat N^{d,\delta}(y) $ for $d\le 17$ and $\delta\le min(2d-2,28)$.
\item We have also computed the $N^{d,\delta}(y)$ for $d\le 20$, $\delta\le 20$, again within this realm $N^{d,\delta}(y)=\widehat N^{d,\delta}(y) $ for $\delta\le 2d-2$.

\item We computed $N^{\PP^1\times\PP^1,cF+dH),\delta}$ for arbitrary $\delta$ and $c,d\le 8$. We find in this realm 
$N^{(\PP^1\times\PP^1,cF+dH),\delta}=\widetilde N^{(\PP^1\times\PP^1,cF+dH),\delta}$ for $\delta\le min(2c,2d)$. 
\item We computed $N^{(\Sigma_m,cF+dH),\delta}(y)$ for $m\le 10$, $\delta\le 10$, $d\le 6$, $c\le 10$.
We find in this realm
$N^{(\Sigma_m,cF+dH,\delta}(y)=\tN^{(\Sigma_m,cF+dH,\delta}(y)$ if $\delta\le min(2d,c)$.
\end{enumerate}

\subsection{On the relation with irreducible refined Severi degrees}
\label{sec:relation_irreducible_reducible}

We show that the irreducible refined Severi degree, formally defined
in (\ref{eqn:irred_refined_Severi_definition}) for $\PP^2$, Hirzebruch
surfaces and rational ruled surfaces, agrees with the
refined enumeration of irreducible tropical curves. It therefore follows that
also the irreducible refined Severi degree has non-negative coefficients.

\begin{theorem}
\label{thm:irreducible_refined_Severi_degree}
The tropical irreducible refined Severi degree $N_0^{\Delta,
  \delta}(y)$ agrees with the irreducible refined Severi degree
defined in (\ref{eqn:irred_refined_Severi_definition}).
\end{theorem}

The refined multiplicity of an irreducible tropical curve by
definition has non-negative integer coefficients in $y^{\pm
  1}$. Therefore, we have shown the following.

\begin{corollary}
$N_0^{\Delta, \delta}(y)$ has non-negative integer coefficients.
\end{corollary}

\begin{proof}[Proof of
  Theorem~\ref{thm:irreducible_refined_Severi_degree}.]
Recall the relation
\eqref{eqn:refined_in_terms_of_irreducible_refined}) between refined
Severi degrees and their tropical irreducible analog
\begin{equation}\label{eqn:irred_red_rel}
N^{\Delta, \delta}(y) = \sum_{\Pi = \cup \Pi_i} \sum_{(\Delta_i,
  \delta_i)} \prod_i N^{\Delta_i, \delta_i}_0(y),
\end{equation}
where the first sum is over all partitions of $\Pi$, and the second
sum is over all pairs $(\Delta_i, \delta_i)$ which satisfy (cf. 
\eqref{eqn:irreducible_refined_Severi_degree_conditions})
\begin{equation}
\label{eqn:irreducibl_reducible_relation}
\begin{split}
|\Pi_i| =& |\Delta_i \cap \ZZ^2| - 1 - \delta_i, \quad \text{for all }
1 \le i \le t, \\
\Delta = &\Delta_1 + \cdots + \Delta_t  \quad \text{(Minkowski
  sum)},\\
\delta = & \sum_{i=1}^t \delta_i + \sum_{1 \le i < j \le t}
\M(\Delta_i, \Delta_j).
\end{split}
\end{equation}
Here, again $\M(\Delta_i, \Delta_j) = \tfrac{1}{2}(\Area(\Delta_i +
\Delta_j)-\Area(\Delta_i)-\Area(\Delta_j))$ is the mixed area of
the polygons $\Delta_i$ and $\Delta_j$.

Any collection of lattice
polygons $\Delta_1, \Delta_2, \dots, \Delta_t, \Delta$ and non-negative
integers $\delta_1, \dots, \delta_t, \delta$ satisfying the second and
third condition of (\ref{eqn:irreducibl_reducible_relation}) also satisfy
\[
\sum_{i=1}^t (\dim \Delta_i - \delta_i) = \dim \Delta
- \delta,
\]
where we write $\dim \Delta = |\Delta \cap \ZZ^2| - 1$. Indeed, both
sides equal the number of point conditions of a tropical curve of
degree $\Delta$ with $\delta$ nodes which has
irreducible components of degrees $\Delta_i$ with $\delta_i$ nodes, respectively.
Furthermore, we have $\mult(C; y) = \prod_{i=1}^t
\mult(C_i; y)$. 

The exponential generating functions of the refined Severi degrees $N^{\Delta, \delta}(y)$ and the tropical
irreducible refined Severi degree $N_{0, \trop}^{\Delta, \delta}(y)$
thus satisfy
\begin{equation}
\label{eqn:irred_tropical}
\exp \left( \sum_{\Delta,\delta} \frac{z^{\dim \Delta -\delta}}{(\dim
    \Delta -\delta)!} v^\Delta  N_{0,\trop}^{\Delta,\delta}(y) \right)=
1+\sum_{\Delta,\delta}\frac{z^{\dim \Delta-\delta}}{(\dim \Delta-\delta)!} v^\Delta   N^{\Delta,\delta}(y),
\end{equation}
where we define $v^\Delta \cdot v^{\Delta'} := v^{\Delta + \Delta'}$
for lattice polygons $\Delta$ and $\Delta'$ and both sums are over all
lattice polygons $\Delta$ (up to translation) and $\delta \ge 0$. 
Comparing (\ref{eqn:irred_tropical}) and
(\ref{eqn:irred_refined_Severi_definition}), the result follows.
\end{proof}

\section{$y$-Weighted Floor Diagrams and Templates}
\label{sec:FD}

Floor diagrams are purely combinatorial representations of tropical
curves. They exist for all ``$h$-transverse'' polygons $\Delta$. We
focus mostly on the cases $S = \PP^2$, $\Sigma_m$, and $\PP(1,1,m)$,
all whose moment polygons are $h$-transverse. More specifically, if
we consider tropical curves through a vertically stretched
point configuration (see Definition~\ref{def:verticallyStretched}) the
tropical curves are uniquely encoded by a ``marking'' of a floor
diagram and, vice versa, every marked floor diagram corresponds to a
tropical curve. This gives a purely combinatorial way to compute
refined Severi degrees for toric surfaces with $h$-transverse
polygons. Floor diagrams were invented (in the unrefined setting) by
Brugall\'e and Mikhalkin~\cite{BM1, BM2}.

\subsection{Floor Diagrams}
\label{sec:floor_diagrams}

We now briefly review the marked floor diagrams of Brugall\'e and
Mikhalkin~\cite{BM1, BM2} for surfaces $S=\PP^2$, $S=\PP(1,1,m)$, and $S = \F_m$,
with some emphasis on the $\PP^2$ case. We present them in the notation
of Ardila and Block~\cite{AB10}, following
Fomin and Mikhalkin~\cite{FM}. In each case, we fix a polygon
$\Delta$ (cf.\ Figure~\ref{fig:2polygons}):
\begin{itemize}
\item ($\PP^2$ case) $\Delta = \conv((0,0), (0,d), (d, 0))$, for $d \ge 1$, or
\item ($\F_m$ case) $\Delta = \conv((0,0), (0,d), (c, d), (c
  + md, 0))$, for $c,d,m \ge 1$, or
\item ($\PP(1,1,m)$ case) $\Delta = \conv((0,0), (0,d), (d m, 0))$, for
  $d, m \ge 1$. In this case, set $c = 0$.
\end{itemize}

\begin{figure}[htbp]
\vspace{-3.5em}
 \begin{tabular}{cc}
\includegraphics[scale=0.3]{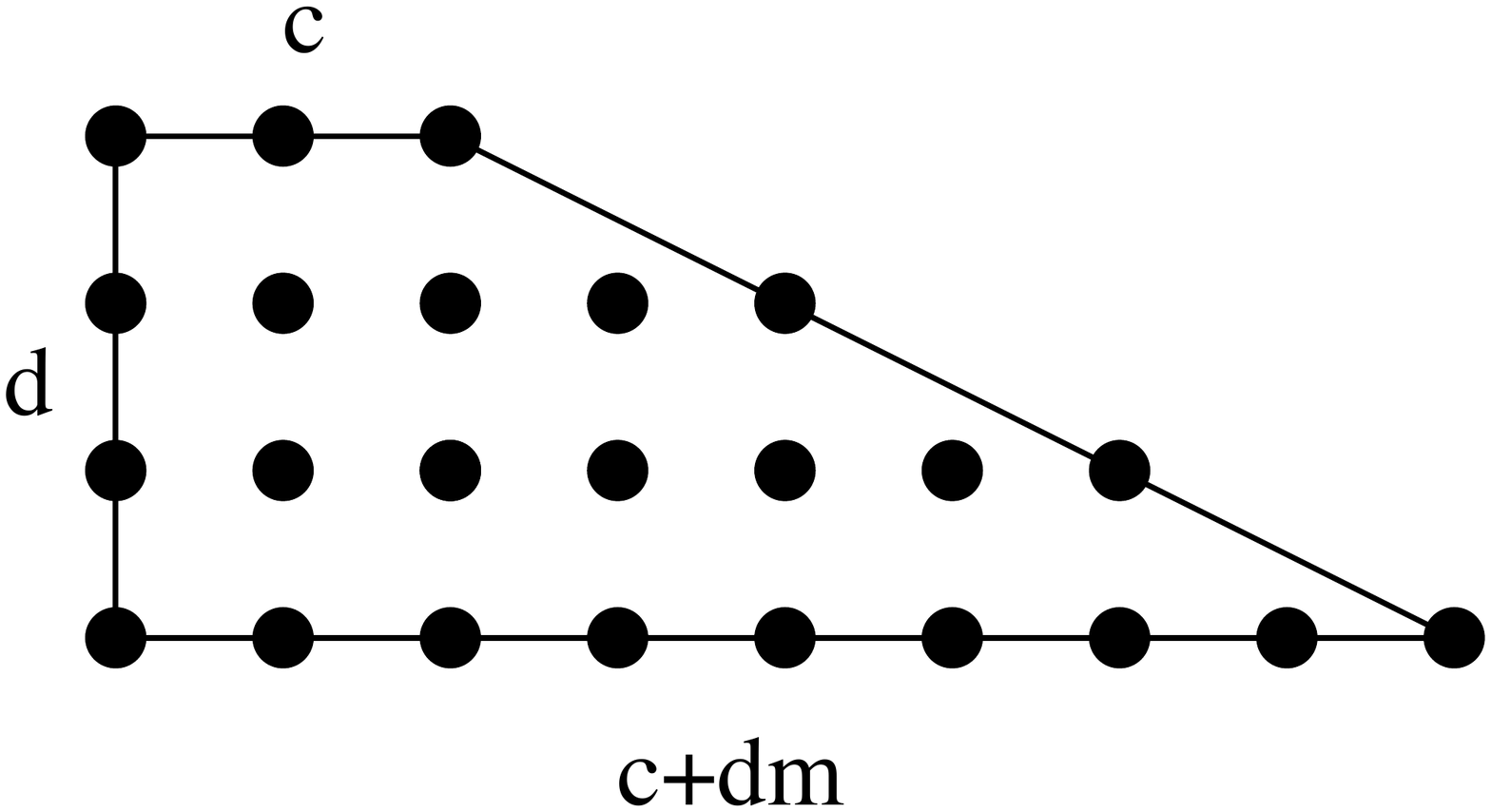}  \quad \quad \quad
&
\quad \quad \quad
\raisebox{15pt}{
\includegraphics[scale=0.25]{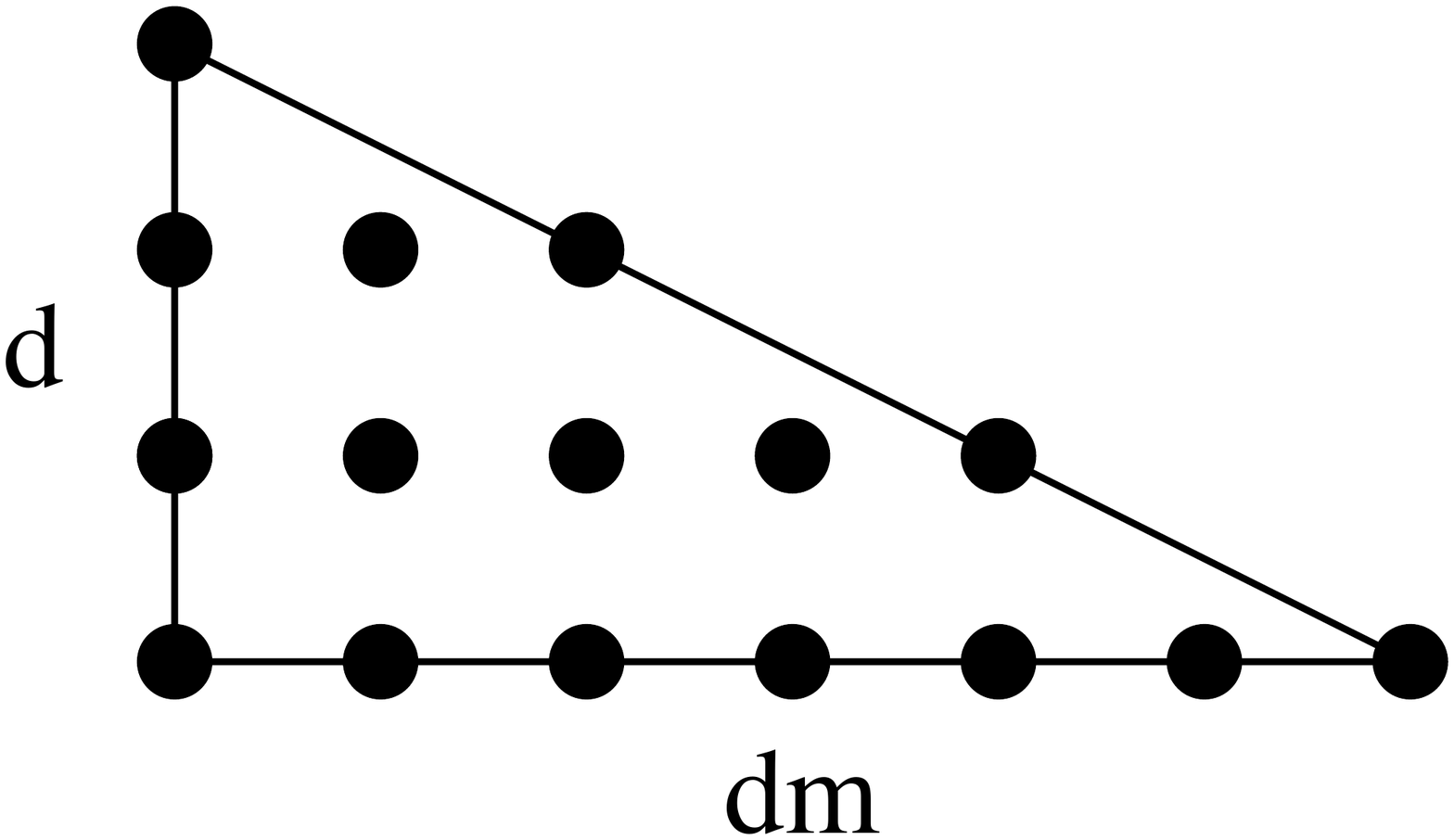}}
\vspace{-4.5em}\\
 \end{tabular}
\caption{Lattice polygons of the Hirzebruch surface
  $\Sigma_m$ with line bundle $L =dH + cF$ (left) and $\PP(1,1,m)$ with
  $L = dH$ (right). In both cases $m=2$. For $S = \PP^2$, we set $c =
  0$ and $m = 1$.}
\label{fig:2polygons}
\end{figure}

\begin{definition}
A \emph{$\Delta$-floor diagram} $\D$ consists of:
\begin{enumerate}
\item A graph on a vertex set $\{1, \dots, d\}$, possibly with
  multiple edges, with edges directed $i \to j$ if $i < j$.
\item A sequence $(s_1, \dots, s_d)$ of non-negative integers such
  that $s_1 + \cdots + s_d =c$. (If $S = \PP(1,1,m)$ then all $s_i$
  equal $0$.)
\item (Divergence Condition) For each vertex $j$ of $\D$, we have 
\[
\text{div}(j) \stackrel{\text{def}}{=} \sum_{ \tiny
     \begin{array}{c}
  \text{edges }e\\
j \stackrel{e}{\to} k
     \end{array}
} \wt(e) -   \sum_{ \tiny
     \begin{array}{c}
  \text{edges }e\\
i \stackrel{e}{\to} j
     \end{array}
} \wt(e)\le m + s_j.
\]
\end{enumerate}
The last condition says that at every vertex of $\D$ the total weight of the outgoing edges
is larger by at most $m + s_j$ than the total weight of the incoming edges.
\end{definition}

We loosely think of $\Delta$ as the \emph{degree} of the floor
diagram $\D$. If $S = \PP^2$, we say that $\D$ is of degree $d$.
A floor diagram is \emph{connected} if its underlying graph is. If $\D$ is
connected its \emph{genus} 
is
the genus of the underlying graph. A connected floor diagram $\D$ of
degree $\Delta$ and genus $g$ has \emph{cogenus} $\delta(\D)$ equal to
the number of interior lattice points in $\Delta$ minus $g$. 

If $\D$ is not connected, there are lattice polygons $\Delta_1,
\Delta_2, \dots$ such that their Minkowski sum equals
$\Delta_1+\Delta_2+\cdots = \Delta$ and the $\Delta_i$ are the degrees
of the connected components of $\D$. Let $\delta_1, \delta_2, \dots$ be the
cogenera of the connected components.
Similarly to the case of tropical curves, we define the \emph{cogenus}
\[
\delta(\D) = \sum_{i} \delta_i + \sum_{i < j} \M(\Delta_i, \Delta_j),
\]
where again $\M(\Delta_i, \Delta_j) := \tfrac{1}{2}(\Area(\Delta_i+\Delta_j)
- \Area(\Delta_i) - \Area(\Delta_j))$ is the \emph{mixed area} of $\Delta_i$
and $\Delta_j$. As before, $\Area(-)$ is the normalized area,
given by twice the Euclidian area in $\RR^2$.

The refined multiplicity of tropical curves (see
Definition~\ref{def:refinedMultiplicityTropicalCurve}) translates to floor
diagram as follows, yielding a purely combinatorial formula for the
refined Severi degrees for $\F_m$ and $\PP(1,1,m)$ in Definition~\ref{def:combrefined}.

\begin{definition}
\label{def:refined_multiplicity_FD}
We define the \emph{refined multiplicity} $\mult(\D, y)$ of a floor
diagram $\D$ as
\[
\mult(\D, y) = \prod_{\text{edges }e} \left( [\wt(e)]_y\right)^2. 
\]
\end{definition}
Notice that the weight $\mult(\D, y)$ is a Laurent polynomial in $y$ with positive integral
coefficients. We draw floor diagrams using the convention that vertices in
increasing order are arranged left to right. Edge weights of $1$
are omitted.

\begin{example}
\label{ex:floordiagram}
An example of a floor diagram for $\PP^2$ of degree $d = 4$, genus $ g=1$,
 cogenus $\delta = 2$, divergences $1,1,0,-2$, and multiplicity
 $\mult(\D; y) = (y^{-1/2} + y^{1/2})^2 = y^{-1} + 2 + y$ is drawn
 below.

\begin{picture}(50,35)(-150,-20)\setlength{\unitlength}{4pt}\thicklines
\oooo\Eeee\eEee\eeOe
\put(15,1.5){\makebox(0,0){$2$}}
\put(7,0){\vector(1,0){1}} 
\put(17,0){\vector(1,0){1}}
\put(27.5,1.75){\vector(2,-1){1}}
\put(27.5,-1.75){\vector(2,1){1}}
\end{picture}
\end{example}

To a floor diagram we associate a last statistic, as in
\cite[Section~1]{FM}. Notice that this statistic is independent of $y$.

\begin{definition}
\label{def:marking}
A \emph{marking} of a floor diagram $\D$ is defined by the following
four step process

{\bf Step 1:} 
For each vertex $j$ of $\D$ create $s_j$ new indistinguishable vertices and
connect them to $j$ with new edges directed towards $j$.

{\bf Step 2:} For each vertex $j$ of $\D$ create $m+s_j- div(j)$ new
indistinguishable vertices and connect them to $j$ with new edges
directed away from $j$. This makes the divergence of vertex $j$ equal
to $m$.

{\bf Step 3:} Subdivide each edge of the original floor
diagram $\D$ into two
directed edges by introducing a new
vertex for each edge. The new edges inherit their weights and orientations. Denote the
resulting graph $\widetilde{\D}$.

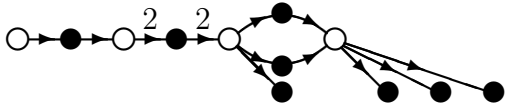
\begin{figure}[htbp]
\begin{center}
\begin{picture}(50,34)(45,-20)\setlength{\unitlength}{4pt}\thicklines
\oooo\Eeee\eEee\eeOe
\put(12.5,2){\makebox(0,0){$2$}}
\put(17.5,2){\makebox(0,0){$2$}}
\put(2.5,0){\vector(1,0){1}}
\put(7.5,0){\vector(1,0){1}}
\put(12.5,0){\vector(1,0){1}}
\put(17.5,0){\vector(1,0){1}}
\put(27.5,1.75){\vector(2,-1){1}}
\put(27.5,-1.75){\vector(2,1){1}}
\put(22.5,1.75){\vector(2,1){1}}
\put(22.5,-1.75){\vector(2,-1){1}}

\put(5,0){\circle*{2}}
\put(15,0){\circle*{2}}
\put(25,2.5){\circle*{2}}
\put(25,-2.5){\circle*{2}}
\put(20.7,-0.7){\line(1,-1){4}}
\put(20.7,-0.7){\vector(1,-1){3}}
\put(30.6,-0.8){\line(1,-1){4}}
\put(30.8,-0.6){\line(2,-1){9}}
\put(30.9,-0.4){\line(3,-1){14}}
\put(30.6,-0.8){\vector(1,-1){2.5}}
\put(30.8,-0.6){\vector(2,-1){5}}
\put(30.9,-0.4){\vector(3,-1){7.5}}
\put(25,-5){\circle*{2}}
\put(35,-5){\circle*{2}}
\put(40,-5){\circle*{2}}
\put(45,-5){\circle*{2}}
\end{picture}
\end{center}
\caption{The result of applying Steps 1-3 to the floor diagram of Example~\ref{ex:floordiagram}}
\label{fig:markingAfterStep3}
\end{figure}

{\bf Step 4:} Linearly order the vertices of $\widetilde{\D}$
extending the order of the vertices of the original floor
diagram $\D$ such that, as before, each edge is directed from a
smaller vertex to a larger vertex.

The extended graph $\widetilde{\D}$ together with the linear order on
its vertices is called a \emph{marked floor diagram}, or a
\emph{marking} of the original floor diagram $\D$.

\begin{figure}[htbp]
\begin{center}
\begin{picture}(50,40)(65,-15)\setlength{\unitlength}{4pt}\thicklines
\put(12.5,2){\makebox(0,0){$2$}}
\put(17.5,2){\makebox(0,0){$2$}}
\multiput(0,0)(10,0){3}{\circle{2}}
\multiput(40,0)(10,0){1}{\circle{2}}
\multiput(5,0)(10,0){5}{\circle*{2}}
\multiput(30,0)(10,0){1}{\circle*{2}}
\multiput(50,0)(5,0){2}{\circle*{2}}
\put(1,0){\line(1,0){8}}
\put(11,0){\line(1,0){8}}
\put(2.5,0){\vector(1,0){1}}
\put(7.5,0){\vector(1,0){1}}
\put(12.5,0){\vector(1,0){1}}
\put(17.5,0){\vector(1,0){1}}
\put(21,0){\line(1,0){3}}
\put(22.5,0){\vector(1,0){1}}
\qbezier(20.6,0.6)(22,6)(25,6)\qbezier(25,6)(28,6)(29.4,0.6)
\put(25,6){\vector(1,0){1}}
\qbezier(20.8,-0.6)(23.75,-3)(27.5,-3)\qbezier(27.5,-3)(32.25,-3)(34.2,-0.6)
\put(27.5,-3){\vector(1,0){1}}
\qbezier(25.8,0.6)(28.75,3)(32.5,3)\qbezier(32.5,3)(37.25,3)(39.2,0.6)
\put(32.5,3){\vector(1,0){1}}
\put(36,0){\line(1,0){3}}
\put(37.5,0){\vector(1,0){1}}
\put(41,0){\line(1,0){3}}
\put(42.5,0){\vector(1,0){1}}
\qbezier(40.8,0.6)(42,3)(45,3)\qbezier(45,3)(48,3)(49.2,0.6)
\put(45,3){\vector(1,0){1}}
\qbezier(40.8,0.6)(43.5,5)(47.5,5)\qbezier(47.5,5)(51.5,5)(54.2,0.6)
\put(47.5,5){\vector(1,0){1}}
\end{picture}
\end{center}
\caption{A marking of the floor diagram of Figure~\ref{fig:markingAfterStep3}}
\end{figure}
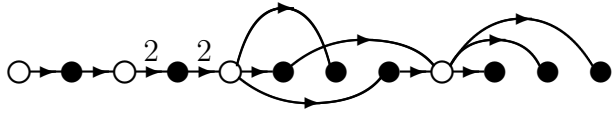
\end{definition}

We want to count marked floor diagrams up to equivalence. Two markings
$\widetilde{\D}_1$, $\widetilde{\D}_2$ of a floor diagram $\D$ are \emph{equivalent} if there exists an 
automorphism of weighted graphs which preserves the vertices of $\D$ and maps $\widetilde{\D}_1$ to
$\widetilde{\D}_2$. The \emph{number of markings} $\nu(\D)$ is the number of
marked floor diagrams $\widetilde{\D}$ up to equivalence.

\begin{example}
The floor diagram $\D$ of Example~\ref{ex:floordiagram} has $\nu(\D) =
3 + 4 = 7$ markings (up to equivalence):
 In step 3 the extra $1$-valent vertex connected to the third white vertex
from the left can be inserted in three ways between the third and fourth
white vertex (up to equivalence) and in four ways right of the fourth
white vertex (again up to equivalence).
\end{example}

With these two statistics, we define a purely combinatorial notion of
refined Severi degrees for $S = \PP^2$, $S = \F_m$, and
$S=\PP(1,1,m)$. The combinatorial invariants agree with the refined
Severi degree $N^{\Delta, \delta}(y)$ of
Section~\ref{sec:tropicalRefinedCounting}
(Theorem~\ref{thm:combRefinedEqualRefined}). They also agree
conjecturally with the
refined invariants of G\"ottsche and Shende if $S$ is smooth and the
line bundle is sufficiently ample (cf.\
Conjecture~\ref{conj:smoothToric} and Theorem~\ref{thm:refinedInvEqualsRefinedSeveri}).

See Remark~\ref{rmk:hTransverse} for a discussion how
to generalize to a much larger family of toric surfaces corresponding
to ``$h$-transverse'' $\Delta$.
Denote by $\FD(\Delta, \delta)$ the set of $\Delta$-floor diagrams $\D$ with cogenus $\delta$.

\begin{definition}
\label{def:combrefined}
Fix $\delta \ge 0$ and let $\Delta$ be as above.
We define the \emph{combinatorial
  refined Severi degree} $N_\comb^{\Delta, \delta}(y)$ to be the Laurent
polynomial in $y$ given by
\begin{equation}
\label{eqn:combrefined}
N_\comb^{\Delta, \delta}(y) = \sum_{\D \in \FD(\Delta, \delta)} \mult(\D; y)
\cdot \nu(\D).
\end{equation} 
\end{definition}

\begin{theorem}
\label{thm:combRefinedEqualRefined}
For $\Delta$ as in Definition~\ref{def:combrefined} and $\delta \ge
0$, the combinatorial refined Severi degree and the refined Severi
degree agree:
\[
N_\comb^{\Delta, \delta}(y) = N^{\Delta, \delta}(y).
\]
\end{theorem}

\begin{proof}
Let $\Pi \subset \RR^2$ be a vertically stretched 
(Definition~\ref{def:verticallyStretched}) configuration of $|\Delta
\cap \ZZ^2| - 1 - \delta$ tropical points. 
In
\cite[Proposition~5.9]{BM2}, Brugall\'e and Mikhalkin construct an
explicit bijection between the set of parametrized tropical curves of
degree~$\Delta$ with $\delta$ nodes passing through $\Pi$ and the set
of marked $\Delta$-floor diagrams of cogenus $\delta$. This bijection
is $y$-weight preserving.
\end{proof}

In the sequel, we will usually write $N$ instead of $N_\comb$ even
while referring to the combinatorial defined refined Severi degree if
no confusion can occur.

\begin{remark}
\label{rmk:hTransverse}
We expect the results in this section to also hold for toric surfaces
from ``$h$-transverse'' polygons $\Delta$: Brugall\'e and
Mikhalkin~\cite{BM2} construct marked floor diagrams for this class of
polygons. One can define a notion of \emph{combinatorial refined Severi
degrees for any toric surface from an ``$h$-transverse polygon''}:
simply replace the multiplicity of a ``$\Delta$-floor diagram'' $\D$ in
\cite[Equation (Severi1)]{AB10} by the $y$-weight
\[
\mult(\D, y) = \prod_{\text{edges }e} \left(
[\wt(e)]_y \right)^2.
\]
Theorem~\ref{thm:combRefinedEqualRefined} can then be extended to the
more general setting. We omit the details here to avoid too many technicalities.
\end{remark}

\subsection{Templates}

The following gadget was introduced by Fomin and Mikhalkin~\cite{FM}.

\begin{definition}
\label{def:template}
A \emph{template} $\Gamma$ is a directed graph (possibly with multiple
edges) on vertices
$\{0, \dots, l\}$, where $l\ge 1$, with edge
weights $\wt(e) \in \ZZ_{>0}$, satisfying:
\begin{enumerate}
\item If $i \to j$ is an edge, then $i <j$.
\item Every edge $i \stackrel{e}{\to} i+1$ has weight $\wt(e) \ge
  2$. (No ``short edges.'')
\item For each vertex $j$, $1 \le j \le l-1$, there is an edge
  ``covering'' it, i.e., there exists an edge $i \to k$ with $i <j <k$.
\end{enumerate}
\end{definition}

Every template $\Gamma$ comes with some numerical
data associated with it. Its \emph{length} $\ell(\Gamma)$ is the number of
vertices minus $1$. Its \emph{cogenus} $\delta(\Gamma)$
is
\begin{equation}
\label{eqn:templatecogenus}
\delta(\Gamma) = \sum_{i \stackrel{e}{\to} j} \bigg[(j-i) \wt(e) -1
\bigg].
\end{equation}
We define its \emph{$y$-multiplicity} $\mult(\Gamma, y)$ to be
\[
\mult(\Gamma, y) = \prod_{\text{edges }e} \left(
[\wt(e)]_y \right)^2.
\]
See Figure~\ref{fig:templates} for examples.

For $1 \le j \le \ell(\Gamma)$, let $\varkappa_j = \varkappa_j(\Gamma)$
denote the sum of the weights of edges $i \to k$ with $i
< j \le k$. So $\varkappa_j(\Gamma)$ equals the total weight of the
edges of $\Gamma$ from a vertex left of $j$ to a vertex right of or
equal to $j$. Define
\[
k_{\min}(\Gamma) = \max_{1 \le j \le l}(\varkappa_j -j +1).
\]
This makes $k_{\min}(\Gamma)$ the smallest positive integer $k$ such
that $\Gamma$ can appear in a floor diagram on $\{1, 2, \dots \}$ with
left-most vertex $k$.
Lastly, set
 \begin{displaymath}
   \eps_0(\Gamma)  = \left\{
     \begin{array}{ll}
       1 &  \text{if all edges starting at } 0 \text{ have weight }1,\\
       0 &  \text{otherwise,}
     \end{array}
   \right.
\end{displaymath}
and
 \begin{displaymath}
   \eps_1(\Gamma)  = \left\{
     \begin{array}{ll}
       1 &  \text{if all edges arriving at } l \text{ have weight }1,\\
       0 &  \text{otherwise,}
     \end{array}
   \right.
\end{displaymath}
Figure~\ref{fig:templates} (taken from Fomin-Mikhalkin~\cite{FM})
shows all templates $\Gamma$ with $\delta(\Gamma) \le 2$.

Notice that, for each $\delta$, there are only a finite number of
templates with cogenus $\delta$.
At $y = 1$, we recover Fomin and Mikhalkin's template
multiplicity $\prod_e \wt(e)^2$. It is clear that
$\mult(\Gamma, y)$ is a Laurent polynomial with positive integral coefficients.

\begin{figure}[htbp]
\begin{center}
\begin{tabular}{c|c|c|c|c|c|c|c}
$\Gamma$ &
$\delta(\Gamma)$ & $\ell(\Gamma)$ 
& $\mult(\Gamma; y)$ %multiplicity
& $\eps_0(\Gamma)$ & $\eps_1(\Gamma)$ & $\varkappa(\Gamma)$ & $k_{\min}(\Gamma)$
\\
\hline \hline
&&&&&&&\\[-.1in]
\begin{picture}(95,8)(-10,-4)\setlength{\unitlength}{2.5pt}\thicklines
\oo
\put(5,2){\makebox(0,0){$\scriptstyle 2$}}
\Eeee
\end{picture}
& 1 & 1 & $y^{-1}+2+y$ & 0 & 0 & (2) & 2
\\[.15in]
\begin{picture}(95,8)(-10,-4)\setlength{\unitlength}{2.5pt}\thicklines
\ooo
\qbezier(0.8,0.6)(10,5)(19.2,0.6)
\end{picture}
& 1 & 2 & 1 & 1 & 1 & (1,1) & 1
\\
\hline \hline
&&&&&&&\\[-.1in]
\begin{picture}(95,8)(-10,-4)\setlength{\unitlength}{2.5pt}\thicklines
\oo
\put(5,2){\makebox(0,0){$\scriptstyle 3$}}
\Eeee
\end{picture}
& 2 & 1 & $y^{-2}+2y^{-1}+3+2y+y^2$ & 0 & 0 & (3) & 3
\\[.15in]
\begin{picture}(95,8)(-10,-4)\setlength{\unitlength}{2.5pt}\thicklines
\oo
\put(5,3.5){\makebox(0,0){$\scriptstyle 2$}}
\put(5,-3.5){\makebox(0,0){$\scriptstyle 2$}}
\qbezier(0.8,0.6)(5,2)(9.2,0.6)
\qbezier(0.8,-0.6)(5,-2)(9.2,-0.6)
\end{picture}
& 2 & 1 & $y^{-2}+4y^{-1}+6+4y+y^2$ & 0 & 0 & (4) & 4
\\[.15in]
\begin{picture}(95,8)(-10,-4)\setlength{\unitlength}{2.5pt}\thicklines
\ooo
\qbezier(0.8,0.6)(10,4)(19.2,0.6)
\qbezier(0.8,-0.6)(10,-4)(19.2,-0.6)
\end{picture}
& 2 & 2 & 1 & 1 & 1 & (2,2) & 2
\\[.15in]
\begin{picture}(95,8)(-10,-4)\setlength{\unitlength}{2.5pt}\thicklines
\ooo
\qbezier(0.8,0.6)(10,4)(19.2,0.6)
\put(5,-2){\makebox(0,0){$\scriptstyle 2$}}
\Eeee
\end{picture}
& 2 & 2 & $y^{-1}+2+y$ & 0 & 1 & (3,1) & 3
\\[.15in]
\begin{picture}(95,8)(-10,-4)\setlength{\unitlength}{2.5pt}\thicklines
\ooo
\qbezier(0.8,0.6)(10,4)(19.2,0.6)
\put(15,-2){\makebox(0,0){$\scriptstyle 2$}}
\eEee
\end{picture}
& 2 & 2 & $y^{-1}+2+y$ & 1 & 0 & (1,3) & 2
\\[.15in]
\begin{picture}(95,8)(-10,-4)\setlength{\unitlength}{2.5pt}\thicklines
\oooo
\qbezier(0.8,0.6)(15,6)(29.2,0.6)
\end{picture}
& 2 & 3 & 1 & 1 & 1 & (1,1,1) & 1
\\[.15in]
\begin{picture}(95,8)(-10,-4)\setlength{\unitlength}{2.5pt}\thicklines
\oooo
\qbezier(0.8,0.6)(10,5)(19.2,0.6)
\qbezier(10.8,0.6)(20,5)(29.2,0.6)
\end{picture}
& 2 & 3 & 1 & 1 & 1 & (1,2,1) & 1
\\[-.05in]
\end{tabular}
\end{center}
\caption{The templates with $\delta(\Gamma) \le 2$.}
\label{fig:templates}
\end{figure}

\subsection{Decomposition into Templates}
\label{sec:decomposition}

A labeled floor diagram $\D$ with $d$ vertices decomposes into an ordered
collection $(\Gamma_1, \dots, \Gamma_s)$ of templates as
follows. If $S = \PP^2$ or $\PP(1,1,m)$, then we set as before $c = 0$.
We treat $S = \PP^2$ as the special case of $\PP(1,1,m)$ for
$m = 1$.

First, add an additional vertex $0$ ($< 1$) to $\D$ and connect it to
every vertex $j$ of $\D$ by $s_j$ many new edges of weight $1$ from
$0$ to $j$ for each $1 \le j \le d$. (For $S = \PP^2$ and $S=\PP(1,1,m)$, there is nothing to
do, as $s_j=0$ for all $j$.) Second, add an additional vertex $d+1$
($> d$), together with $m+s_j-\dive(j)$ new edges of weight $1$ from
$j$, for each $1 \le j \le d$. The \emph{divergence sequence} of the
resulting diagram $\D'$ is $\aa := (c, m, \dots, m) \in \ZZ_{\ge 0}^{d+1}$,
after we remove the (superfluous) last entry.
Now remove all \emph{short edges}
from $\D'$, that is, all edges of weight~$1$ between consecutive
vertices. The result is an ordered collection of templates $(\Gamma_1,
\dots, \Gamma_s)$, listed 
left to right. We also keep track of the initial vertices $k_1, \dots,
k_s$ of these templates.

Conversely, given the collection of templates $\GGamma = (\Gamma_1,
\dots, \Gamma_s)$,
 the initial vertices $k_1, \dots, k_s$, and the
divergence sequence $(c, m, \dots, m) \in \ZZ_{\ge 0}^{d+1}$, this
process is easily reversed. To recover $\D'$, we first place the
templates at their starting points $k_i$ in the interval $[0, \dots,
M]$, and add in all short edges we removed from $\D'$. More precisely, 
we need to add $(a_0 + \cdots + a_{j-1} - \varkappa_{k_j- j}(\Gamma_i))$
short edges between $j-1$ and $j$, where $\Gamma_i$ is the template
containing $j$. The sequence $s$ records the number $s_j$ of edges
between vertices $0$ and $j$.
Finally, we remove the first and last vertices and their incident
edges to obtain $\D$.

\begin{example}
\label{ex:templatedecomposition}
An example for $S = \PP^2$ of the decomposition of a labeled floor diagram into
templates is illustrated below. Here, $k_1 = 2$ and $k_2 = 4$ and all
$s_j = 0$.

\begin{picture}(50,121)(-165,-103)\setlength{\unitlength}{3.0pt}\thicklines
\ooooo\Eeee\eEee\eeEe \eeOe \eeeE 
\put(15,2){\makebox(0,0){$2$}}
\put(7,0){\vector(1,0){1}}
\put(17,0){\vector(1,0){1}}
\put(27,0){\vector(1,0){1}}
\put(27.5,1.75){\vector(2,-1){1}}
\put(27.5,-1.75){\vector(2,1){1}}
\put(35,2){\makebox(0,0){$3$}}
\put(37,0){\vector(1,0){1}}
\put(20,-8){\LARGE$\updownarrow$}
\put(-13,-1){$\D = $}
\end{picture}
\begin{picture}(50,121)(-111,-55)\setlength{\unitlength}{3pt}\thicklines
\oooooo\Eeee\eEee\eeEe \eeOe \eeeE \eeeBB \eeeeE \eeeeO
\put(15,2){\makebox(0,0){$2$}}
\put(7,0){\vector(1,0){1}}
\put(17,0){\vector(1,0){1}}
\put(27,0){\vector(1,0){1}}
\put(27.5,1.75){\vector(2,-1){1}}
\put(27.5,-1.75){\vector(2,1){1}}
\put(35,2){\makebox(0,0){$3$}}
\put(37,0){\vector(1,0){1}}
\put(47,0){\vector(1,0){1}}
\put(40,5){\vector(1,0){1}}
\put(47.5,1.75){\vector(2,-1){1}}
\put(47.5,-1.75){\vector(2,1){1}}
\put(20,-8){\LARGE$\updownarrow$}
\put(-13,-1){$\D' = $}
\end{picture}
\begin{picture}(50,121)(-57,-10)\setlength{\unitlength}{3pt}\thicklines
\oooooo \eEee  \eeeE \eeeBB 
\put(-3,-1){\big (}
\put(1,-1){\big )}
\put(15,2){\makebox(0,0){$2$}}
\put(17,0){\vector(1,0){1}}
\put(35,2){\makebox(0,0){$3$}}
\put(37,0){\vector(1,0){1}}
\put(40,5){\vector(1,0){1}}
\put(-22,-1){$(\Gamma_1, \Gamma_2) = $}
\end{picture}
\end{example}

We record, for each ordered template collection $\GGamma =
(\Gamma_1, \dots, \Gamma_s)$, 
all valid ``positions'' $k_i$ that
can occur in the template decomposition of a $\Delta$-floor diagram by
the lattice points in a polytope. There are two cases.
If $S = \PP^2$, we set
\begin{equation}
\label{eqn:P2polytope}
\begin{split}
A_{\GGamma}(d) =  \{ &\kk \in \RR^s \, : \, k_i \ge
k_{\min}(\Gamma_i),\\
& 
k_i + \ell(\Gamma_i)  \le k_{i+1} \, \, \,  ( 1 \le i < s), \, k_s +
\ell(\Gamma_s) \le d + \eps_1(\Gamma_s) \}.
\end{split}
\end{equation}
If $S = \PP(1,1,m)$ or $S = \F_m$, we set
\begin{equation}
\label{eqn:FmAndP111polytope}
\begin{split}
A_{\GGamma}(d) =  \{ &\kk \in \RR^s \, : 
\, k_1 \ge 1 -\eps_0(\Gamma_1),  \\
& 
k_i + \ell(\Gamma_i)  \le k_{i+1} \, \, \,  ( 1 \le i < s), \, k_s +
\ell(\Gamma_s) \le d + \eps_1(\Gamma_s) \}.
\end{split}
\end{equation}
The first inequality in (\ref{eqn:P2polytope}) says that, due to the divergence condition, templates cannot appear too early in a
floor diagram. The first inequality in (\ref{eqn:FmAndP111polytope})
says that the first starting position can be $0$ precisely when all
outgoing edges of the first vertex of $\Gamma_1$ have weight $1$.
The second resp.\ third inequality in (\ref{eqn:P2polytope}) and (\ref{eqn:FmAndP111polytope}) say that templates
cannot overlap resp.\ cannot hang over at the end of the floor
diagram.

We note that the lattice points in $A_{\GGamma}(d)$ in
(\ref{eqn:FmAndP111polytope}) record all template positions if
the divergence at the first vertex is at least $2\delta$: 
the quantity $\varkappa_j(\Gamma)$ is maximal, for a given
$\delta(\Gamma) = \delta$,
when $\Gamma$ is the 
template with two vertices and $\delta$ edges between them, each with
weight $2$, and $j=1$. The condition $\text{div}(1) \ge 2 \delta$ implies
then that every collection of lattice points in the polytope can be the
  sequence of positions of templates, and vice versa.
We always make the assumption $\text{div}(1) \ge 2 \delta$ in
Section~\ref{sec:refinednodepolys}, where we prove polynomiality
of the refined Severi degrees for parameters in this regime (cf.
Theorem~\ref{thm:refinedSeveridegreepoly}).

\subsection{Multiplicity, Cogenus, and Markings.}
The refined multiplicity, cogenus, and markings of a floor diagram behave well
under template decomposition, as in the unrefined case. If a floor
diagram $\D$ has template decomposition $\GGamma$, then by definition
\[
\mult(\D;y) = \prod_{i=1}^s \mult(\Gamma_i; y).
\]
Furthermore, the decomposition of Section~\ref{sec:decomposition} is
cogenus preserving, i.e., $\delta(\D) = \sum_{i=1}^s \delta(\Gamma_i)$  (see~\cite[Section~3.3.2]{AB10}).
The number of markings of floor diagrams is expressible in terms of
the number of ``markings of the templates'': In Step~4 in
Definition~\ref{def:marking}, instead of linearly ordering
$\tilde{\D}$, we can order each template individually. To make this
precise, associate to each template $\Gamma$ a polynomial
$P_\Gamma(c,m; k)$ in $k$ which depends also on the parameters $c$ and
$m$ of the polygon $\Delta$ (cf.\ Figure~\ref{fig:2polygons}). Specifically, let
$\Gamma_{(c,m,k)}$ denote the graph obtained from $\Gamma$ by first
adding
\[
c + (k+j-1)m - \varkappa_j(\Gamma)
\]
short edges, making the divergence of all vertices $m$,  and then
subdividing each  of the resulting graphs by introducing a new
vertex for each edge. Let $P_\Gamma(c,m;k)$ be the number of linear
extensions, up to equivalence, of the vertex poset 
of the graph
$\Gamma_{(c,m,k)}$ extending the vertex order of $\Gamma$. Then
\[
\nu(\D) = \prod_{i=1}^s P_{\Gamma_i}(c,m;k_i).
\]

We can summarize the previous discussion in the following proposition.

\begin{proposition}
The combinatorial refined Severi degree for
\begin{enumerate}
\item $S = \PP^2$, any $\delta \ge 1$ and $d \ge 1$, or 
\item $S = \F_m$ resp.\ $S =\PP(1,1,m)$, $\delta \ge 1$ and $m,c,d \ge
  1$ with $c+m \ge 2 \delta$
\end{enumerate}
is given by
\begin{equation}
\label{eqn:refinedseveridiscreteintegral}
N_\comb^{\Delta,\delta}(y) = \sum_{\GGamma: \, \sum_i \delta(\Gamma_i) = \delta}
\left[ \Big( \prod_{i=1}^s \mult(\Gamma_i, y) \Big)
\sum_{\kk \in A_{\GGamma}(d) \cap \ZZ^s} \Big( \prod_{i = 1}^s 
  \ P_{\Gamma_i}(c,m; k_i) \Big) \right],
\end{equation}
the first sum running over all templates collections $\GGamma =
(\Gamma_1, \dots, \Gamma_s)$ with $\sum_{i = 1}^s \delta(\Gamma_i) =
\delta$.

If $S = \F_m$ one can relax condition $m \ge 2 \delta$ to $m+c \ge 2 \delta$.
\end{proposition}

For $y = 1$ and $S = \PP^2$, expression (\ref{eqn:refinedseveridiscreteintegral})
specializes to \cite[(5.13)]{FM}. For $y = 1$ and $S = \F_m$ resp.\
$S=\PP(1,1,m)$, expression (\ref{eqn:refinedseveridiscreteintegral})
specializes to \cite[Proposition~3.3]{AB10}.

\section{Polynomiality Proofs}
\label{sec:refinednodepolys}

We now use floor diagrams and templates to prove
Theorem~\ref{thm:refinedSeveridegreepoly}
and
Proposition~\ref{prop:leadingCoefficients}. 
The argument
for the former
 is based on the combinatorial formula
(\ref{eqn:refinedseveridiscreteintegral}). 
Our technique is a
$q$-analog extension of Fomin and Mikhalkin's
method \cite[Section~5]{FM} for the $\PP^2$ and Ardila and Block's
\cite{AB10} for $\F_m$ and $\PP(1,1,m)$.
The method provides an algorithm
to compute refined node polynomials for any $\delta$; see
Remark~\ref{rem:smallrefinedpolys} for a list for $\delta \le 2$ for
$\PP^2$.

\newtheorem*{thm:refinedSeveridegreepoly}{Theorem \ref{thm:refinedSeveridegreepoly}}
\begin{thm:refinedSeveridegreepoly}
For fixed $\delta \ge 1$:
\begin{enumerate}
\item ($\PP^2$) There is a polynomial $N_\delta(d; y) \in \QQ[y^{\pm
    1}][d]$  of degree $2\delta$ in $d$ such that, for $d \ge \delta$,
\[
N_\delta(d;y) = N^{d, \delta}(y).
\]

\item (Hirzebruch surface) There is a polynomial $N_\delta(c, d,m;y)
 \in \QQ[y^{\pm 1}][c,d,m]$ of degree $\delta$ in $c,m$ and $2\delta$
 in $d$ such that, for 
$c + m \ge 2 \delta$ and $d \ge \delta$
\[
N_\delta(c, d,m;y) = N^{(\F_m, cF+dH), \delta}(y).
\]

\item ($\PP(1, 1, m)$) There is a polynomial $N_\delta(d,m;y) \in
  \QQ[y^{\pm 1}][d,m]$ of degree $2\delta$ in $d$ and $\delta$ in $m$
  such that, for  $d \ge \delta$   and $m \ge 2 \delta$,
\[
N_\delta(d,m;y) = N^{\PP(1,1,m),dH), \delta}(y).
\]
\end{enumerate}
\end{thm:refinedSeveridegreepoly}

\begin{proof}[Proof of Theorem~\ref{thm:refinedSeveridegreepoly}]
The proof for $S = \PP^2$ is essentially the proof of
\cite[Theorem~5.1]{FM}, suped-up with refined multiplicities. For $S =
\Sigma_m$ and $S=\PP(1,1,m)$ our argument is a special (but now
refined) case of the proof of \cite[Theorem~1.2]{AB10}.
We first want to show that, for $S = \PP(1,1,m)$ resp.\ $S=\Sigma_m$ and
fixed $\delta$, the
expression in (\ref{eqn:refinedseveridiscreteintegral}) is polynomial 
in $d$ and $m$ resp.\ $c$, $d$ and $m$ for appropriately large values
of $c$, $d$ and $m$. As before, for $S = \PP(1,1,m)$, we set $c =
0$. The case $S = \PP^2$ we treat at the end.

The number of template collections $\GGamma = (\Gamma_1, \dots,
\Gamma_s)$ with fixed cogenus $\sum_{i=1}^s \delta(\Gamma_i) = \delta$
is finite. The factor $\prod_{i=1}^s \mult(\Gamma_i, y)$ is simply a
Laurent polynomial in $y$; it thus remains to show that the second sum in
(\ref{eqn:refinedseveridiscreteintegral}) is polynomial for
appropriately large $d$ and $m$, and also $c$ if $S = \Sigma_m$.

Since for each template $\Gamma_i$ and any $j$, we have
$\varkappa_j(\Gamma_i) \le 2\delta \le c + 
m$,  each individual template $\Gamma_i$ can ``float freely''
between $k_i = \eps_0(\Gamma_i)$ and $d-\ell(\Gamma_i) +
\eps_1(\Gamma_i)$. Thus, as $c+m \ge 2\delta$, the valid starting
positions $k_i$ of all templates are given by the inequalities of
$A_\GGamma(d)$ as in (\ref{eqn:FmAndP111polytope}).

If $d \ge \delta$ then $A_\GGamma(d)$ is non-empty as
\[
\eps_0(\Gamma_1)+ \ell(\Gamma_1) + \cdots + \ell(\Gamma_s) - \eps_1(\Gamma_s)
\le \delta.
\]
In fact, the combinatorial type of $A_\GGamma(d)$ does not change if
$d \ge \delta$: it is always combinatorially equivalent to a simplex. The
inequalities are given by $A \cdot \kk \le b(d)$ for a unimodular matrix
$A$ and a vector $b(d)$ of linear forms in $d$.

For each lattice point $(k_1, \dots, k_s)$ in $A_\GGamma(d)$, the
number of markings $P_{\Gamma_i}(c,m; k_i)$ of $\Gamma_i$ at position
$k_i$ is polynomial in $k_i, c$ and $m$ provided that $c+m \ge 2\delta$
\cite[Lemma~5.8]{FM}. Thus, for $\kk \in A_\GGamma(d) \cap \ZZ^s$,
\begin{equation}
\label{eqn:markings_of_template_collection}
\prod_{i=1}^s P_{\Gamma_i}(c,m;k_i)
\end{equation}
is a polynomial in $c,m,k_1, \dots k_s$. From the explicit description
of $P_{\Gamma_i}(c,m;k_i)$, it is not hard to see that the degree of
$P_{\Gamma_i}(c,m;k_i)$ in $k_i$, in $c$, and in $m$ is bounded above
by the number of edges of $\Gamma_i$ and thus by
$\delta(\Gamma_i)$. Hence, if $c+m \ge 2 \delta$, the number
(\ref{eqn:markings_of_template_collection}) of  markings of the
template collection $\GGamma$ is of degree at most $\delta$ in $c$ and
 in $m$, and at most $\delta(\Gamma_i)$ in $k_i$.

By \cite[Lemma~4.9]{AB10}, the second sum in
(\ref{eqn:refinedseveridiscreteintegral}) is a piecewise polynomial in
$c$, $d$, and $m$: the second sum is a ``discrete integral'' of a
polynomial over the facet-unimodular polytope $A_\GGamma(d)$.
But for $c+m \ge 2\delta$ and $d \ge \delta$, the combinatorial type
of $A_\GGamma(d)$ does not change, $A_\GGamma(d)$ is a dilation of a
unit simplex by the (non-negative) number
\[
d -\left( \eps_0(\Gamma_1)+ \ell(\Gamma_1) + \cdots + \ell(\Gamma_s) -
\eps_1(\Gamma_s) \right).
\]
Hence the second sum in (\ref{eqn:refinedseveridiscreteintegral}) is
polynomial in $c$, $d$, and $m$ for $c+m \ge 2 \delta$ and $d \ge
\delta$. This polynomial is of degree at most $\delta$ in $c$ and in
$m$. As the number $s$ of templates in the template collection
$\GGamma$ is bounded by $\delta$, we (discretely) integrate over at
most $\delta$ dimensions in  (\ref{eqn:refinedseveridiscreteintegral})
and thus the degree of the refined Severi degree in $d$ is at most
$\delta(\Gamma_1) + \cdots + \delta(\Gamma_s) + s \le 2 \delta$.

To conclude the result for $S=\PP(1,1,m)$ set $c = 0$.

For $S = \PP^2$, the proof is identical to the proof of
\cite[Theorem~1.3]{FB}; we only need to replace $\mult(\Gamma_i)
(=\mult(\Gamma_i; 1))$ by $\mult(\Gamma_i; y)$ throughout (e.g.,
(\ref{eqn:refinedseveridiscreteintegral}) at $y=1$ becomes
\cite[(3.1)]{FB}). The proof to further reduce the threshold value for
polynomiality in $d$ of $N^{d, \delta}(y)$ from $2\delta$ to $\delta$
(as in the theorem) relies on another statistic ``$s(\Gamma)$''
\cite[p.~13]{FB}. The two key Lemmas~4.2 and~4.3 of \cite{FB} only
involve the markings of a floor diagram and are thus verbatim in the
refined case. The degree bound follows as in the case of $\PP(1,1,m)$
(with $m=1$). For $S=\PP^2$, the degree bound $2\delta$ in $d$ is
tight: a template collection $\GGamma$ with each $\Gamma_j$ a
template with $\delta(\Gamma_j) = 1$ for $1 \le j \le \delta$
contributes to $N_\delta(d; y)$ in degree $2 \delta$ in $d$.
\end{proof}

\begin{remark}
\label{rem:smallrefinedpolys}
Expression (\ref{eqn:refinedseveridiscreteintegral}) gives, in
principle, an algorithm to compute refined node polynomials. The
algorithm of \cite[Section~3]{FB}, based on 
the algorithm of Fomin and Mikhalkin~\cite[Section~5]{FM}, easily
adapts to the refined case. Below we show $N_\delta(d; y)$, for $S =
\PP^2$ and for $\delta \le 2$ as computed by this method. 
(Note that \thmref{thm:refinedInvEqualsRefinedSeveri} determines (by another method) the $N_\delta(d; y)$ for $\delta\le 10$.)
\[
\small
\begin{split}
N_1(d; y) = & \, \tfrac{1}{2}y d^2-\tfrac{3}{2}y d+y +2 d^2-3
d+1+\tfrac{1}{2} y^{-1} d^2- \tfrac{3}{2}y^{-1} d+y^{-1},\\
N_2(d; y) = & \, \tfrac{1}{8} y^2 d^4-\tfrac{3}{4} y^2
d^3+\tfrac{11}{8} y^2 d^2-\tfrac{3}{4} y^2 d+y d^4-\tfrac{9}{2} y
d^3+2 y d^2+\tfrac{21}{2} y d-9 y+\tfrac{9}{4}  d^4-\tfrac{15}{2}  d^3
\\
&-\tfrac{3}{4}  d^2+21 d-15+ y^{-1}
d^4-\tfrac{9}{2} y^{-1} d^3+2 y^{-1} d^2+\tfrac{21}{2} y^{-1} d-9 y^{-1}+\tfrac{1}{8}
y^{-2} d^4 \\
& -\tfrac{3}{4} y^{-2} d^3 +\tfrac{11}{8} y^{-2} d^2-\tfrac{3}{4} y^{-2} d.\\
\end{split}
\]
\end{remark}

\begin{proof}[Proof of Proposition~\ref{prop:leadingCoefficients}]
To a floor diagram $\D$, we associated the new statistic
\[
i(\D) = \sum_{e \in \D} \wt(e) \left( \len(e) - 1\right).
\]
It captures how much of the cogenus is contributed by edges of length
greater than $1$.
By degree considerations, one can see that a floor diagram $\D$
contributes only to the coefficients $p^\delta_{d,i}$ of $N^{d, \delta}(y)$
with $i(\D) \le i$. 
To compute $p^\delta_{d,0}$, it thus suffices to consider
only the floor diagrams of degree $d$ with cogenus $\delta$ and $i(\D)
= 0$. Furthermore, each such floor diagram has $\mult(\D; y)$ a degree
$\delta$ polynomial in $y$ and $y^{-1}$ with leading coefficient
$1$. It, thus, suffices to show that the number of \emph{marked} floor diagrams with
$d(\D) = d$, $\delta(\D) = \delta$, and $i(\D) = 0$ equals $\binom{\binom{d-1}{2}}{\delta}$. 

Each such marked floor diagram arises as follows: let $\D_0$ be the
unique floor diagram of degree $d$ and cogenus $0$ ($\D_0$ has one
edge of weight $1$ between vertex $1$ and $2$, two edges of weight $1$
between vertex $2$ and $3$, and so on). The genus of $\D_0$ is
$\binom{d-1}{2}$. Subdivide each edge of $\D_0$ by introducing a new vertex and
  order all vertices linearly, extending the linear order of the $d$
  original vertices. Call a cycle in $\D_0$ of length $2$
  \emph{contractible} if the two midpoints corresponding to the two
  edges are adjacent in the linear order.
Choose $\delta$ contractible cycles and ``contract'' each
  cycle by  identifying the two edges and the two midpoints to obtain the
  graph $\D_1$. To each
 edge in $\D_1$ assign a weight equal to the number of edges of
  $\D_0$ that were identified in obtaining $\D_1$. Note that $\D_1$
  comes with a linear order on its vertices and is, thus, a marked
  floor diagram with $\delta(\D_1) = \delta$ and $i(\D_1) = 0$, and
  all such marked floor diagrams arise this way.
\end{proof}

\section{Refined Relative Severi Degrees}
\label{sec:RefRelSevDegs}

In this section, we generalize refined Severi degrees to include tropical
tangency conditions. We then show that, in the case of the surfaces $S
= \Sigma_m$ and $S = \PP(1,1,m)$, the resulting invariants satisfy the
recursion of G\"ottsche and Shende (Definition~\ref{refCHrec})
and thus both invariants agree. Our definitions are a refinement of
\cite{GM052} for $S = \PP^2$ and \cite{IKS09} for arbitrary toric surfaces. 

Throughout this section, $\alpha = (\alpha_1, \alpha_2,\dots)$ and
$\beta = (\beta_1, \beta_2, \dots)$ denote infinite sequences of
non-negative integers with only finitely many non-zero entries. Recall
the notations $|\alpha| = \sum_{i\ge 1} \alpha_i$ and $I \alpha =
\sum_{i \ge 1} i \alpha_i$.

\begin{figure}[htbp]
 \begin{tabular}{cc}
\includegraphics[scale=0.30]{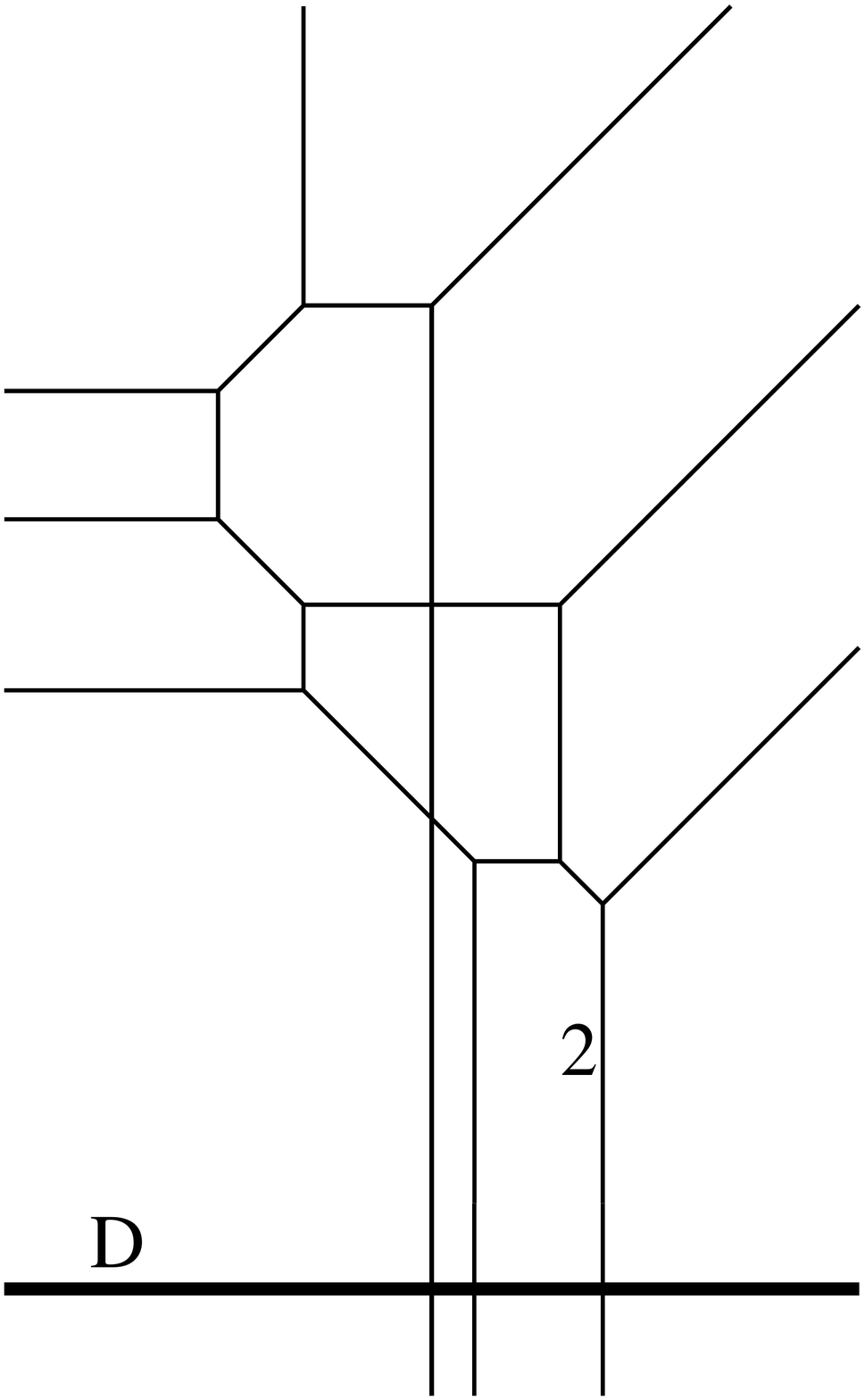} \quad \quad \quad
&
\quad
\raisebox{25pt}{
\includegraphics[scale=0.20]{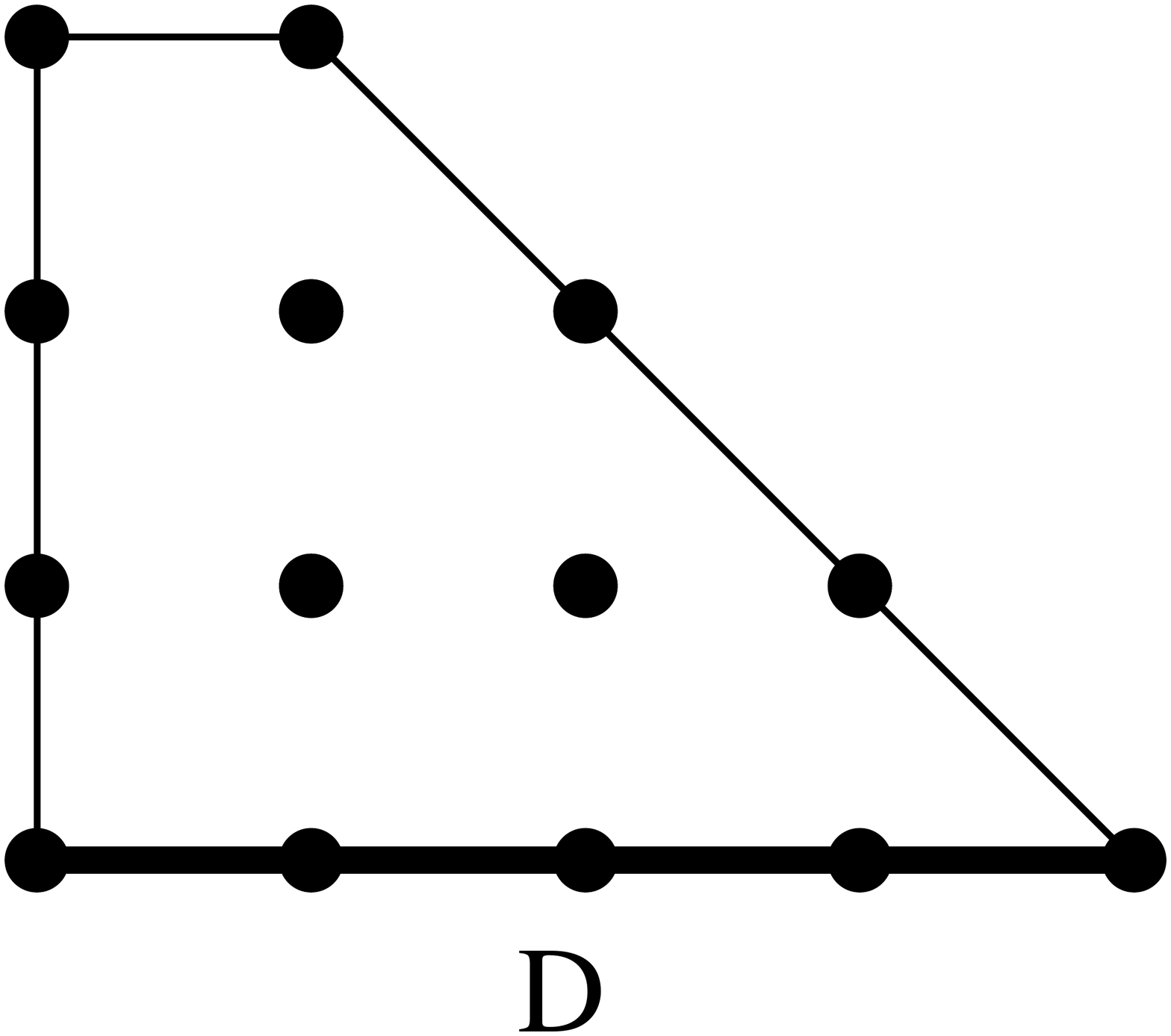}} \\
 \end{tabular}
\caption{A tropical curve $C$ with tangency $\beta = (2,1)$ to a tropical divisor
  $D$. The tropical divisor corresponds to the bottom horizontal edge
  (also denoted $D$) of
 the polygon $\Delta$ of the Hirzebruch surface $\Sigma_1$.}
\label{fig:tropicalBoundaryDivisor}
\end{figure}

\begin{definition}
Let $\Delta$ be a lattice polygon and $h: C \to \RR^2$ a
parametrized tropical curve of degree $\Delta$ (see
Definition~\ref{def:parametrizedTropicalCurve}). Again we simply write
$C$ instead of $h: C \to \RR^2$. Let $D$ be an edge of
$\Delta$ and $l(D)$ its lattice length.
\begin{enumerate}
\item The \emph{tropical boundary
  divisor} of $D$ is a (classical) line in $\RR^2$ parallel to $D$
and sufficiently
far in the
direction dual to $D$ (so that all intersections with $C$ are
orthogonal). Abusing notation, we denote the tropical boundary
divisor by $D$ also.

\item We say that the tropical curve $C$ is \emph{tangent to $D$
    of order $\beta$} if the partition of edge weights of the
  unbounded edges of $C$ orthogonal to $D$ is $\beta$. (I.e., if
  there are $\beta_1$ such edges of weight $1$, $\beta_2$ of weight
  $2$, and so on.)
\end{enumerate}
\end{definition}

See Figure~\ref{fig:tropicalBoundaryDivisor} for an example.
Throughout, we fix the following data:
\begin{enumerate}
\item a tropical boundary divisor $D$ (corresponding to an edge $D$ of
  $\Delta$),
\item two sequences $\alpha$ and $\beta$ with $I \alpha + I \beta$
  equal the lattice length $l(D)$ of the edge $D$, and 
\item a tropically generic
point configuration $\Pi$ of $n= |\Delta \cap \ZZ^2| -1 -\delta -I(\alpha+\beta) +
|\alpha| + |\beta|$ points with precisely $|\alpha|$ points on
$D$.
\end{enumerate}
The number of points $n$ is chosen so that the resulting curve count
is non-zero and finite (unless $\delta$ is very large). 

As in the classical case, we distinguish two types of tangencies:
tangencies to $D$ at a \emph{fixed} point (i.e., a point in $\Pi$), the number of such of
multiplicity $i$ we denote by $\alpha_i$. The other type of tangency
to $D$ is at unspecified or \emph{free} points; we denote the number
of such of multiplicity $i$ by $\beta_i$. The following is a
refinement of \cite[Definition~4.1]{GM052}.

\begin{definition}
\begin{enumerate}
\item A tropical curve $C$ passing through $\Pi$ is
  \emph{$(\alpha, \beta)$-tangent to $D$} if precisely $\alpha_i+\beta_i$ 
unbounded edges of $C$ are orthogonal to and intersect $D$ and have multiplicity
$i$ and, further,
$\alpha_i$ of the edges pass through $\Pi \cap D$.

\item The subdivision of $\Delta$ dual to the tropical curve $C$
  is the \emph{combinatorial type} of $C$.

\item The \emph{refined relative multiplicity}
  $\mult_{\alpha,\beta}(C;y)$ of a tropical curve
  $(\alpha,\beta)$-tangent to $D$ is
\begin{equation}
\label{eqn:refinedRelativeCurveMult}
\mult_{\alpha,\beta}(C;y) = \frac{1}{\prod_{i \ge 1}
  ([i]_y)^{\alpha_i}} \cdot \mult(C;y).
\end{equation}

\item The \emph{refined relative Severi degree $N^{\Delta,\delta}(\alpha,
  \beta)(y)$} is the number of $\delta$-nodal
tropical curves $C$ of degree $\Delta$ passing through $\Pi$ that
are $(\alpha, \beta)$-tangent to $D$, counted with multiplicity
$\mult_{\alpha,\beta}(C;y)$.

\item The \emph{refined relative irreducible Severi degree $N_0^{\Delta,\delta}(\alpha,
  \beta)(y)$} is the number of irreducible 
tropical curves $C$ of degree $\Delta$ with $\delta$  nodes passing through $\Pi$ that
are $(\alpha, \beta)$-tangent to $D$, counted with multiplicity
$\mult_{\alpha,\beta}(C;y)$.
\end{enumerate}
\end{definition}

Both $N^{\Delta,\delta}(\alpha, \beta)(y)$ and $N_0^{\Delta,\delta}(\alpha,
\beta)(y)$ in general depend on the tropical boundary divisor $D$. To
simplify notation, we surpress this dependence. We discuss the cases $S =
\PP(1,1,m)$ and $S = \Sigma_m$ in detail later and will always choose $D$ to be a
horizontal line $y = const$, for $const << 0$, cf.\
Figure~\ref{fig:tropicalBoundaryDivisor}.

\begin{theorem}
\label{thm:refinedRelativeSeveriIndep}
The refined relative Severi degree
$N^{\Delta,\delta}(\alpha,\beta)(y)$ and the refined relative irreducible 
Severi degree $N_0^{\Delta,\delta}(\alpha, \beta)(y)$ are independent of the
tropical point configuration if it is generic.
\end{theorem}

\begin{proof}
The invariance of the refined relative irreducible Severi degree
$N_0^{\Delta,\delta}(\alpha, \beta)(y)$ follows from a rather
straightforward modification of Itenberg
and Mikhalkin's proof~\cite[Theorem~1]{IM12} of the independence of
the refined irreducible Severi degree $N_0^{\Delta,\delta}(y)$. We are
brief here, in order to not repeat a lengthy argument. The
modification with respect to \cite{IM12} is to allow combinatorial
types of tropical curves with arbitrary tangency conditions to one
tropical divisor. The result then follows from the observation that
Itenberg and Mikhalkin's argument also holds in this setting.

Let $\Pi = \{p_1, p_2,\dots, p_n\}$ be a configuration of $n = |\Delta \cap \ZZ^2| -
1 - I(\alpha + \beta) + |\alpha| + |\beta|$ tropical points. It
suffices to show the invariance if we smoothly perturb the points $\Pi$ to
$\Pi(t) = \{p_1, \dots, p_{k-1}, p_k(t), p_{k+1}, \dots, p_{n} \}$,
for some $1 \le k \le n$, and all $t \in [-\eps, \epsilon]$ for
some $\eps >0$ and $\Pi(0) = \Pi$ such that $\Pi(t)$ is tropically
generic for $t \neq 0$.

Fix an irreducible tropical curve $h: C \to \RR^2$ with genus
$g$ with $\Pi \subset h(C)$ that is $(\alpha,
\beta)$-tangent to $D$. Let $S^{\pm}(t)$ be the set of tropical curves $h^{\pm}(t):
C \to \RR^2$ that are $(\alpha, \beta)$-tangent to $D$ with
$\Pi(\eps) \subset h^{\pm}(\eps)(C)$ for $\pm \eps
\in [0, t]$ that deform to $h$, i.e., with $h^{\pm}(0) = h$.

In the following, we conclude that
\begin{equation}
\label{eqn:4vertexResolution}
\sum_{C^+ \in S^{+}(t)} \mult_{\alpha, \beta}(C^+; y) =
 \sum_{C^- \in S^{-}(t)} \mult_{\alpha, \beta}(C^-; y).
\end{equation}
If $h: C \to \RR^2$ has no $4$-valent vertex, then for $t > 0$
small enough, $|S^{+}(t)| = |S^{-}(t)| = 1$ and the combinatorial
types of $C^+$ and $C^-$ agree and
(\ref{eqn:4vertexResolution}) follows. Otherwise, every $4$-valent
vertex of $h$ is perturbed as shown in \cite[Figure~6]{IM12} because for
$t > 0$ small enough the combinatorial type of $h(t)$ changes only
locally around the $4$-valent vertex. (The detailed argument is in the
proof of \cite[Lemma~3.3]{IM12}; their proof also holds if we fix
multiplicity of unbounded edges of $h$ (to incorporate the $\beta$-tangency
conditions) as well as point conditions on these edges very far away
(to incorporate the $\alpha$-tangency conditions).) The refined
relative multiplicity $\mult_{\alpha, \beta}$ on both sides of
(\ref{eqn:4vertexResolution}) equals $\frac{1}{\prod_{i \ge 1}
  ([i]_y)^{\alpha_i}}$ times the refined (non-relative) multiplicity
$\mult(C; y)$. Thus, to show that the difference between both sides of
(\ref{eqn:4vertexResolution}) is zero it suffices to show that
\begin{equation}
\label{eqn:4vertexResolutionNonRelative}
\sum_{C^+ \in S^{+}(t)} \mult(C^+; y) =
 \sum_{C^- \in S^{-}(t)} \mult(C^-; y).
\end{equation}
As the tropical curves on both sides of this equation differ only
locally around the $4$-valent vertices of $h$, the argument to prove
(\ref{eqn:4vertexResolutionNonRelative}) is identical to the proof of
\cite[Lemma~3.3]{IM12}. 
The invariance of the refined relative Severi degree
$N^{\Delta,\delta}(\alpha, \beta)(y)$ then follows from (\ref{eqn:irred_red_rel}).

\end{proof}

\begin{remark}
The refined relative Severi degree $N^{\Delta,\delta}(\alpha,\beta)(y)$ is a
symmetric (under $y \leftrightarrow y^{-1}$) Laurent polynomial in
$y^{1/2}$ with
non-negative integer coefficients (not in $y$ in general). As before, one may ask what the
coefficients of $N^{\Delta,\delta}(\alpha,\beta)(y)$ count. 
\end{remark}

\begin{theorem}\label{thm:refined=tropical}
For all polygons $\Delta$, with $X(\Delta) = \PP(1,1,m)$ or $X(\Delta)
= \Sigma_m$, the refined relative tropical Severi degrees satisfy
\eqref{refrec} with $L =L(\Delta)$. Therefore,
the refined relative Severi degrees defined via the
recursion~\ref{refCHrec} and the refined relative tropical Severi
degrees agree.
\end{theorem}

\begin{figure}[htbp]
 \begin{tabular}{cc}
\includegraphics[scale=0.30]{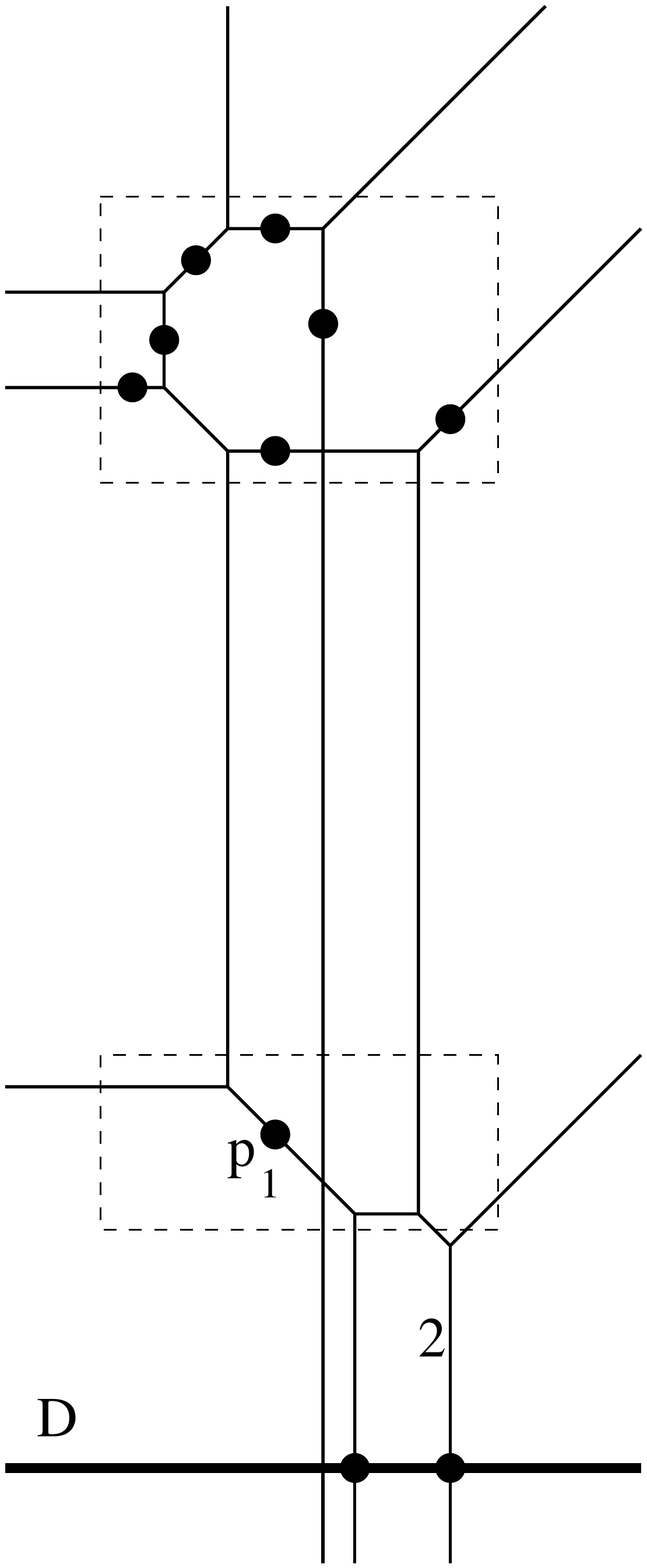} \quad \quad
&
\raisebox{40pt}{
\includegraphics[scale=0.17]{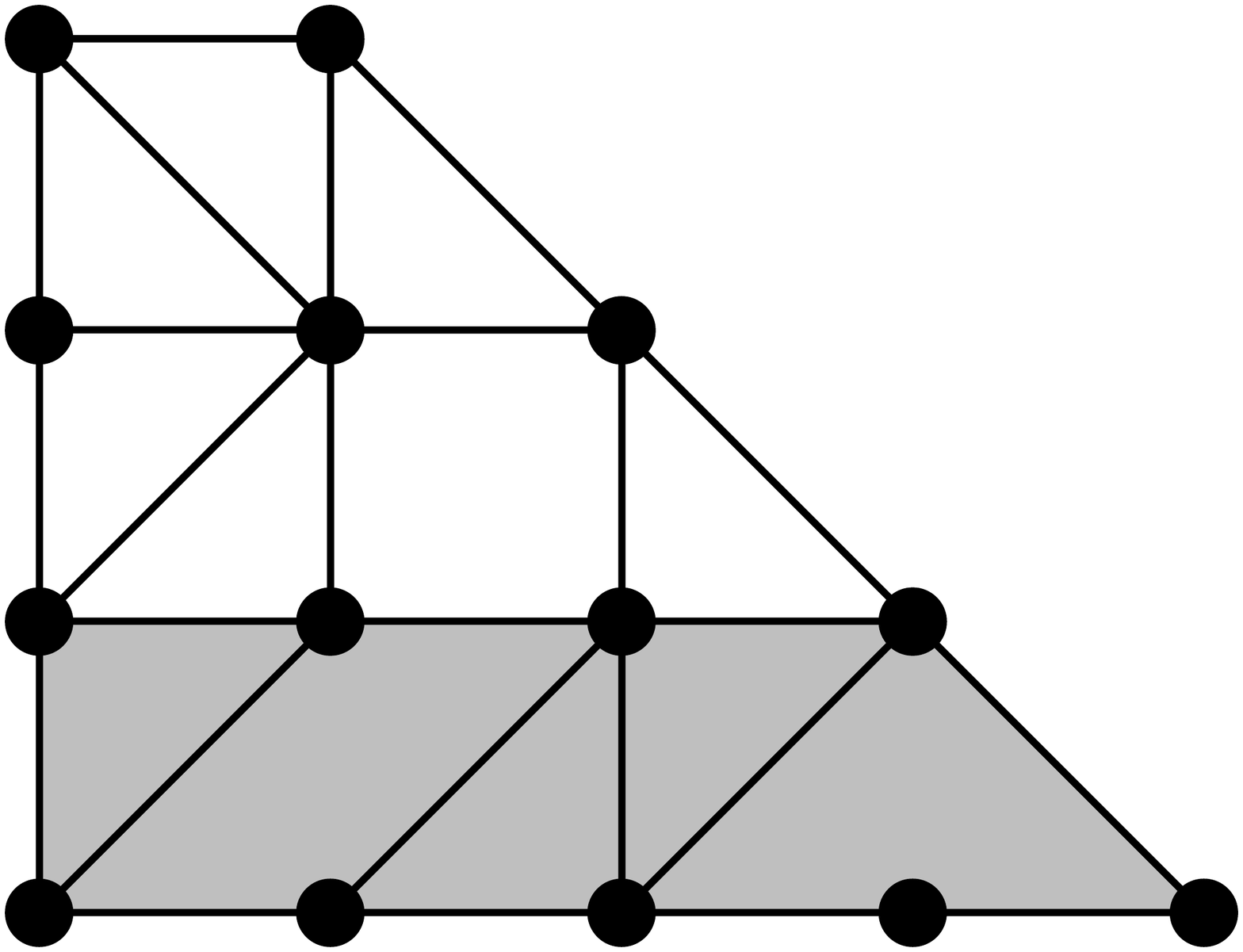}} \\
 \end{tabular}
\caption{A $((1,1),(1))$-tangent curve $C$ to a tropical divisor $D$. All
  point conditions are in a small vertical strip, $p_1$ is far from
  all other points. The curve $C$ decomposes into the ``upper''
  part $C'$ and the ``lower'' part containing $p_1$. $C'$ is
  $((0), (3))$-tangent to $D$. $C$ and $C'$ have $\delta =
  2$ and $\delta' = 1$ as can be see from the dual subdivision on the
  right. The shaded part of the polygon is the difference between the
  degree $\Delta$ of $C$ and $\Delta'$ of $C'$.}
\label{fig:TropicalCurveDecomposition}
\end{figure}

\begin{proof}
Our proof follows closely and extends the argument of Gathmann and
Markwig's proof of their \cite[Theorem 4.3]{GM052}, where they proved
this result in the non-refined case (i.e., $y =1$) for the surface $S =
\PP^2$. Instead of points in a horizontal strip, we consider points in
a vertical strip. The Gathmann-Markwig proof rests on an observation of Mikhalkin
\cite[Lemma~4.20]{Mi05} that holds for any toric surface. We use it in
generalizing their argument.

Fix a small $\eps > 0$ and a large real number $M$. Consider a
tropical generic point configuration $\Pi = \{ p_1, p_2, \dots, p_n\}$
such that
\begin{enumerate}
\item the $x$-coordinates of all $p_i$ (including those on the divisor
  $D$) are within the interval $(-\eps, \eps)$,
\item the point $p_1$ is not on the divisor $D$ but its $y$-coordinate
  is less than $-M$,
\item all points $p_i \neq p_1$ not lying on $D$ have $y$-coordinate
  in the interval $(-\eps, \eps)$.
\end{enumerate}

Let $C$ be a tropical curve of degree $\Delta$ with $\delta$
nodes. Then $C$ is of the following form:
\begin{enumerate}
\item all vertices of $C$ have $x$-coordinate in $(-\eps, \eps)$,
\item there are constants $a$ and $b$, depending only on $\Delta$, with
  $-N < a < b < -\eps$ so that $C$ has no vertices in the strip
  $\RR \times [a, b]$; all edges in this strip are vertical.
\end{enumerate}
See Figure~\ref{fig:TropicalCurveDecomposition} for an illustration.
This follows directly from the verbatim argument
in~\cite{GM052}; note that their argument rests on
\cite[Lemma~4.20]{Mi05} which applies to arbitrary $\Delta$, so in
particular to $S = \PP(1,1,m)$ and $S = \Sigma_m$.

There are two cases:
\begin{enumerate}
\item Case 1: $p_1$ lies on a vertical edge with weight $k \ge 1$. Then all edges of
  $C$ with $y$-coordinates $\le -\eps$ are vertical 
  by the
  Gathmann-Markwig argument. We can move $p_1$ down onto the divisor
  $D$ and obtain a tropical curve with one more ``fixed'' tangency
  conditions. The weight of $C$ is $[k]_y$ times the weight of
  this new curve. The total contribution of tropical curves through
  $\Pi$ with $p_1$ on a vertical edge is thus
\[
\sum_{k:\beta_k>0} [k]_y \cdot  N^{\Delta,\delta}(\alpha+e_k,\beta-e_k)(y).
\]
\item Case 2: $p_1$ does not lie on a vertical half-ray. Then $C$
  can be broken into two pieces: let $C'$ be the curve with bounded
  edges in the vicinity of the points $p_2, \dots, p_n$ that do not
  lie on $D$. The other piece, containing $p_1$, consists of the
  bounded edges of $C$ in the vicinity of $p_1$, one unbounded
  edge in direction $(-1,0)$ and $(1,m)$, respectively, and some
  vertical edges. See Figure~\ref{fig:TropicalCurveDecomposition}
  for an illustration of this decomposition. By construction, the degree $\Delta'$ of $C'$ is the
  lattice polygon obtained from $\Delta$ by removing a horizontal
  strip of width one at the bottom of $\Delta$. 

Next, we determine in how many ways $C'$ can be extended to a
tropical curve of degree $\Delta$ that is $(\alpha, \beta)$-tangent to
the divisor $D$ and passes through $\Pi$. We know that $C'$ is
$(\alpha', \beta')$ tangent to $D$, for some $\alpha' \le \alpha$ and
$\beta' \ge \beta$. There are $\binom{\alpha}{\alpha'}$ ways to choose
which vertical edges of $C$ through a point in $\Pi \cap D$
belong to $C'$. Similarly, there are are $\binom{\beta'}{\beta}$
ways to choose which vertical edges of $C'$ intersecting $D$ but
not containing a point in $\Pi$ belong to $C$ (for more details
see~\cite{GM052}).

To show that the tangency conditions $\alpha'$ and $\beta'$ satisfy $I
\alpha' + I \beta' = H(L-H)$,
recall that degree $\Delta'$ of
$C'$ is the polygon obtained from $\Delta$ by removing from $\Delta$ the bottom
strip of lattice width $1$. Furthermore, $I\alpha' + I \beta'$ equals the lattice
length of the bottom edge of $\Delta'$.
We argue for each surface separately. 
\begin{enumerate}[(a)]
\item $S = \PP^2$: Here $H$ is the class of a line and $L$ is the
  class of a degree $d$ curve. Thus, we have $H(L-H) = d -1$, the
  length of the bottom edge of $\Delta'$.
\item $S = \Sigma_m$: In this case, we defined $H$ as the class of a
  section with $H^2 = m$. Then $H(L-H) = c + (d-1)m$. Recall that $\Delta = \conv((0,0), (0,d), (c, d), (c
  + dm,0))$. The bottom edge of $\Delta'$ has lattice length $c +
  (d-1)m = H(L-H)$. 
\item $S = \PP(1,1,m)$: Here $H$ is the class of a line, and we
  have $H(L-H) = (d-1)m$. As $\Delta =
  \conv((0,0), (0,d), (dm, 0))$, $H(L-H)$ is precisely the lattice
  length of $\Delta'$.
\end{enumerate}

Next, we relate the the $y$-multiplicities of $C$ and
$C'$. We have
\[
\mult(C; y) = \prod\limits_{i\ge 1} \left( [i]_y \right)^{\alpha_i
  - \alpha'_i + \beta'_i - \beta_i} \mult(C'; y)
\]
and, therefore,
\[
\mult_{\alpha, \beta}(C; y) =
 \frac{1}{\prod_{i \ge 1} ([i]_y)^{\alpha_i}} \mult(C; y) =
\prod\limits_{i \ge 1} ([i]_y)^{\beta'_i- \beta_i} \mult_{\alpha',
  \beta'}(C'; y).
\]

Now, we show that the cogenus $\delta'$ of $C'$ satisfies
\[
\delta - \delta' = I \alpha' + I \beta' - |\beta' - \beta|.
\]
By definition, $\delta - \delta'$ counts the number of parallelograms
in the horizontal bottom strip of width $1$ in the dual subdivision
$\Delta_C$. This number equals the number of unbounded edges of
$C'$ that intersect $D$ and are
unbounded in $C$. But this number is precisely the length of the
upper edge of the width $1$ strip minus the number of edges of $C'$,
that
become bounded as edges in $C$, and thus equals $I \alpha' + I
\beta' - |\beta' - \beta|$.

The recursive formula now follows: by the balancing condition, the
$(\alpha', \beta')$-tangent curve $C'$ can be completed to a
$(\alpha, \beta)$-tangent curve $C$ with $p_1 \in C
\backslash C'$ in a unique way, once we choose which vertical
edges of $C$ through a point in $\Pi \cap D$ belong to $C'$
and which vertical edges of $C'$ intersecting $D$ but not
containing a point in $C$ belong to $C$
(giving $\binom{\alpha}{\alpha'} \cdot \binom{\beta'}{\beta}$ choices).
\end{enumerate}
Checking the initial conditions is trivial.
\end{proof}

\begin{remark}
Note that in the case of rational ruled surfaces $\Sigma_m$, the  above proof works also if we allow $m$ to be negative. Then $\Sigma_m=\Sigma_{-m}$, but with the role
of $E$ and $H$ exchanged (this corresponds to exchanging the top and the bottom edge of $\Delta$). Expressed on  $\Sigma_m$, the proof thus also shows  the recursion \eqref{refrec}, with the same initial conditions, but everywhere  with $H$ replaced by $E$ and $\alpha$, $\beta$ specifying contacts along $E$ instead along $H$.
\end{remark}

\subsection{Refined Relative Node Polynomials for Plane Curves}

We now extend the floor diagram technique to refined relative Severi
degree for $S = \PP^2$. Then we show a 
polynomiality result (Theorem~\ref{thm:refinedRelativePolynomialP2})
about refined relative Severi degrees of $\PP^2$, refining the result
of~\cite[Theorem~1.1]{FB10}. We expect a similar, more technical
argument to work also for $S=\Sigma_m$ and $S=\PP(1,1,m)$, but restrict ourselves to
$\PP^2$ for simplicity. The following definitions are a quite straightforward
refinement of \cite[Section~3.2]{FM}.

As we are only concerned with $S = \PP^2$, we denote by $\FD(d,
\delta)$ the set of $\Delta$-floor diagrams with $\Delta =
\conv((0,0), (0,d), (d,0))$ and cogenus $\delta$, for any $d \ge
1$. Also, let $\FD_{\text{conn}}(d, \delta)$ denote the
collection of connected such floor diagrams.
Let $\alpha$ and $\beta$ be two
sequences of non-negative integers with only finitely many non-zero entries.

To each floor diagram $\D \in
\FD(d, \delta)$, there is a statistic
$\nu_{\alpha, \beta}(\D)$, counting the number of ``$(\alpha,
\beta)$-markings'' of $\D$ as in the non-relative case. The precise definition is given in
\cite[Definition~2.3]{FB10}, a reformulation of
\cite[Definition~3.13]{FM}.
Intuitively, $\nu_{\alpha, \beta}(\D)$ counts the number of tropical
curves of degree $\Delta$ that
\begin{enumerate}
\item are $(\alpha,\beta)$-tangent to $D = \{ y = const\}$, where
  $const << 0$ (so $D$ is a very far down horizontal line),
\item ``correspond'' to the floor diagram $\D$ (in the sense of
  \cite[Theorem~3.17]{FM}), and
\item that pass through a vertically stretched point configuration.
\end{enumerate}

The refined relative Severi degree of $\PP^2$ can be expressed purely
combinatorially in terms of the $y$-weighted floor diagrams of
Section~\ref{sec:FD}:
To simplify the formula,
define (cf., \cite[(3.6)]{FM}) for the unrefined setting)
\begin{displaymath}
\label{eqn:refined_relative_multiplicity_FD}
\mult_\beta(\D, y) = \prod_{i \ge 1} ([i]_y)^{\beta_i} \cdot \mult(\D, y).
\end{displaymath}

\begin{proposition}
\begin{enumerate}
\item
\label{itm:RefRelSevP2}
For any $d \ge 1$ and $\delta \ge 1$, the refined relative Severi
degree of $\PP^2$ is given by
\[
N^{d, \delta}(\alpha, \beta)(y) = \sum_{\D \in \FD(d, \delta)} 
\mult_\beta(\D, y) \cdot \nu_{\alpha, \beta}(\D).
\]
\item
\label{itm:RefRelIredSevP2}
For any $d \ge 1$ and $\delta \ge 1$, the refined irreducible relative
Severi degree of $\PP^2$ is computed by
\[
N^{d, \delta}_0(\alpha, \beta)(y) =  \sum_{\D \in \FD_{\text{conn}}(d, \delta)} 
\mult_\beta(\D, y) \cdot \nu_{\alpha, \beta}(\D).
\]
\end{enumerate}
\end{proposition}

\begin{proof}
We first prove Part~\ref{itm:RefRelIredSevP2}. We may assume, by
Theorem~\ref{thm:refinedRelativeSeveriIndep}, that the tropical point
configuration is vertically stretched. By \cite[Theorem~3.17]{FM},
there is a bijection $f$ between irreducible tropical curves of degree $d$ and
cogenus $\delta$ that are $(\alpha,
\beta)$-tangent to $D$ and $(\alpha, \beta)$-marked floor diagrams
$\D$ with $\D \in \FD_{\text{conn}}(d, \delta)$, where we used that
these tropical curves have genus $g = \binom{d-1}{2} - \delta$.
By \cite[Theorem~3.17]{FM} (see also \cite[Theorem 3.7]{FM} for the
non-relative case but with more details), the map $f$ preserves the unrefined
multiplicity ($y = 1$) for any such tropical curve $C$ with
corresponding floor diagram $\D$: 
\[
\mult_{\alpha, \beta}(C,1) = \mult_\beta(\D, 1).
\]
By definition of the refined multiplicities of tropical curves and
floor diagrams (Definitions~\ref{def:refinedMultiplicityTropicalCurve}
and~\ref{def:refined_multiplicity_FD} and Equation
(\ref{eqn:refined_relative_multiplicity_FD})), 
the bijection $f$ preserves also the refined multiplicities:
\[
\mult_{\alpha, \beta}(C,y) = \mult_\beta(\D, y),
\]
as $\frac{1}{\prod_{i \ge 1} ([i]_y)^{\alpha_i}}\mult(C, y) =
\prod_{i \ge 1} ([i]_y)^{\beta_i} \cdot \mult(\D, y)$, 
and Part~\ref{itm:RefRelIredSevP2} follows.

Part~\ref{itm:RefRelSevP2} follows from Part~\ref{itm:RefRelIredSevP2}
by a straightforward refined extension
of the inclusion-exclusion procedure of~\cite[Section~1]{FM} that was
used to conclude~\cite[Corollary~1.9]{FM} (the non-relative unrefined count of
reducible curves via floor diagrams) from~\cite[Theorem~1.6]{FM} (the
non-relative unrefined count of irreducible curves via floor diagrams).
\end{proof}

\begin{theorem}
\label{thm:refinedRelativePolynomialP2}
For any $\delta \ge 1$, there is a polynomial
\[
N_\delta(\alpha; \beta; y) \in \QQ[y^{\pm 1}][\alpha_1, \dots, \alpha_\delta;
\beta_1, \dots, \beta_\delta]
\]
in $\alpha_i$ and $\beta_i$ with coefficients in $\QQ[y^{\pm 1}]$
such that, for any $\alpha$ and $\beta$ with $|\beta| \ge \delta$, we
have
\[
 N^{d, \delta}(\alpha;\beta)(y) = \prod_{i \ge 1} ([i]_y)^{\beta_i}
\frac{(|\beta| - \delta)!}{\beta_1! \beta_2! \cdots}\cdot  N_\delta(\alpha; \beta; y).
\]
The coefficients of the polynomial $N_\delta(\alpha; \beta; y)$ are
preserved under the transformation $y \leftrightarrow y^{-1}$.
\end{theorem}

We call $N_\delta(\alpha; \beta;y)$ the \emph{refined relative node
  polynomial} of $\PP^2$.

\begin{proof}
The proof is identical to the non-refined argument in~\cite[Theorem~1.1]{FB10}
but with the unrefined multiplicity of a floor diagram replaced by the
refined multiplicity of the present paper.
\end{proof}

\bibliographystyle{amsplain}
\bibliography{References_Florian}

\end{document}